\label{key}
\documentclass[10]{article}
\usepackage{amsfonts,amssymb,amsmath,amsthm,cite}
\usepackage{graphicx}

\usepackage[applemac]{inputenc}
\usepackage{amsmath,amssymb,amsthm, hyperref, euscript}
\usepackage[matrix,arrow,curve]{xy}
\usepackage{graphicx}
\usepackage{tabularx}
\usepackage{float}
\usepackage{hyperref}
\usepackage{tikz}
\usepackage{slashed}
\usepackage{mathrsfs}

\usetikzlibrary{matrix}

\usepackage[T1]{fontenc}
\usepackage{amsfonts,cite}
\usepackage{graphicx}

\usepackage[applemac]{inputenc}

\usepackage[sc]{mathpazo}
\DeclareFontFamily{T1}{pzc}{}
\DeclareFontShape{T1}{pzc}{m}{it}{1.8 <-> pzcmi8t}{}
\DeclareMathAlphabet{\mathpzc}{T1}{pzc}{m}{it}

\theoremstyle{plain}
\newtheorem{prop}{Proposition}[section]

\newtheorem{lem}[prop]{Lemma}
\newtheorem{cor}[prop]{Corollary}
\newtheorem{thm}[prop]{Theorem}

\theoremstyle{definition}
\newtheorem{defn}[prop]{Definition}
\newtheorem{empt}[prop]{}
\newtheorem{exm}[prop]{Example}
\newtheorem{rem}[prop]{Remark}

\newtheorem{fact}[prop]{Fact}

\newcommand{\vertiii}[1]{{\left\vert\kern-0.25ex\left\vert\kern-0.25ex\left\vert #1
		\right\vert\kern-0.25ex\right\vert\kern-0.25ex\right\vert}}
\newcommand{\Ga}{\Gamma}                     
\newcommand{\Coo}{C^\infty}                  

\usepackage[sc]{mathpazo}
\newbox\ncintdbox \newbox\ncinttbox 
\setbox0=\hbox{$-$} \setbox2=\hbox{$\displaystyle\int$}
\setbox\ncintdbox=\hbox{\rlap{\hbox
		to \wd2{\hskip-.125em \box2\relax\hfil}}\box0\kern.1em}
\setbox0=\hbox{$\vcenter{\hrule width 4pt}$}
\setbox2=\hbox{$\textstyle\int$} \setbox\ncinttbox=\hbox{\rlap{\hbox
		to \wd2{\hskip-.175em \box2\relax\hfil}}\box0\kern.1em}

\newcommand{\A}{\mathcal{A}}                 

\newcommand{\C}{\mathbb{C}}                  
\newcommand{\half}{\tfrac{1}{2}}             
\newcommand{\hookto}{\hookrightarrow}        
\newcommand{\La}{\Lambda}                    
\newcommand{\la}{\lambda}                    

\newcommand{\mop}{{\mathchoice{\mathbin{\star_{_\theta}}}
		{\mathbin{\star_{_\theta}}}           
		{{\star_\theta}}{{\star_\theta}}}}    
\newcommand{\N}{\mathbb{N}}                  
\newcommand{\om}{\omega}                     
\newcommand{\eps}{\varepsilon}                    
\newcommand{\Q}{\mathbb{Q}}                  
\newcommand{\R}{\mathbb{R}}                  
\newcommand{\set}[1]{\{\,#1\,\}}             

\renewcommand{\SS}{\mathcal{S}}              
\DeclareMathOperator{\supp}{\mathfrak{supp}}            
\newcommand{\T}{\mathbb{T}}                  
\renewcommand{\th}{\theta}                   
\newcommand{\Z}{\mathbb{Z}}                  
\renewcommand{\.}{\cdot}                     

\newcommand{\sS}{\mathcal{S}}       

\newcommand{\Om}{\Omega}       

\newcommand{\al}{\alpha}          
\newcommand{\ga}{\gamma}          
\newcommand{\Th}{\Theta}          
\renewcommand{\th}{\theta}        

\def\<#1|#2>{\langle#1\stroke#2\rangle} 
\def\?#1|#2?{\{#1\stroke#2\}}        

\def\<#1,#2>{\langle#1,#2\rangle}            
\def\ee_#1{e_{{\scriptscriptstyle#1}}}       
\def\wick:#1:{\mathopen:#1\mathclose:}       

\newbox\ncintdbox \newbox\ncinttbox 
\setbox0=\hbox{$-$}
\setbox2=\hbox{$\displaystyle\int$}
\setbox\ncintdbox=\hbox{\rlap{\hbox
		to \wd2{\box2\relax\hfil}}\box0\kern.1em}
\setbox0=\hbox{$\vcenter{\hrule width 4pt}$}
\setbox2=\hbox{$\textstyle\int$}
\setbox\ncinttbox=\hbox{\rlap{\hbox
		to \wd2{\hskip-.05em\box2\relax\hfil}}\box0\kern.1em}

\newcommand{\stroke}{\mathbin|}   

\newcommand{\Hom}{\mathrm{Hom}}       
\newcommand{\Aut}{\mathrm{Aut}}       

\title{Inverse Limits of Noncommutative Covering Projections}
\begin{document}
\maketitle  \setlength{\parindent}{0pt}
\begin{center}
\author{
{\textbf{Petr R. Ivankov*}\\
e-mail: * monster.ivankov@gmail.com \\
}
}
\end{center}

\vspace{1 in}

\begin{abstract}
\noindent

\paragraph{}

The Gelfand - Na\u{i}mark theorem supplies the one to one correspondence between commutative $C^*$-algebras and locally compact Hausdorff spaces. So any noncommutative $C^*$-algebra can be regarded as a generalization of a topological space.  Generalizations of several topological invariants can be defined by algebraical methods. This article contains a pure algebraical construction of inverse limits in the category of (noncommutative) covering projections. It is proven that Moyal planes are  inverse limits of covering projections of noncommutative tori.
\end{abstract}
\tableofcontents

\section{Motivation. Preliminaries}
\paragraph{}
Some notions of the geometry have noncommutative generalizations based on the Gelfand-Na\u{i}mark theorem.
\begin{thm}\label{gelfand-naimark}\cite{arveson:c_alg_invt} (Gelfand-Na\u{i}mark). 
	Let $A$ be a commutative $C^*$-algebra and let $\mathcal{X}$ be the spectrum of A. There is the natural $*$-isomorphism $\gamma:A \to C_0(\mathcal{X})$.
\end{thm}

\paragraph{}From the theorem it follows that a (noncommutative) $C^*$-algebra can be regarded as a generalized (noncommutative)  locally compact Hausdorff topological space. Following theorem gives a pure algebraic description of covering projections of compact spaces.
\begin{thm}\label{pavlov_troisky_thm}\cite{pavlov_troisky:cov}
	Suppose $\mathcal X$ and $\mathcal Y$ are compact Hausdorff connected spaces and $p :\mathcal  Y \to \mathcal X$
	is a continuous surjection. If $C(\mathcal Y )$ is a projective finitely generated Hilbert module over
	$C(\mathcal X)$ with respect to the action
	\begin{equation*}
	(f\xi)(y) = f(y)\xi(p(y)), ~ f \in  C(\mathcal Y ), ~ \xi \in  C(\mathcal X),
	\end{equation*}
	then $p$ is a finite-fold (or equivalently a finitely listed) covering.
\end{thm} 
\paragraph{} This article contains a pure algebraic construction of inverse limits in the category of covering projections. 

\paragraph{} This article assumes an elementary knowledge of following subjects:
\begin{enumerate}
	\item Set theory \cite{halmos:set};
	\item Category theory  \cite{spanier:at};
	\item General topology \cite{munkres:topology};
	\item Algebraic topology \cite{spanier:at};

	\item $C^*$-algebras, $C^*$-Hilbert modules and $K$-theory  \cite{arveson:c_alg_invt, blackadar:ko,pedersen:ca_aut,murphy,takesaki:oa_ii}.
	
\end{enumerate}
\paragraph{}
The terms ``set'', ``family'' and ``collection'' are synonyms.

\paragraph{}Following table contains a list of special symbols.
\newline
\begin{tabular}{|c|c|}
	\hline
	Symbol & Meaning\\
	\hline
 & \\
	$A^G$  & Algebra of $G$-invariants, i.e. $A^G = \{a\in A \ | \ ga=a, \forall g\in G\}$\\
	$\mathrm{Aut}(A)$ & Group of *-automorphisms of $C^*$-algebra $A$\\
	$B(H)$ & Algebra of bounded operators on Hilbert space $H$\\
	$\mathbb{C}$ (resp. $\mathbb{R}$)  & Field of complex (resp. real) numbers \\
	$C(\mathcal{X})$ & $C^*$-algebra of continuous complex valued \\
	& functions on a compact space $\mathcal{X}$\\
	$C_0(\mathcal{X})$ & $C^*$-algebra of continuous complex valued functions\\ 
	&  on a locally compact topological  space $\mathcal{X}$ equal to $0$ at infinity\\
	 $C_b(\mathcal{X})$ & $C^*$ - algebra of bounded  continuous complex valued \\
	  & functions on a topological space $\mathcal{X}$ \\
	$C_c(\mathcal{X})$ & Algebra of continuous complex valued functions on a \\
	 &  topological  space $\mathcal{X}$ with compact support\\
	$G(\widetilde{\mathcal{X}} | \mathcal{X})$ & Group of covering transformations of a covering projection  $\widetilde{\mathcal{X}} \to \mathcal{X}$ \cite{spanier:at}  \\
$\mathfrak{ess~supp}\left( f\right)$ & Essential support of a Borel-measured function $f$ \cite{elliot:an} \\
	$H$ & Hilbert space \\
	$\varinjlim$ & Inductive limit \\
	$\varprojlim$ & Projective limit \\
	$M(A)$  & The multiplier algebra of $C^*$-algebra $A$\\
	$\mathbb{N}$  & Set of positive integer numbers\\
	$\mathbb{N}^0$  & Set of nonnegative integer numbers\\
	

	
	$\supp \varphi$ & Support of a continuous map $\varphi: \mathcal X \to \mathbb{C}$,\\
	 & which is the closure of the set $\left\{x \in \mathcal X~|~ \varphi\left( x\right)\neq 0 \right\}$\\
	$U(A) \subset A $ & Group of unitary operators of algebra $A$\\
	
	$\mathbb{Z}$ & Ring of integers \\
	
	$\mathbb{Z}_n$ & Ring of integers modulo $n$ \\
	$\overline{k} \in \mathbb{Z}_n$ & An element in $\mathbb{Z}_n$ represented by $k \in \mathbb{Z}$  \\
	$\overline{p} \in \mathbb{R}$ & An element in $\mathbb{T}^1\approx S^1= \R / 2\pi \Z$ represented by $p \in \mathbb{R}$  \\
	$X \backslash A$ & Difference of sets  $X \backslash A= \{x \in X \ | \ x\notin A\}$\\
	$|X|$ & Cardinal number of the finite set $X$\\ 
	$f|_{A'}$& Restriction of a map $f: A\to B$ to $A'\subset A$, i.e. $f|_{A'}: A' \to B$\\ 
	\hline
\end{tabular}
\paragraph{}


	



\subsection{Prototype. Inverse limits of covering projections in topology}\label{inf_to}
\subsubsection{Topological construction}
\paragraph{} This subsection is concerned with a topological construction of the inverse limit in the category of covering projections. 
\begin{defn}\label{comm_cov_pr_defn}\cite{spanier:at}
	Let $\widetilde{\pi}: \widetilde{\mathcal{X}} \to \mathcal{X}$ be a continuous map. An open subset $\mathcal{U} \subset \mathcal{X}$ is said to be {\it evenly covered } by $\widetilde{\pi}$ if $\widetilde{\pi}^{-1}(\mathcal U)$ is the disjoint union of open subsets of $\widetilde{\mathcal{X}}$ each of which is mapped homeomorphically onto $\mathcal{U}$ by $\widetilde{\pi}$. A continuous map $\widetilde{\pi}: \widetilde{\mathcal{X}} \to \mathcal{X}$ is called a {\it covering projection} if each point $x \in \mathcal{X}$ has an open neighborhood evenly covered by $\widetilde{\pi}$. $\widetilde{\mathcal{X}}$ is called the {
		\it covering space} and $\mathcal{X}$ the {\it base space} of the covering projection.
\end{defn}
\begin{defn}\cite{spanier:at}
	A fibration $p: \mathcal{\widetilde{X}} \to \mathcal{X}$ with unique path lifting is said to be  {\it regular} if, given any closed path $\omega$ in $\mathcal{X}$, either every lifting of $\omega$ is closed or none is closed.
\end{defn}
\begin{defn}\cite{spanier:at}
A topological space $\mathcal X$ is said to be \textit{locally path-connected} if the path components of open sets are open.
\end{defn}
\paragraph{} Denote by $\pi_1$ the functor of fundamental group \cite{spanier:at}.
\begin{thm}\label{locally_path_lem}\cite{spanier:at}
Let $p: \widetilde{\mathcal X} \to \mathcal X$ be a fibration with unique path lifting and assume that a nonempty $\widetilde{\mathcal X}$ is a locally path-connected space. Then $p$ is regular if and only if for some $\widetilde{x}_0 \in  \widetilde{\mathcal X}$, $\pi_1\left(p\right)\pi_1\left(\widetilde{\mathcal X}, \widetilde{x}_0\right)$ is a normal subgroup of $\pi_1\left(\mathcal X, p\left(\widetilde{x}_0\right)\right)$.
\end{thm}
\begin{defn}\label{cov_proj_cov_grp}\cite{spanier:at}
	Let $p: \mathcal{\widetilde{X}} \to \mathcal{X}$ be a covering projection.  A self-equivalence is a homeomorphism $f:\mathcal{\widetilde{X}}\to\mathcal{\widetilde{X}}$ such that $p \circ f = p$. This group of such homeomorphisms is said to be the {\it group of covering transformations} of $p$ or the {\it covering group}. Denote by $G(\mathcal{\widetilde{X}}|\mathcal{X})$ this group.
\end{defn}

\begin{thm}\label{g_cov_thm}\cite{spanier:at}
Let $G$ be a properly discontinuous group of homeomorphisms of a space $\mathcal X$. Then the projection $\mathcal X \to \mathcal X/G$ of $\mathcal X$ to the orbit space $\mathcal X/G$ is a covering projection. If $\mathcal X$ is connected then this covering projection is regular and $G$ is its group of covering transformations.
\end{thm}
\begin{empt}\label{fin_prop_disc_empt}
	\cite{spanier:at} A finite group action without fixed points on a Hausdorff space is properly discontinuous.
\end{empt}
\begin{prop}\cite{spanier:at}
	If $p: \mathcal{\widetilde{X}} \to \mathcal{X}$ is a regular covering projection and $\mathcal{\widetilde{X}}$ is connected and locally path connected, then $\mathcal{X}$ is homeomorphic to space of orbits of $G(\mathcal{\widetilde{X}}|\mathcal{X})$, i.e. $\mathcal{X} \approx \mathcal{\widetilde{X}}/G(\mathcal{\widetilde{X}}|\mathcal{X})$. So $p$ is a principal bundle.
\end{prop}
\begin{cor}\label{top_cov_from_pi1_cor}\cite{spanier:at}
	Let $p: \widetilde{\mathcal X} \to \mathcal X$ be a fibration with a unique path lifting. If $ \widetilde{\mathcal X}$ is connected and locally path-connected and $\widetilde{x}_0 \in \widetilde{\mathcal X}$ then $p$ is regular if and only if $G\left(\widetilde{\mathcal X}~|~{\mathcal X} \right)$ transitively acts on each fiber of $p$, in which case 
	$$
	\psi: G\left(\widetilde{\mathcal X}~|~{\mathcal X} \right) \approx \pi_1\left(\mathcal X, p\left( \widetilde{x}_0\right)  \right) / \pi_1\left( p\right)\pi_1\left(\widetilde{\mathcal X}, \widetilde{x}_0 \right).  
	$$
\end{cor}
\begin{defn}\label{top_sec_defn}
	Let $\mathcal{X}$ be a  second-countable \cite{munkres:topology} locally compact connected Hausdorff space. The sequence of finitely listed covering projections  	
	\begin{equation*}
	\mathcal{X} = \mathcal{X}_0 \xleftarrow{}... \xleftarrow{} \mathcal{X}_n \xleftarrow{} ... 
	\end{equation*}
	is said to be a \textit{(topological)  finite covering sequence} if following conditions hold:
	\begin{itemize}
		\item   The space $\mathcal{X}_n$ is a  second-countable \cite{munkres:topology} locally compact connected Hausdorff space for any $n \in \mathbb{N}^0$;
		\item Any  map $\mathcal{X}_m  \xleftarrow{} \mathcal{X}_n$ is a regular finitely listed covering projection;
		\item If $k < l < m$ are any nonnegative integer numbers then there is the natural exact sequence
		$$
	\{e\}\to	G\left(\mathcal X_m~|~\mathcal X_l\right) \to 	G\left(\mathcal X_m~|~\mathcal X_k\right)\to 	G\left(\mathcal X_l~|~\mathcal X_k\right)\to \{e\}.
		$$ 
			\end{itemize} 
		For any finite covering sequence we will use following notation
	\begin{equation*}
		\mathfrak{S} = \left\{\mathcal{X} = \mathcal{X}_0 \xleftarrow{}... \xleftarrow{} \mathcal{X}_n \xleftarrow{} ...\right\}= \left\{ \mathcal{X}_0 \xleftarrow{}... \xleftarrow{} \mathcal{X}_n \xleftarrow{} ...\right\},~~\mathfrak{S} \in \mathfrak{FinTop}.
	\end{equation*}

\end{defn}
\begin{exm}
Let $	\mathfrak{S} = \left\{	\mathcal{X} = \mathcal{X}_0 \xleftarrow{}... \xleftarrow{} \mathcal{X}_n \xleftarrow{} ... \right\}$ be a sequence of second-countable  locally compact connected Hausdorff spaces and finitely listed covering projections such that $\mathcal X_n$ is  locally path-connected for any $n \in \mathbb{N}$. It follows from Lemma \ref{locally_path_lem} that if $p > q$ and $f_{pq}:\mathcal X_p \to \mathcal X_q$ then $\pi_1\left(f_{pq}\right)\pi_1\left(\mathcal X_p, x_0\right)$ is a normal subgroup of $\pi_1\left(\mathcal X_q, f_{pq}\left(x_0\right) \right)$. Otherwise from the Corollary \ref{top_cov_from_pi1_cor} it follows that
$$
G\left(\mathcal X_p~|~\mathcal X_q\right) \approx \pi_1\left(\mathcal X_q, f_{pq}\left(x_0\right)\right) / \pi_1\left(f_{pq}\right)\pi_1\left(\mathcal X_p, x_0\right).
$$
Following sequence 
\begin{equation*}
		\begin{split}
	\{e\}\to	\pi_1\left(\mathcal X_l, f_{ml}\left(x_0\right) \right)/ \pi_1\left(f_{ml}\right)\pi_1\left(\mathcal X_m, x_0\right) \to 
	\pi_1\left(\mathcal X_k, f_{mk}\left(x_0\right)\right) / \pi_1\left(f_{mk}\right)\pi_1\left(\mathcal X_m, x_0\right)\to \\
	\to \pi_1\left(\mathcal X_k, f_{mk}\left( x_0\right) \right)/ \pi_1\left(f_{lk}\right)\pi_1\left(\mathcal X_l, f_{ml}\left( x_0\right) \right)	\to \{e\}
	\end{split}
\end{equation*}
is exact. Above sequence is equivalent to the following  sequence
		$$
	\{e\}\to	G\left(\mathcal X_m~|~\mathcal X_l\right) \to 	G\left(\mathcal X_m~|~\mathcal X_k\right)\to 	G\left(\mathcal X_l~|~\mathcal X_k\right)\to \{e\}
		$$ 
	which is also exact. Thus $\mathfrak{S} \in \mathfrak{FinTop}$.
\end{exm}
\begin{defn}\label{top_cov_trans_defn} Let  $\left\{\mathcal{X} = \mathcal{X}_0 \xleftarrow{}... \xleftarrow{} \mathcal{X}_n \xleftarrow{} ...\right\} \in \mathfrak{FinTop}$, and let
 $\widehat{\mathcal{X}} = \varprojlim \mathcal{X}_n$ be the inverse limit  in the category of topological spaces and continuous maps (See \cite{spanier:at}). If $\widehat{\pi}^0: \widehat{\mathcal{X}} \to \mathcal{X}_0$ is the natural continuous map then homeomorphism $g$ of the space $\widehat{\mathcal{X}}$ is said to be a \textit{covering  transformation} if the following condition holds
	$$
	\widehat{\pi}^0 = \widehat{\pi}^0 \circ g.
	$$
	The group $\widehat{G}$ of covering homeomorphisms is said to be the \textit{group of  covering  transformations} of $\mathfrak S$. Denote by $G\left(\widehat{\mathcal{X}}~|~\mathcal X \right)\stackrel{\text{def}}{=}\widehat{G}$. 
\end{defn}
\begin{lem}\label{top_surj_group_lem}
Let  $\left\{\mathcal{X} = \mathcal{X}_0 \xleftarrow{}... \xleftarrow{} \mathcal{X}_n \xleftarrow{} ...\right\} \in \mathfrak{FinTop}$, and let
$\widehat{\mathcal{X}} = \varprojlim \mathcal{X}_n$ be the inverse limit  in the category of topological spaces and continuous maps. For any $n \in \mathbb{N}$ there is the natural surjective homomorphism $h_n:G\left(\widehat{\mathcal{X}}~|~\mathcal X \right) \to G\left(\mathcal{X}_n~|~\mathcal X \right)$ and $\bigcap_{n \in \mathbb{N}} \ker h_n$ is a trivial group.
\end{lem}
\begin{proof}
	For any $n \in \mathbb{N}$ there is the natural continuous map $\widehat{\pi}^n:\widehat{\mathcal{X}} \to \mathcal{X}_n$. Let $x_0 \in \mathcal{X}_0$ and $\widehat{x}_0 \in \widehat{\mathcal{X}}$ be such that $\widehat{\pi}^0\left( \widehat{x}_0\right) = x_0$. Let $\widehat{x}' \in \widehat{\mathcal{X}}$ be such that $\widehat{\pi}^0\left( \widehat{x}'\right)=x_0$.  If  $x'_n = \widehat{\pi}^n\left(\widehat{x}' \right)$, $x_{n} = \widehat{\pi}^n\left(\widehat{x}_0 \right)$ then $\pi^n\left(x_{n} \right)=\pi^n\left(x'_{n} \right)$, where $\pi^n : \mathcal X_n \to \mathcal X$. Since $\pi^n$ is regular there is the unique $g_n \in G\left( \mathcal{X}_n~|~\mathcal{X}\right)$ such that $x'_n = g_n x_{n}$. The sequence $\left\{g_n \in G\left( \mathcal{X}_n~|~\mathcal{X}\right)\right\}_{n \in \N}$ satisfies to the following condition
\begin{equation*}
g_m \circ	\pi^n_m = \pi^n_m \circ g_n
\end{equation*}
	where $n > m$ and $\pi^n_m : \mathcal X_n \to \mathcal X_m$. Let us define $\widehat{g} \in G\left(\widehat{\mathcal{X}}~|~\mathcal X \right)$ such that
\begin{equation*}\tag{*}
	\widehat{\pi}^n\left( \widehat{g}\widehat{x}''\right) = g_n \widehat{\pi}^n \left(\widehat{x}''\right); ~~ \forall n \in \N, ~ \forall \widehat{x}'' \in \widehat{\mathcal{X}}.
\end{equation*}

	If $x''_n = \widehat{\pi}^n\left(\widehat{x}''\right)$ then the sequence $\left\{x''^{g}_n = g_nx''_n \right\}$ satisfies to the following condition
	\begin{equation*}\tag{**}
\pi^n_m\left(x''^{g}_n \right) = {x}''^{g}_m.
	\end{equation*}
	From (**) and properties of inverse limits it follows that there is the unique $\widehat{x}''^g \in \widehat{\mathcal{X}}$ such that 
	$$
	\widehat{\pi}^n \left( \widehat{x}''^g\right) = x''^g_n; ~~ \forall n \in \N.
	$$
The required homeomorphism $\widehat{g} \in  G\left(\widehat{\mathcal{X}}~|~\mathcal X \right)$ which satisfies to (*) is given by
	$$
	\widehat{g} \widehat{x}'' = \widehat{x}''^g.
	$$
The homeomorphism $\widehat{g}$ uniquely depends on $\widehat{x}'\in \left( \widehat{\pi}^0\right)^{-1} \left( x_0\right)$  it follows that there is the 1-1 map $\varphi:\left( \widehat{\pi}^0 \right)^{-1}\left(x_0 \right)\xrightarrow{\approx} G\left(\widehat{\mathcal{X}}~|~\mathcal X \right)$. Since the covering projection $\pi^n : \mathcal{X}_n\to\mathcal X$ is regular there is the 1-1 map	$\varphi^n:\left( \pi^n \right)^{-1}\left(x_0 \right)\xrightarrow{\approx} G\left(\mathcal{X}_n~|~\mathcal X \right)$. The surjective map
	$$
	\left( \widehat{\pi}^0 \right)^{-1}\left(x_0 \right) \to \left( \pi^n \right)^{-1}\left(x_0 \right)
	$$
	induces the surjective homomorphism $G\left(\widehat{\mathcal{X}}~|~\mathcal X \right) \to G\left(\mathcal{X}_n~|~\mathcal X \right)$. If $\widehat{g} \in \bigcap_{n \in \mathbb{N}} \ker h_n$ is not trivial then $\widehat{g} \widehat{x}_0 \neq  \widehat{x}_0$ and there is $n \in \mathbb{N}$ such that $\widehat{\pi}_n\left(\widehat{x}_0\right)\neq \widehat{\pi}_n\left(\widehat{g}\widehat{x}_0\right)= h_n\left(\widehat{g} \right) \widehat{\pi}_n\left(\widehat{x}_0\right)$ so $h_n\left(\widehat{g} \right)  \in G\left(\mathcal{X}_n~|~\mathcal X \right)$ is not trivial and $\widehat{g} \notin \ker h_n$. From this contradiction it follows that $\bigcap_{n \in \mathbb{N}} \ker h_n$ is a trivial group.
\end{proof}
\begin{defn}\label{top_?oh_defn}
	Let $\mathfrak{S} = \left\{ \mathcal{X}_0 \xleftarrow{}... \xleftarrow{} \mathcal{X}_n \xleftarrow{} ...\right\}$ be 
	a finite covering sequence. The pair $\left(\mathcal{Y},\left\{\pi^{\mathcal Y}_n\right\}_{n \in \mathbb{N}} \right) $ of a (discrete) set $\mathcal{Y}$ with and  
	 surjective  maps $\pi^{\mathcal Y}_n:\mathcal{Y} \to \mathcal X_n$ is said to be a \textit{coherent system} if for any $n \in \mathbb{N}^0$ following diagram  
	\newline
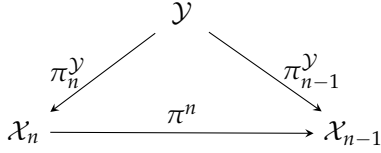
	\begin{tikzpicture}
	\matrix (m) [matrix of math nodes,row sep=3em,column sep=4em,minimum width=2em]
	{
		& \mathcal{Y}  &  \\
		\mathcal{X}_n &  &  \mathcal{X}_{n-1} \\};
	\path[-stealth]
	(m-1-2) edge node [left] {$\pi^{\mathcal Y}_n~$} (m-2-1)
	(m-1-2) edge node [right] {$~\pi^{\mathcal Y}_{n-1}$} (m-2-3)
	(m-2-1) edge node [above] {$\pi^n$}  (m-2-3);
	
	\end{tikzpicture}
	\newline
	is commutative.	
\end{defn}

\begin{defn}\label{comm_top_constr_defn}
	Let $\mathfrak{S} = \left\{ \mathcal{X}_0 \xleftarrow{}... \xleftarrow{} \mathcal{X}_n \xleftarrow{} ...\right\}$ be 
	a topological finite covering sequence. A coherent system $\left(\mathcal{Y},\left\{\pi^{\mathcal Y}_n\right\} \right)$ is said to
 be a \textit{covering projection} of $\mathfrak{S}$ if $\mathcal Y$ is a topological space and $\pi^{\mathcal Y}_n$ is a regular covering projection for any $n \in \mathbb{N}$.  We will use following notation $\left(\mathcal{Y},\left\{\pi^{\mathcal Y}_n\right\} \right)\downarrow \mathfrak{S}$ or simply $\mathcal{Y} \downarrow \mathfrak{S}$.
\end{defn}
\begin{defn}\label{top_spec_defn}
	Let $\left(\mathcal{Y},\left\{\pi^{\mathcal Y}_n\right\} \right)$ be  a coherent system   of $\mathfrak{S}$ and $y \in \mathcal{Y}$. A subset  $\mathcal V \subset \mathcal{Y}$ is said to be \textit{special} if $\pi^{\mathcal Y}_0\left(\mathcal{V} \right)$ is evenly covered by $\mathcal{X}_1 \to \mathcal{X}_0$ and for any  $n \in \mathbb{N}^0$ following conditions hold:
	\begin{itemize}
		\item $\pi^{\mathcal Y}_n\left(\mathcal{V} \right) \subset \mathcal X_n$ is  an open connected set. 
		\item The restriction $\pi^{\mathcal Y}_n|_{\mathcal V}:\mathcal{V}\to \pi^{\mathcal Y}_n\left( {\mathcal V}\right) $ is a bijection;
	\end{itemize}
	
\end{defn}
\begin{rem}
	If $\left(\mathcal{Y},\left\{\pi^{\mathcal Y}_n\right\} \right)$ is  a covering projection of $\mathfrak{S}$ then the topology of $\mathcal{Y}$ is generated by special sets.
\end{rem}
\begin{lem}\label{top_equ_borel_set_lem}
Let $\widehat{\mathcal{X}} = \varprojlim \mathcal{X}_n$ be the inverse limit of the sequence $\mathcal{X}_0 \xleftarrow{}... \xleftarrow{} \mathcal{X}_n \xleftarrow{} ...$ in the category of topological spaces and continuous maps. Any special set of $\widehat{\mathcal{X}}$ is a Borel subset of $\widehat{\mathcal{X}}$.
	
	\end{lem}\label{top_spec_borel_lem}
	\begin{proof}
		If $\mathcal U_n\subset \mathcal X_n$ is an open set then $\widehat{\pi}_n^{-1} \left(\mathcal U_n \right) \subset \widehat{\mathcal X}$ is open. 
	If $\widehat{\mathcal U}$ is a special set then $\widehat{\mathcal U} = \bigcap_{n \in \mathbb{N}} \widehat{\pi}_n^{-1} \circ \widehat{\pi}_n\left(\widehat{\mathcal U}\right)$, i.e. $\widehat{\mathcal U}$ is a countable intersection of open sets. So $\widehat{\mathcal U}$ is a Borel subset.	
	\end{proof}
\begin{defn}\label{comm_top_constr_morph_defn}
	Let us consider the situation of the Definition \ref{comm_top_constr_defn}. A \textit{morphism} from $\left(\mathcal{Y}',\left\{\pi^{\mathcal Y'}_n\right\}\right)\downarrow\mathfrak{S}$ to $\left(\mathcal{Y}'',\left\{\pi^{\mathcal Y''}_n\right\}\right)\downarrow\mathfrak{S}$ is a covering projection $f: \mathcal{Y}' \to \mathcal{Y}''$ such that
	$$
	\pi_n^{\mathcal Y''} \circ f= \pi_n^{\mathcal Y'} 
	$$
	for any $n \in \N$.
	
\end{defn}
\begin{empt}\label{comm_top_constr}
	There is a category with objects and morphisms described by Definitions \ref{comm_top_constr_defn}, \ref{comm_top_constr_morph_defn}. Denote by $\downarrow \mathfrak S$ this category.
		\end{empt}
		\begin{lem}\label{comm_universal_covering_lem}
			There is the final object of the category $\downarrow \mathfrak S$  described in \ref{comm_top_constr}.
		\end{lem}
\begin{proof}
	Let $\widehat{\mathcal{X}} = \varprojlim \mathcal{X}_n$ be the inverse limit of the sequence $\mathcal{X}_0 \xleftarrow{}... \xleftarrow{} \mathcal{X}_n \xleftarrow{} ...$ in the category of topological spaces and continuous maps. Denote by $\widetilde{\mathcal X}$ a topological space which coincides with  $\widehat{\mathcal X}$ as a set and the topology on $\widetilde{\mathcal X}$ is generated by special sets.  If $x_n \in \mathcal X_n$ is a point then there is $\widetilde{x}\in \widetilde{\mathcal X}=\widehat{\mathcal X}$ such that $x_n = \widehat{\pi}^n\left( \widetilde{x}\right)$ and there is a special subset $\widehat{\mathcal U}$ such that $\widetilde{x} \in \widehat{\mathcal U}$.  From the construction of special subsets it follows that
	\begin{itemize}
		\item $\mathcal U_n = \widehat{\pi}^n\left( \widehat{\mathcal U}\right)$ is an open neighborhood of $x_n$;
		\item  $$\left(\widehat{\pi}^n \right)^{-1} \left(\mathcal U_n \right) = \bigsqcup_{g \in \ker\left(G\left(\widehat{\mathcal{X}}~|~\mathcal X \right) \to G\left(\mathcal{X}_n~|~\mathcal X \right) \right)  } g \widehat{\mathcal U};$$

		\item For any $g \in \ker\left(G\left(\widehat{\mathcal{X}}~|~\mathcal X \right) \to G\left(\mathcal{X}_n~|~\mathcal X \right) \right)$ the set $g \widehat{\mathcal U}$  mapped homeomorphically onto $\mathcal{U}_n$.
	\end{itemize}
So the natural map 	$\pi^{\widetilde{\mathcal X}}_n:\widetilde{\mathcal X} \to \mathcal X_n$ is a covering projection and $\widetilde{\mathcal X}$ is the object of $\downarrow \mathfrak S$. If $\left(\mathcal{Y},\left\{\pi^{\mathcal Y}_n\right\} \right)$ is  a covering projection of $\mathfrak{S}$ then there is the natural continuous map $\mathcal{Y}\to \widehat{\mathcal{X}}$. Since the continuous map $\widetilde{\mathcal{X}}\to \widehat{\mathcal{X}}$ is bijective there is the natural map $\widetilde{\pi}:\mathcal{Y} \to \widetilde{\mathcal{X}}$. If $\widetilde{x} \in \widetilde{\mathcal{X}}$ is a point then there is $y \in \mathcal Y$ such that $\widetilde{x} = \widetilde{\pi}\left( y\right)$. Let $G \subset G\left(\mathcal Y~|~\mathcal X \right)$ be such that $\widetilde{\pi}\left( Gy\right)  =  \widetilde{x}$. If $\widehat{\mathcal U}$ is a special neighborhood of $\widetilde{x}$ then there is a connected neighborhood $\mathcal U_y$ of $y$ which is mapped homeomorphically onto $\widehat{\pi}_0\left(\widehat{\mathcal U}\right) \subset \mathcal X_0$. It follows that
$$
\widetilde{\pi}^{-1}\left(\widehat{\mathcal U}\right)= \bigsqcup_{g \in G}g \mathcal U_y,
$$
i.e. $\widehat{\mathcal U}$ is evenly covered by $\widetilde{\pi}$, i.e. $\widetilde{\pi}$ is a covering projection.
\end{proof}
\begin{defn}\label{top_topological_inv_lim_defn}
	 The final object $\left(\widetilde{\mathcal{X}},\left\{\pi^{\widetilde{\mathcal X}}_n\right\} \right)$ of the category $\downarrow\mathfrak{S}$ is said to be the \textit{inverse limit} of $\mathfrak{S}$.  The notation $\left(\widetilde{\mathcal{X}},\left\{\pi^{\widetilde{\mathcal X}}_n\right\} \right) = \varprojlim \mathfrak{S}$ or simply $~\widetilde{\mathcal{X}} =  \varprojlim \mathfrak{S}$ will be used.
\end{defn}
\begin{lem}\label{top_biject_lem}
If
		$\mathfrak{S} = \left\{\mathcal{X} = \mathcal{X}_0 \xleftarrow{}... \xleftarrow{} \mathcal{X}_n \xleftarrow{} ...\right\} \in \mathfrak{FinTop}$, $~\widehat{\mathcal X} = \varprojlim  \mathcal X_n$, $~\widetilde{\mathcal X} = \varprojlim \mathfrak{S}$ then there is the unique isomorphism $G\left(\widetilde{\mathcal{X}}~|~\mathcal{X} \right) \xrightarrow{\approx} G\left(\widehat{\mathcal{X}}~|~\mathcal{X} \right)$ induced by the map $\widetilde{\mathcal{X}} \to \widehat{\mathcal{X}}$.

\end{lem}

\begin{proof}

Since $\widetilde{\mathcal X}$ coincides with $\widehat{\mathcal X}$ as a set, and  the topology of $\widetilde{\mathcal X}$ is finer then the topology of $\widehat{\mathcal X}$ there is the natural injective map $G\left(\widetilde{\mathcal{X}}~|~\mathcal{X} \right)\to G\left(\widehat{\mathcal{X}}~|~\mathcal{X} \right)$. Let 

If $\widehat{g}\in G\left(\widehat{\mathcal{X}}~|~\mathcal{X} \right)$ and $\widehat{\mathcal U}$ is a special set, then for any $n \in \mathbb{N}$ following condition holds:
\begin{equation}\tag{*}
\widehat{\pi}^n\left(\widehat{g} \widehat{\mathcal U} \right)= h_n\left( \widehat{g}\right)\circ\widehat{\pi}^n\left( \widehat{\mathcal U} \right)  
\end{equation}
where $\widehat{\pi}^n: \widehat{\mathcal X} \to \mathcal X_n$ is the natural map, and $h_n : G\left(\widehat{\mathcal{X}}~|~\mathcal{X} \right)\to G\left(\mathcal{X}_n~|~\mathcal{X} \right)$ is given by the Lemma \ref{top_surj_group_lem}. $h_n(g)$ is a homeomorphism, so from (*) it follows that $\widehat{\pi}^n\left(\widehat{g} \widehat{\mathcal U} \right)$ is an open subset of $\mathcal X_n$. Hence $\widehat{g} \widehat{\mathcal U}$ is special. So $g$ maps special sets to special sets. Since topology of $\widetilde{\mathcal X}$ is generated by special sets the map $\widehat{g}$ is a homeomorphism of $\widetilde{\mathcal X}$, i.e. $\widehat{g} \in  G\left(\widetilde{\mathcal{X}}~|~\mathcal{X} \right)$.

\end{proof}

\subsubsection{Algebraic construction in brief}\label{comm_alg_constr_susub}
\paragraph{}
The inverse limit of covering projections $\widetilde{\mathcal X}$ is obtained from inverse limit of topological spaces $\widehat{\mathcal X}$ by a change of a topology. The topology of  $\widetilde{\mathcal X}$ is finer then topology of $\widehat{\mathcal X}$, it means that $C_0\left(\widehat{\mathcal X}\right)$ is a subalgebra of $C_b\left(\widetilde{\mathcal X}\right)$. The topology of $\widetilde{\mathcal X}$ is obtained from topology of $\widehat{\mathcal X}$ by addition of special subsets. Addition of new sets to a topology is equivalent to addition of new elements to $C_0\left(\widehat{\mathcal X}\right)$. To obtain $C_b\left(\widetilde{\mathcal X}\right)$ we will add to  $C_0\left(\widehat{\mathcal X} \right) $ special elements (See Definition \ref{special_el_defn}). If $\widetilde{\mathcal U}\subset \widetilde{\mathcal X}$ is a special set and $\widetilde{a} \in C_c\left( \widetilde{\mathcal X}\right)$ is positive element such that $\widetilde{a}|_{\widetilde{\mathcal X} \backslash \widetilde{\mathcal U}}= \{0\}$ then following condition holds
$$
a\left(  \widetilde{\pi}_n\left(\widetilde{x}\right) \right)= \left( \sum_{\widehat{g} \in \widehat{G}} \widehat{g} \widetilde{a}\right) \left(  \widetilde{\pi}_n\left(\widetilde{x}\right) \right)= 
\left\{\begin{array}{c l}
\widetilde{a}\left( \widetilde{x} \right)  & \widetilde{x} \in \widetilde{\mathcal U} \\
0 & \widetilde{\pi}_n\left(\widetilde{x}\right) \notin \widetilde{\pi}_n\left(\widetilde{\mathcal U}\right)
\end{array}\right. \in C_c\left(\mathcal X_0\right) ,
$$
From above equation it follows that
\begin{equation}\tag{*}
a^2 = \sum_{\widehat{g} \in \widehat{G}} \widehat{g} \widetilde{a}^2.
\end{equation}
The equation (*) is purely algebraic and related to special subsets. From the Theorem \ref{comm_main_thm} it follows that the algebraic condition (*)  is sufficient for  construction of  $C_0\left( \widetilde{\mathcal X}\right)$. Thus noncommutative inverse limits of covering projections can be constructed by purely algebraic methods.

\subsection{Locally compact spaces and partitions of unity}

\paragraph*{} In this article we consider second-countable locally compact Hausdorff spaces only \cite{munkres:topology}. So we will say a "topological space" (resp. "compact  space" ) instead "locally compact second-countable Hausdorff space" (resp. "compact second-countable Hausdorff space"). This subsection contains well known facts, I follow to \cite{chun-yen:separability,munkres:topology}. 
\begin{defn}\cite{munkres:topology}
If $\phi: \mathcal X \to \mathbb{R}$ is continuous then \textit{support} of $\phi$ is defined to be the closure of the set $\phi^{-1}\left(\mathbb{R}\backslash \{0\}\right)$ Thus if $x$ lies outside the support, there is some neighborhood of $x$ on which $\phi$ vanishes. Denote by $\supp \phi$ the support of $\phi$.
\end{defn}
\paragraph{}
Similarly we define the support of $\C$-valued continuous functions.
There are two equivalent definitions of $C_0\left(\mathcal{X}\right)$ and both of them are used in this article.
\begin{defn}\label{c_c_def_1}
An algebra $C_0\left(\mathcal{X}\right)$ is the norm closure of the algebra $C_c\left(\mathcal{X}\right)$ of compactly supported continuous functions.
\end{defn}
\begin{defn}\label{c_c_def_2}
A $C^*$-algebra $C_0\left(\mathcal{X}\right)$ is given by the following equation
\begin{equation*}
C_0\left(\mathcal{X}\right) = \left\{\varphi \in C_b\left(\mathcal{X}\right) \ | \ \forall \varepsilon > 0 \ \ \exists K \subset \mathcal{X} \ ( K \text{ is compact}) \ \& \ \forall x \in \mathcal X \backslash K \ \left|\varphi\left(x\right)\right| < \varepsilon  \right\},
\end{equation*}
i.e.
\begin{equation*}
\left\|\varphi|_{\mathcal X \backslash K}\right\| < \varepsilon.
\end{equation*}
\end{defn}
\begin{thm}\label{comm_sep_thm}\cite{chun-yen:separability}
For a locally compact Hausdorff space $\mathcal X$, the following are
equivalent:
\begin{enumerate}
\item[(a)] The abelian $C^*$-algebra $C_0\left(\mathcal X \right)$  is separable;
\item[(b)]$\mathcal X$ is $\sigma$-compact and metrizable;
\item[(c)]  $\mathcal X$ is second countable.
\end{enumerate}

\end{thm}

\begin{thm}\cite{munkres:topology}
	Every metrizable space is paracompact.
\end{thm}

\begin{defn}\cite{munkres:topology}
	Let $\left\{\mathcal U_\alpha\in \mathcal X\right\}_{\alpha \in J}$ be an indexed open covering of $\mathcal{X}$. An indexed family of functions 
	\begin{equation*}
	\phi_\alpha : \mathcal X \to \left[0,1\right]
	\end{equation*}
	is said to be a {\it partition of unity }, dominated by $\left\{\mathcal{U}_\alpha \right\}_{\alpha \in J}$, if:
	\begin{enumerate}
		\item[(a)] $\mathfrak{supp}\left( \phi_\alpha\right) \subset U_\alpha$;
		\item[(b)] The family $\left\{\supp \phi_\al\right\}$ is locally finite;
		\item[(c)] $\sum_{\alpha \in J}\phi_\alpha\left(x\right)=1$ for any $x \in \mathcal X$.
	\end{enumerate}
\end{defn}

\begin{thm}\cite{munkres:topology}
	Let $\mathcal X$ be a paracompact Hausdorff space, let $\left\{\mathcal U_\alpha\in \mathcal X\right\}_{\alpha \in J}$ be an indexed open covering of $\mathcal{X}$. Then there exists a partition of unity, dominated by $\left\{\mathcal{U}_\alpha \right\}$.  
\end{thm}
  \begin{defn}\label{lift_desc_defn}
  	Let $\widetilde{\pi}: \widetilde{\mathcal X} \to \mathcal X$ be a covering projection and let $\widetilde{\mathcal U}\subset \widetilde{\mathcal X}$ be a connected open set which is mapped homeomorphically onto $\mathcal U= \widetilde{\pi}\left(\widetilde{\mathcal U}\right)$. If $\varphi \in C_0\left(\mathcal{X}\right)$ is such that $\mathfrak{supp}\left( \varphi\right) \subset \mathcal U$ then a function $\widetilde{\varphi} \in C_0\left(\widetilde{\mathcal X}\right)$ given by
  	\begin{equation*}
  	\widetilde{\varphi}\left(\widetilde{x}\right)=\left\{
  	\begin{array}{c l}
  	\varphi\left(\widetilde{\pi}\left(\widetilde{x}\right)\right) & \widetilde{x} \in \widetilde{\mathcal{U}}  \\
  	0 & \widetilde{x} \notin \widetilde{\mathcal{U}}
  	\end{array}\right.
  	\end{equation*}
  	is said to be the $\widetilde{\mathcal{U}}$-{\it lift } of $\varphi$. 
  	If $\widetilde{\varphi} \in C_0\left(\widetilde{\mathcal{X}}\right)$  is such that $\mathfrak{supp}\left( \widetilde{\varphi}\right)  \subset \widetilde{\mathcal U}$  then a function $\varphi \in C_0\left(\mathcal X\right)$ given by
  	\begin{equation*}
  	\varphi\left(x\right)=\left\{
  	\begin{array}{c l}
  	\widetilde{\varphi}\left(\widetilde{x}\right) & x \in \mathcal{U}~ \&  \  \widetilde{x} \in \widetilde{\mathcal U}~ \& \  \widetilde{\pi}\left(\widetilde{x}\right)=x   \\
  	0 & x \notin \mathcal{U}      \end{array}\right.
  	\end{equation*}
  	is said to be the {\it descent} of $\widetilde{\varphi}$, and we write $\varphi \stackrel{\mathrm{def}}{=} \mathfrak{Desc}_{\widetilde{\pi}}\left(\widetilde{\varphi}\right)$.  
  	If the closure $\overline{\widetilde{\mathcal{U}}}$ of $\widetilde{\mathcal{U}}$ is compact and mapped homeomorphically on  $\overline{\pi}\left(\overline{\widetilde{\mathcal{U}}}\right)$ and $\phi \in C_0\left(\mathcal{X}\right)$ then  $\widetilde{\phi} \in C\left(\overline{\widetilde{\mathcal{U}}}\right)$ given by $\widetilde{\phi}(x)= \phi\left(\widetilde{\pi}\left(x\right)\right)$ is said to be  the \textit{compact}  $\overline{\widetilde{\mathcal{U}}}$-\textit{lift}. 
  \end{defn}

\paragraph{}

Results of this subsection are used for a construction of partitions of unity  of covering projections. Let $\mathcal{X}$ be a second countable locally compact Hausdorff space, and let $\widetilde{\pi}:\widetilde{\mathcal X} \to \mathcal X$ be a regular covering projection. Let $\left\{\mathcal U_\iota \subset  \mathcal X\right\}_{\iota \in I}$ be a family such that $ \mathcal X = \bigcup_{\iota \in I}  \mathcal U_\iota$ and for any $\iota\in I$
\begin{itemize}
\item  $\mathcal U_\iota$ is an open connected and relatively compact subset of  $\mathcal X$;
\item  $\mathcal U_\iota$ is evenly covered by $\widetilde{\pi}$;
\end{itemize}
There is the partition of unity dominated by $\left\{\mathcal U_\iota\right\}_{\iota \in I}$ given by
\begin{equation*} 
1_{C_b\left(\mathcal{X}\right)}=1_{M\left(C_0\left(\mathcal{X}\right)\right)}= \sum_{\iota \in I}a_\iota~,
\end{equation*}
i.e. $\supp \left( a_\iota\right) \subset \mathcal{U}_\iota$. Since $\mathcal U_\iota$ is relatively compact we have $a_\iota \in C_c\left(\mathcal X\right)$. For any $\iota$ we can select an open connected subset $\widetilde{\mathcal{U}}_\iota$ such that the restriction $\widetilde{\pi}|_{\widetilde{\mathcal{U}}_\iota}: \widetilde{\mathcal{U}}_\iota \to \mathcal{U}_\iota$ is a homeomorphism and an element $\widetilde{a}_\iota\in C_c\left(\widetilde{\mathcal X}\right)$ which is the $\widetilde{\mathcal U}_\iota$-lift of $a_\iota$.
There is the following partition of the unity
\begin{equation}\label{covering_partition_eqn}
1_{C_b\left(\widetilde{\mathcal{X}}\right)}= 1_{M\left(C_0\left(\widetilde{\mathcal{X}}\right)\right)}= \sum_{(g,\iota) \in G\left(\widetilde{\mathcal{X}}|\mathcal{X}\right) \times I}g \widetilde{a}_\iota.
\end{equation}
If $\widetilde{e}_\iota = \sqrt{\widetilde{a}_\iota}$ and $g \in G\left(\widetilde{\mathcal{X}}|\mathcal{X}\right)$ is a nontrivial element then from $g\widetilde{\mathcal{U}}_\iota \bigcap \widetilde{\mathcal{U}}_\iota= \emptyset$ it follows that
\begin{equation}\label{zero_ge_iota_eqn}
\widetilde{e}_\iota\left( g \widetilde{e}_\iota\right) = \left(g \widetilde{e}_\iota \right) \widetilde{e}_\iota=0;~ \text{ for any nontrivial } g \in G\left(\widetilde{\mathcal{X}}|\mathcal{X}\right).
\end{equation}

\begin{defn}\label{cov_cov_part_unity_defn}
If $\widetilde{\pi}:\widetilde{\mathcal X} \to \mathcal X$ is a covering projection then a described above family $\left\{\widetilde{\mathcal{U}}_\iota \subset\widetilde{\mathcal X}\right\}_{\iota \in I}$ is said to be a \textit{generalized fundamental domain}. 
The partition of unity \eqref{covering_partition_eqn} is said to be a \textit{covering partition of unity} (subordinated to $\left\{\widetilde{\mathcal{U}}_\iota \right\})$.
\end{defn}

\subsection{Hilbert $C^*$-modules}
\paragraph{} We refer to \cite{blackadar:ko} 
for definition of Hilbert $C^*$-modules, or simply Hilbert modules. If $X_A$ is a right Hilbert $C^*$-module then we denote by $\langle\cdot,\cdot\rangle_{X_A}$ the $A$-valued sesquilinear product on $X_A$, i.e. $\langle\xi,\eta\rangle_{X_A}\in A$; $\forall\xi,\eta \in X_A$.
 For any $\xi, \zeta \in X_A$ let us define an $A$-endomorphism $\theta_{\xi, \zeta}$ given by  $\theta_{\xi, \zeta}(\eta)=\xi \langle \zeta, \eta \rangle_{X_A}$ where $\eta \in X_A$. Operator  $\theta_{\xi, \zeta}$ is said to be a {\it rank one} operator and will be denoted by $\xi \rangle\langle \zeta$. The norm completion of an algebra generated by rank-one operators  $\theta_{\xi, \zeta}$ is said to be the {\it algebra of compact operators $\mathcal{K}(X_A)$}. We suppose that there is a left action of $\mathcal{K}(X_A)$ on $X_A$ which is $A$-linear, i.e. action of  $\mathcal{K}(X_A)$ commutes with action of $A$.

\subsection{Strong and/or weak extension}
\paragraph{}In this section I follow to \cite{pedersen:ca_aut}.
\begin{defn}\cite{pedersen:ca_aut}
	Let $A$ be a $C^*$-algebra.  The {\it strict topology} on the multiplier algebra $M(A)$ is the topology generated by seminorms $\vertiii{x}_a = \|ax\| + \|xa\|$, ($a\in A$). If $x \in M(A)$  and a sequence of partial sums $\sum_{i=1}^{n}a_i$ ($n = 1,2, ...$), ($a_i \in A$) tends to $x$ in the strict topology then we shall write
	\begin{equation*}
	x = \sum_{i=1}^{\infty}a_i.
	\end{equation*}
\end{defn}
\begin{defn}\cite{pedersen:ca_aut} Let $H$ be a Hilbert space. The {\it strong} topology on $B\left(H\right)$ is the locally convex vector space topology associated with the family of seminorms of the form $x \mapsto \|x\xi\|$, $x \in B(H)$, $\xi \in H$.
\end{defn}
\begin{defn}\cite{pedersen:ca_aut} Let $H$ be a Hilbert space. The {\it weak} topology on $B\left(H\right)$ is the locally convex vector space topology associated with the family of seminorms of the form $x \mapsto \left|\left(x\xi, \eta\right)\right|$, $x \in B(H)$, $\xi, \eta \in H$.
\end{defn}

\begin{thm}\label{vN_thm}\cite{pedersen:ca_aut}
	Let $M$ be a $C^*$-subalgebra of $B(H)$, containing the identity operator. The following conditions are equivalent:
	\begin{itemize}
		\item $M=M''$ where $M''$ is the bicommutant of $M$;
		\item $M$ is weakly closed;
		\item $M$ is strongly closed.
	\end{itemize}
\end{thm}

\begin{defn}
	Any $C^*$-algebra $M$ is said to be a {\it von Neumann algebra} or a {\it $W^*$- algebra} if $M$ satisfies to the conditions of the Theorem \ref{vN_thm}.
\end{defn}

\begin{defn}\cite{pedersen:ca_aut}
	Let  $A \subset B(H)$ be a $C^*$-algebra and  $A$ acts non-degenerately on $H$. Denote by $A''$ the strong closure of $A$ in $B(H)$. $A''$ is an unital weakly closed $C^*$-algebra . The algebra  $A''$ is said to be the  {\it enveloping von Neumann algebra} or  the {\it enveloping $W^*$-algebra}  of $A$. 
\end{defn}

\begin{lem}\label{increasing_convergent_w}\cite{pedersen:ca_aut} Let $\Lambda$ be an increasing net in the partial ordering.  Let $\left\{x_\lambda \right\}_{\la \in \La}$ be an increasing net of self-adjoint operators in $B\left(H\right)$, i.e. $\la \le \mu$ implies $x_\la \le x_\mu$. If $\left\|x_\la\right\| \le \ga$ for some $\ga \in \mathbb{R}$ and all $\la$ then $\left\{x_\lambda \right\}$ is strongly convergent to a self-adjoint element $x \in B\left(H\right)$ with $\left\|x_\la\right\| \le \ga$.
\end{lem}


\begin{empt}\label{comm_gns_constr}
	Any state $\tau$ on a $C^*$-algebra $A$  induces a  GNS representation  \cite{murphy}. There is a $\mathbb{C}$-valued product on $A$ given by
	\begin{equation*}
	\left(a, b\right)=\tau\left(a^*b\right).
	\end{equation*}
	This product induces a scalar product on $A/\mathcal{I}_\tau$ where $\mathcal{I}_\tau =\left\{a \in A \ | \ \tau(a^*a)=0\right\}$. So $A/\mathcal{I}_\tau$ is a phe-Hilbert space. Let denote by $L^2\left(A, \tau\right)$ the Hilbert  completion of $A/\mathcal{I}_\tau$.  The Hilbert space  $L^2\left(A, \tau\right)$ is a space of a  GNS representation of $A$.
\end{empt}

\begin{empt}\label{l2_mu}
	If $\mathcal X$ is a second-countable locally compact Hausdorff space then from the Theorem \ref{comm_sep_thm} it follows that $C_0\left(\mathcal X\right)$ is a separable algebra.  $C_0\left(\mathcal X\right)$ has a state $\tau$ such that associated with $\tau$   GNS representation  \cite{murphy} is faithful. From \cite{bogachev_measure_v2} it follows that the state $\tau$ can be represented as the following integral
	\begin{equation}\label{hilb_integral}
	\tau\left(a\right)= \int_{\mathcal X}a \ d\mu
	\end{equation}
where $\mu$ is a positive Radon measure. In analogy with the Riemann integration, one can define the integral of a 	bounded continuous function $a$ on $\mathcal{X}$. There is a $\mathbb{C}$-valued product on $C_0\left(\mathcal X\right)$ given by
	\begin{equation*}
	\left(a, b\right)=\tau\left(a^*b\right)= \int_{\mathcal X}a^*b \ d\mu,
	\end{equation*}
	whence $C_0\left(\mathcal X\right)$ is a phe-Hilbert space. Denote by $L^2\left(C_0\left(\mathcal X\right), \tau\right)$ or $L^2\left(\mathcal X, \mu\right)$ the Hilbert space completion of $C_0\left(\mathcal X\right)$. From  \cite{murphy,takesaki:oa_ii} it   follows that $W^*$-enveloping algebra of $C_0\left(\mathcal X\right)$ is isomorphic to the algebra $L^{\infty}\left(\mathcal X, \mu\right)$ (of classes of) essentially bounded complex-valued measurable functions. The space $L^{\infty}\left(\mathcal X, \mu\right)$ depends only on null-measurable sets and does not depend on $\mu$-measures of non-null measurable \cite{bogachev_measure_v2}  sets, so the notation $L^{\infty}\left(\mathcal X, \mu\right)$ will be replaced with $L^{\infty}\left(\mathcal X\right)$. 
\end{empt}
\subsection{Galois extensions and covering projections}\label{galois_subsection}

 \paragraph*{}This article compiles ideas of algebra and topology. 
 \begin{empt}\label{alg_top_constr}
 
 Following table contains the mapping between Galois extension of fields and topological covering projections.
\newline
\begin{tabular}{|c|c|}
	\hline
Topology & Theory of fields\\
\hline
	Covering projection & Algebraic extension of the field\\
	Regular covering projection & Normal extension\\
	Unramified covering projection  & 	Separable extension\\
	Universal covering projection  & 	Algebraic closure \\
	\hline
\end{tabular}
 \paragraph*{}
 Complex algebraic varieties have structure of both functional fields \cite{hartshorne:ag} and Hausdorff locally compact spaces, i.e. both columns of the above table reflect the same phenomenon. Analogy between Galois extensions of fields and covering projections is discussed in \cite{hajac:toknotes, milne:etale}. During last decades the advanced theory of Galois extensions of noncommutative algebras had been developed.  Galois extensions became special cases of Hopf-Galois ones. There is the  Category Theory approach to Hopf-Galois extensions, based on theory of monads, comonads and adjoint functors \cite{hajac:toknotes}. But this approach cannot be used for nonunital algebras. The Category Theory replaces elements of algebras with endomorphisms of objects. However any object of the Category Theory contains the identity endomorphism which corresponds to the unity of the algebra. To avoid this obstacle the notion of the non-unital $C^*$-category was introduced \cite{mitchener:c_cat}. But this notion is incompatible with adjoint functors. So the Category Theory of Galois extensions cannot be directly used for nonunial $C^*$-algebras. But there is a modification of the Galois theory which  replaces the unity with the approximate unity.  Following theory which had been developed long time ago, is used as a prototype.
 
  \end{empt}
 \begin{empt}\textit{Galois theory of noncommutative algebras}. I follow to \cite{demeyer:genreal_galois,miyashita_fin_outer_gal}.
Let $K$ be a field. Throughout $\La$ will denote a $K$ algebra, $C$ will denote the center of $\La$ ($C=\mathfrak{Z}(\La)$). $G$ will denote a finite group represented as ring automorphisms of $\La$ and $\Ga$ the subring of all elements of $\La$ left invariant by all the automorphisms in $G$ ($\Ga = \La^G$). Let $\Delta\left(\La:G\right)$ be the crossed product of $\La$ and $G$ with trivial factor
set. That is 
$$\Delta\left(\La:G\right) = \sum_{\sigma \in G} \La U_\sigma \text{ such that}$$
$$
x_1U_\sigma x_2 U_\tau = x_1 \sigma(x_1)U_{\sigma  \tau}; ~ x_1, x_2 \in \La; ~ \sigma,\tau \in G.
$$
View $\La$ as a right $\Ga$ module and define $$j : \Delta\left(\La:G\right) \to \Hom_{\Ga}\left(\La, \La\right) \text{ by}$$ 
$$
j\left(aU_\sigma\right)\left( x\right)  = a \sigma(x);~a,x \in \La;~\sigma \in G. 
$$
 \end{empt}
 \begin{thm}\label{galois_equiv_thm}\cite{demeyer:genreal_galois} The following are equivalent:
 \begin{enumerate}
 \item[(a)]$\La$ is finitely generated projective as a right $\Ga$ module and
$j : \Delta\left(\La:G\right) \to \Hom_{\Ga}\left(\La, \La\right)$ is an isomorphism.
 \item[(b)] There exists $x_1,...,x_n, y_1,...,y_n \in \La$ such that
\begin{equation*}
\sum_{j =1}^{n} x_j\sigma(y_j)=\left\{
\begin{array}{c l}
   1 & \sigma \in  G \text{ is trivial}\\
   0 & \sigma \in  G \text{ is not trivial}
\end{array}\right.
\end{equation*}
 \end{enumerate}
 \end{thm}
\begin{defn}\label{galois_unital_defn}\cite{demeyer:genreal_galois}
If condition (a) the Theorem \ref{galois_equiv_thm} holds then $\La$ is said to be a \textit{Galois extension} of $\Ga$.
\end{defn}

\section{Noncommutative finite covering projections}
\begin{defn}
	If $A$ is a $C^*$-algebra then an action of a group $G$ is said to be {\it involutive } if $ga^* = \left(ga\right)^*$ for any $a \in A$ and $g\in G$. Action is said to be \textit{non-degenerated} if for any nontrivial $g \in G$ there is $a \in A$ such that $ga\neq a$. 
\end{defn}
\paragraph{}A substantial feature of topological algebras is an existence of limits and infinite sums. We would like supply a $C^*$-algebraic  (nonunital) version  of the Theorem \ref{galois_equiv_thm} and the Definition \ref{galois_unital_defn} which used infinite sums instead finite ones. 
\begin{defn}\label{fin_def}
Let $\pi:A \to \widetilde{A}$ be an injective *-homomorphism of $C^*$-algebras, and let $G$ be a finite group such that following conditions hold:
\begin{enumerate}
\item[(a)] There is a non-degenerated involutive  action of $G$ on $\widetilde{A}$ such that $\widetilde{A}^G=A$, where $\widetilde{A}^G$ is the algebra of $G$-invariants, i. e. $\widetilde{A}^G = \left\{\widetilde{a} \in \widetilde{A}~|~g\widetilde{a}=\widetilde{a};~\forall g\in G\right\}$; 
\item[(b)] There is a finite or countable set $I$ and an indexed by $I$ subset $\{a_{\iota}\}_{\iota \in I} \subset \widetilde{A}$ such that
\begin{equation}\label{can_nc}
\sum_{\iota \in I} a_{\iota}(ga^*_\iota)=\left\{
\begin{array}{c l}
    1_{M(\widetilde{A})} & g \in  G \text{ is trivial}\\
   0 & g \in  G \text{ is not trivial}
\end{array}\right.
\end{equation}
where the sum of the series means the strict convergence \cite{blackadar:ko}.

\end{enumerate}

Then $\pi$ is said to be a {\it finite  noncommutative covering projection}, $G$ is said to be the {\it covering transformation group}. Denote by $G\left(\widetilde{A}~|~A\right)=G$. The algebra $\widetilde{A}$ is said to be the {
\it covering algebra}, and $A$ is called the {\it base algebra} of the covering projection. A triple $\left(A, \widetilde{A}, G\right)$ is also  said to be a {\it finite  noncommutative covering projection}.
\end{defn}
\begin{rem}
	The Theorem \ref{pavlov_troisky_thm} gives an algebraic characterization of finitely listed covering projections of compact Hausdorff topological spaces and covering projections should not be regular, the Definition \ref{fin_def} concerns with regular covering projections of locally compact spaces.
\end{rem}
\begin{rem}
	The Definition  \ref{fin_def} is motivated by the Theorem \ref{comm_fin_thm}.
\end{rem}

\begin{defn}
	Let $\left(A, \widetilde{A}, G\right)$ be a noncommutative finite covering projection.  Algebra  $\widetilde{A}$  is a  countably generated  Hilbert $A$-module with a sesquilinear product given by
	\begin{equation}\label{fin_form_a}
	\left\langle a, b \right\rangle_{\widetilde{A}} = 
	 \sum_{g \in G} g(a^*b).
	\end{equation}
	We say that the structure of Hilbert $A$-module is {\it induced by the covering projection} $\left(A, \widetilde{A}, G\right)$. Henceforth we shall consider $\widetilde{A}$ as a right $A$-module, so we will write $\widetilde{A}_A$. 
\end{defn}

\begin{rem}
	In the commutative case  the product \eqref{fin_form_a}  coincides with the product given by the Theorem \ref{pavlov_troisky_thm}.
\end{rem}
\begin{rem}
From \eqref{can_nc} it follows that
\begin{equation}\label{can_nc_hilb}
1_{M\left(A\right)}=\sum_{\iota \in I} a_\iota \rangle\langle a_\iota~
\end{equation}
where above series is strictly convergent.
 \end{rem}

 \paragraph{}   The    algebra of compact operators $\mathcal{K}\left(\widetilde{A}_A\right)$ as well as $\widetilde{A}$ non-degenerately acts on $\widetilde{A}_A$. So we can formulate the following lemma.
    \begin{lem}\label{fin_compact_lem}
    If  $\left(A, \widetilde{A}, G\right)$ is a finite noncommutative covering projection then  $\widetilde{A} \subset\mathcal{K}\left(\widetilde{A}_A\right)$.
    \end{lem}
\begin{proof}
If $\widetilde{a}\in \widetilde{A}$ then from \eqref{can_nc_hilb} it follows that
$$
\widetilde{a} = \sum_{\iota \in I} a_\iota \rangle\langle a_\iota\cdot \widetilde{a} = \sum_{\iota \in I} a_\iota \rangle\langle a_\iota  \widetilde{a}.
$$
The series $\sum_{\iota \in I} a_\iota \rangle\langle a_\iota  \widetilde{a}$ is norm convergent because $\sum_{\iota \in I} a_\iota \rangle\langle a_\iota$ is strictly convergent. So $\widetilde{a}$ is compact, i.e. $\widetilde{a} \in \mathcal{K}\left(\widetilde{A}_A\right)$.
\end{proof}
\begin{rem}
The Lemma \ref{fin_compact_lem} may be regarded as a generalization of the Theorem \ref{galois_equiv_thm} because if $\widetilde{A}$ is unital then $\widetilde{A}$ is finitely generated projective right $A$-module, if and only if  $\widetilde{A} \subset\mathcal{K}\left(\widetilde{A}_A\right)$ (See \cite{clare_crisp_higson:adj_hilb}).
\end{rem}
\begin{exm}\label{circle_fin}{\it Finite covering projections of the circle $S^1$.} 
There is the universal covering projection $\widetilde{\pi}: \mathbb{R}\to S^1$. 
Let $\widetilde{\mathcal{U}}_1,\ \widetilde{\mathcal{U}}_2 \subset \mathbb{R}$ be such that
\begin{equation*}
\widetilde{\mathcal{U}}_1 = (-\pi -1/2,  1/2), \ \widetilde{\mathcal{U}}_2 = (-1/2, \pi +  1/2). 
\end{equation*} For any $i \in \{1,2\}$ the set $\mathcal{U}_i = \widetilde{\pi}(\widetilde{\mathcal{U}}_i) \subset S^1$ is open, connected and evenly covered. Since $S^1=\mathcal{U}_1\bigcup\mathcal{U}_2$ there is  a partition of unity $a_1, a_2$ dominated by $\{\mathcal{U}_i\}_{i\in \{1,2\}}$ \cite{munkres:topology}, i.e.  $a_i: S^1 \to [0,1]$ are such that
\begin{equation*}
\left\{
\begin{array}{c l}
    a_i(x)>0 & x \in \mathcal{U}_i \\
    a_i(x)=0 & x \notin \mathcal{U}_i
\end{array}\right.; \ i \in \{1,2\}.
\end{equation*}
and $a_1+a_2 = 1_{C(S^1)}$. 
If $e_1, e_2\in C\left(S^1\right)$ are given by
\begin{equation*}
e_i = \sqrt{a_i}; \ i=1,2;
\end{equation*}
then
\begin{equation*}
\left(e_1\right)^2+\left(e_2\right)^2= 1_{C_0(S^1)};~ e^*_1 = e_1;~e^*_2=e_2.
\end{equation*}
Let $\widetilde{e}_i \in C_c(\mathbb{R})$ be given by
\begin{equation}\label{tilde_e_eqn}
\widetilde{e}_i(\widetilde{x})=\left\{
\begin{array}{c l}
e_i(\widetilde{\pi}(\widetilde{x})) & \widetilde{x} \in \widetilde{\mathcal{U}}_i \\
0 & \widetilde{x} \notin \widetilde{\mathcal{U}}_i
\end{array}\right.; \ i \in \{1,2\}.
\end{equation}
Let $\widetilde{\pi}^n: \mathcal{X}_n\to S^1$ be an $n$-listed covering projection then $G(\mathcal{X}_n~|~S^1)\approx \mathbb{Z}_n$. It is well known that $\mathcal{X}_n\approx S^1$ but we use the  $\mathcal{X}_n$ notion for clarity. From $C\left(S^1 \right)= C\left(\mathcal{X}_n \right)^{\mathbb{Z}_n} $ it follows that the condition (a) of the Definition \ref{fin_def} holds. There is a sequence of covering projections $\mathbb{R}\xrightarrow{\pi^n}\mathcal{X}_n \to S^1$. If $\mathcal{U}^n_i = \pi^n\left(\widetilde{\mathcal{U}}_i\right)$ then $\mathcal{U}^n_i \bigcap g\mathcal{U}^n_i = \emptyset$ for any nontrivial $g \in G(\mathcal{X}_n, S^1)$. If $e^n_i \in C(\mathcal{X}_n)$ is given by
\begin{equation}\label{e_n_i}
e^n_i=\mathfrak{Desc}_{\pi^n}\left(\widetilde{e}_i \right) ; \ i \in \{1,2\}
\end{equation}
then
\begin{equation*}
\sum_{i \in \{1,2\}; \ g \in G(\mathcal{X}_n, S^1)} g\left(e^n_i\right)^2 = 1_{C_0(\mathcal{X}_n)}; 
\end{equation*}
and from $g \mathcal U^n_i \bigcap \mathcal U^n_i=\emptyset$ for any notrivial $g \in G(\mathcal{X}_n|S^1)$ it follows that
\begin{equation*} 
 e^n_i(ge^n_i) = 0; \ \text{ for any notrivial } g \in G(\mathcal{X}_n|S^1).
\end{equation*}
If $I_n = G(\mathcal{X}_n|S^1) \times \{1,2\}$
and
\begin{equation}\label{iota_def}
e^n_{\iota} = ge_i^n; \text{ where } \iota = (g,i) \in I_n
\end{equation}
then
\begin{equation}\label{circ_sum}
\sum_{\iota \in I_n} e_{\iota}(ge^*_\iota)=\sum_{\iota \in I_n} e_{\iota}(ge_\iota)=\left\{
\begin{array}{c l}
    1_{C_0(\mathcal{X}_n)} & g \in  G(\mathcal{X}_n~|~S^1) \text{ is trivial}\\
   0 & g \in  G(\mathcal{X}_n~|~S^1) \text{ is not trivial}
\end{array}\right..
\end{equation}

So a natural *-homomorphism $\pi:C(S^1)\to C(\mathcal{X}_n)$ satisfies the condition (b) of the Definition \ref{fin_def}.  So a triple $\left(C_0(S^1),C_0(\mathcal{X}_n), \mathbb{Z}_n\right)$ is a finite noncommutative covering projection.
\end{exm}
\begin{thm}\label{comm_pure_state_thm}\cite{murphy}
The state $\rho$ on an Abelian $C^*$-algebra is pure if and only if $\rho$ is a character.
\end{thm}
\begin{empt}\label{comm_pure_state_map_empt}
If $\mathcal X$ is a locally compact Hausdorff space then from Theorems \ref{gelfand-naimark}, \ref{comm_pure_state_thm} it follows that the set of pure states of $C_0\left( \mathcal X\right)$ coincides with  $\mathcal X$. If $f: C_0\left( \mathcal X\right) \to   C_0\left( \mathcal Y\right)$ is a *-homomorphism  then the natural map $ f^*:\mathcal Y \to  \mathcal X$ is given by
\begin{equation}\label{comm_state_map}
\rho_y \mapsto \tau_{f^*\left(y \right) }
\end{equation}
where $\rho_x : C_0\left( \mathcal Y\right) \to \C$  and  $\tau_{f^*\left(y \right)} : C_0\left( \mathcal X\right) \to \C$ are pure states given by
$$
\rho_y\left( b\right) =b\left(y \right); ~~ \tau_{f^*\left(y \right)}\left( a\right) =a\left(f^*\left( y\right)  \right) ; ~~ \forall b \in C_0\left(  \mathcal Y\right),~~\forall a \in C_0\left(  \mathcal X\right)
$$
such that the following condition holds
$$
\tau_{f^*\left(y \right)}\left(a \right)= \rho_y\left(f\left(a \right)  \right).  
$$
\end{empt}
\begin{thm}\label{ext_pure_state_thm}\cite{murphy}
Let $B$ be a $C^*$-subalgebra of a $C^*$-algebra $A$ and let $\rho$ be a pure state on $B$. Then there is a pure state $\rho'$ on $A$ extending $\rho$.
\end{thm}

\begin{thm}\label{comm_fin_thm}
Following conditions hold:
\begin{enumerate}
\item[(a)] If $\widetilde{p}:\widetilde{\mathcal X} \to \mathcal X$ is a finite  topological regular covering projection and both $\mathcal X$ and $\widetilde{\mathcal X}$ are locally compact Hausdorff spaces then $\left(C_0\left(\mathcal X \right)  , C_0\left(\widetilde{\mathcal X} \right), G\left(\widetilde{\mathcal X} ~|~ \mathcal X \right) \right)$ is a finite noncommutative covering projection;
\item[(b)] If $\mathcal X$ and  $\widetilde{\mathcal X}$ are locally compact Hausdorff spaces and $\left(C_0\left(\mathcal X \right)  , C_0\left(\widetilde{\mathcal X} \right), G \right)$ is a finite noncommutative covering projection then the natural continuous map   $\widetilde{p}:\widetilde{\mathcal X} \to \mathcal X$ is a  regular finitely listed covering projection such that $G\left(\widetilde{\mathcal X} ~|~ \mathcal X \right) \approx G$.  
\end{enumerate}
\end{thm}
\begin{proof} (a) There is an involutive continuous action of the covering  transformation group  $G = G\left(\widetilde{\mathcal X}~|~ \mathcal X\right)$ on $C_0\left(\widetilde{\mathcal X}\right)$ arising from the action of $G$ on $\widetilde{\mathcal X}$,  and $C_0\left(\mathcal X\right)=C_0\left(\widetilde{\mathcal X}\right)^G$, i.e. condition (a) of the Definition \ref{fin_def} holds. Let 
$$
1_{C_b\left(\widetilde{\mathcal{X}}\right)}= 1_{M\left(C_0\left(\widetilde{\mathcal{X}}\right)\right)}= \sum_{(g,\iota) \in G\left(\widetilde{\mathcal{X}}|\mathcal{X}\right) \times I}g \widetilde{a}_\iota
$$ be a covering partition of unity
(See Definition \ref{cov_cov_part_unity_defn}). If $g \in G$ and $\widetilde{e}_{\left(g,\iota\right)}\in C_0\left(\widetilde{\mathcal X}\right)$ is given by
 \begin{equation*}
 \widetilde{e}_{\left(g,\iota\right)} = \sqrt{g \widetilde{a}_\iota}.
 \end{equation*}
then from \eqref{covering_partition_eqn} and \eqref{zero_ge_iota_eqn} it follows that
 \begin{equation*}
 \sum_{\left(g',\iota\right) \in G \times I} \widetilde{e}_{\left(g',\iota\right)}\left(g\widetilde{e}^*_{\left(g',\iota\right)}\right)=\sum_{\left(g',\iota\right) \in G \times I} \widetilde{e}_{\left(g',\iota\right)}\left(g\widetilde{e}_{\left(g',\iota\right)}\right)=\left\{
 \begin{array}{c l}
     1_{M\left(C_0\left(\widetilde{\mathcal{X}}\right)\right)} & g \in  G \text{ is trivial}\\
    0 & g \in  G \text{ is not trivial}
 \end{array}\right..
 \end{equation*}
 So condition (b) of the Definition \ref{fin_def} holds, whence the triple
  $\left(C_0\left(\mathcal{X}\right),C_0\left(\widetilde{\mathcal{X}}\right), G\left(\widetilde{\mathcal X}~|~ \mathcal X\right)\right)$ is a finite noncommutative covering projection.
  \newline
  \newline
  (b) The action of $G$ on $C_0\left( \widetilde{\mathcal X}\right) $ induces the action of $G$ on $\widetilde{\mathcal X}$. If $\widetilde{x} \in \widetilde{\mathcal X}$ is such that $g \widetilde{x} = \widetilde{x}$ for nontrivial $g \in G$ then from \eqref{can_nc} it follows that
  $$
  1=1_{M\left(C_0\left( \widetilde{\mathcal X}\right)\right)}\left(\widetilde{x}\right)= \sum_{\iota \in I} a_{\iota}\left(\widetilde{x}\right)a^*_\iota\left(\widetilde{x}\right)=\sum_{\iota \in I} a_{\iota}\left(\widetilde{x}\right)a^*_\iota\left(g\widetilde{x}\right)=\left(\sum_{\iota \in I} a_{\iota}\left(ga^*_\iota\right)\right)\left(\widetilde{x}\right) = 0.
  $$
  From this contradiction it follows that the  action of $G$ does not have fixed points. 
  From \ref{g_cov_thm} and \ref{fin_prop_disc_empt}  it follows that $\widetilde{\mathcal X}\to\widetilde{\mathcal X}/G$ is a regular covering projection. 
From \ref{comm_pure_state_map_empt} it follows that the natural continuous map $\widetilde{p}:\widetilde{\mathcal X}\to \mathcal X$ is be given by  \eqref{comm_state_map}, i.e. if $\widetilde{\rho}_{\widetilde{x}} : C_0\left(\widetilde{\mathcal X}\right) \to \C$  and  $\rho_{\widetilde{p}\left(\widetilde{x} \right)} : C_0\left( \mathcal X\right) \to \C$ are pure states given by
$$
\widetilde{\rho}_{\widetilde{x}}\left( \widetilde{a}\right) =\widetilde{a}\left(\widetilde{x} \right); ~~ \rho_{\widetilde{p}\left(\widetilde{x} \right)}\left( a\right) =a\left(\widetilde{p}\left(\widetilde{x}  \right)\right)  
$$
then following condition holds
\begin{equation}\tag{*}
\rho_{\widetilde{p}\left(\widetilde{x}\right) }\left(a \right)= \widetilde{\rho}_{\widetilde{x}}\left( a\right); ~~ \forall a \in C\left(  \mathcal X\right).
\end{equation}
 From the Theorem \ref{ext_pure_state_thm} it follows that for any pure state $\rho_x: C_0\left(\mathcal X \right) \to \C$ there is a pure state $\widetilde{\rho}_{\widetilde{x}}: C_0\left(\widetilde{\mathcal X} \right) \to \C$ such that
 \begin{equation}\tag{**}
 \rho_{x}\left(a \right)= \widetilde{\rho}_{\widetilde{x}}\left( a\right); ~~ \forall a \in C\left(  \mathcal X\right).
 \end{equation}
 From (*) and (**) it follows that for any $x\in \mathcal X$ there is $\widetilde{x}\in \widetilde{\mathcal X}$ such that $x = \widetilde{p}\left( \widetilde{x}\right)$, i.e. the map $\widetilde{p}$ is surjective. If $a \in C_0\left(\mathcal X \right)$ then $ga = a$ for any $g \in G$, it follows that
 $$
 \rho_{\widetilde{p}\left(\widetilde{x} \right)}\left( a\right)= \rho_{\widetilde{p}\left(\widetilde{x} \right)}\left( g a\right)= \rho_{\widetilde{p}\left(g\widetilde{x} \right)}\left( a\right),
 $$ 
 i.e.
 \begin{equation}\tag{***}
\widetilde{p}\left(\widetilde{x} \right) = \widetilde{p}\left(g\widetilde{x} \right); ~~ \forall g \in G,~~ \forall \widetilde{x} \in \widetilde{\mathcal X}
 \end{equation}
 From the equation (***) it follows that the surjective map $\widetilde{p}$ induces the surjective map $p:\widetilde{\mathcal X}/G \to \mathcal{X}$. If $p$ is not injective then there are $x_1, x_2 \in \widetilde{\mathcal X}/G$ are such that $x_1\neq x_2$ and $p\left(x_1 \right)=p\left(x_2 \right) = x$. There are $\widetilde{x}_1, \widetilde{x}_2 \in  \widetilde{\mathcal X}$ mapped onto $x_1,x_2$. From $x_1\neq x_2$ it follows that $G\widetilde{x}_1 \bigcap G\widetilde{x}_2 = \emptyset$. Since $G$ is finite and $\widetilde{\mathcal X}$ is Hausdorff there are open sets $\widetilde{\mathcal U}_1, \widetilde{\mathcal U}_2$ such that $\widetilde{x}_j \in \widetilde{\mathcal U}_j$ for any $j=1,2$ and 
 \begin{equation*}\tag{****}
 g' \widetilde{\mathcal U}_j\bigcap g'' \widetilde{\mathcal U}_k = \emptyset \text{ for any  } \left(g',j \right)\neq \left(g'',k \right) \in G \times \{1,2\}.
 \end{equation*}
 
 For any $j = 1,2$ there is a function $\widetilde{a}_j \in C_0\left( \widetilde{\mathcal X}\right) $ such that
  $a_j\left(\widetilde{x}_j \right) = 1$ and 
  $\supp a_j \subset \widetilde{\mathcal U}_j $.
  If $a_j =  \sum_{g \in G} g \widetilde{a}_j$ then $a_j \in C_0\left(\mathcal X \right)$ and $a_j\left( x\right)= 1$, whence $\left(a_1a_2 \right)\left(x \right)=1$  . However from (****) it follows that  $\left(a_1a_2 \right)\left(x \right)=0$. We have a contradiction which proves that the map $p$ is injective.  $p:\widetilde{\mathcal X}/G \to \mathcal{X}$ is a homeomorphism because $p$ is an injective and surjective map. 
 \end{proof}

\begin{defn}
A ring is said to be {\it irreducible } if it is not a direct sum of more than one nontrivial ring. A finite covering projection $(A, \widetilde{A}, G)$ is said to be {\it irreducible} if both $A$ and $\widetilde{A}$ are irreducible. Otherwise  $(A, \widetilde{A}, G)$ is said to be {\it reducible}.
\end{defn}

\paragraph{} Following definition contains a manual composition of noncommutative covering projections. 
\begin{defn}\label{comp_defn}
Let
\begin{equation*}
A=A_0 \xrightarrow{\pi^1} A_1 \xrightarrow{\pi^2} ... \xrightarrow{\pi^n} A_n \xrightarrow{\pi^{n+1}} ...
\end{equation*}
be a finite or countable sequence of $C^*$-algebras and finite noncommutative covering projections. The sequence is said to be {\it composable} if following conditions hold:
\begin{enumerate}
\item[(a)] Any composition $\pi^{n_1}\circ ...\circ\pi^{n_0+1}\circ\pi^{n_0}:A_{n_0}\to A_{n_1}$ corresponds to the noncommutative covering projection $\left(A_{n_0}, A_{n_1}, G\left(A_{n_1}~|~A_{n_0}\right)\right)$;
\item[(b)] If $k < l < m$ then $G\left( A_m~|~A_k\right)A_l = A_l$ (Action of $G\left( A_m~|~A_k\right)$ on $A_l$ means that $G\left( A_m~|~A_k\right)$ acts on $A_m$, so $G\left( A_m~|~A_k\right)$ acts on $A_l$ since $A_l$ a subalgebra of $A_m$);
\item[(c)] If $k < l < m$ are nonegative integers then there is the natural exact sequence of covering transformation groups
\begin{equation*}
 \{e\}\to G\left(A_{m}~|~A_{l}\right) \xrightarrow{\iota} G\left(A_{m}~|~A_{k}\right)\xrightarrow{\pi}G\left(A_{l}~|~A_{k}\right)\to\{e\}
 \end{equation*}
 where the existence of the homomorphism $G\left(A_{m}~|~A_{k}\right)\xrightarrow{\pi}G\left(A_{l}~|~A_{k}\right)$ follows from (b).

\end{enumerate}

\end{defn}
\begin{lem}

If
  $
 \mathcal{X} = \mathcal{X}_0 \xleftarrow{}... \xleftarrow{} \mathcal{X}_n \xleftarrow{} ... 
 $
is  a topological  finite covering sequence such that  $\mathcal{X}_n$ is a second-countable locally compact Hausdorff space for any $n \in \N^0$ then 
\begin{equation*}
C_0\left(\mathcal X\right)=C_0\left(\mathcal X_0\right) \xrightarrow{\pi_1} C_0\left(\mathcal X_1\right) \xrightarrow{} ... \xrightarrow{\pi_n} C_0\left(\mathcal X_n\right) \xrightarrow{} ...
\end{equation*}
is a composable sequence of finite noncommutative covering projections.
\end{lem}
\begin{proof}
		We should check conditions (a), (b), (c) of the Definition \ref{comp_defn}.
		\newline (a) Let $\pi^{n_1}\circ ...\circ\pi^{n_0+1}\circ\pi^{n_0}:C\left(\mathcal X_{n_0}\right)\to C\left(\mathcal X_{n_1}\right)$ be any composition, and $p = n_0,~q=n_1$. Any composition of regular covering projections is a regular covering projection. It follows from  the Theorem \ref{comm_fin_thm} it follows that for any $p < q$ there is a finite noncommutative covering projection $C\left(\mathcal X_{p}\right) \to C\left(\mathcal X_{q}\right)$ such that  $$G\left(C\left(\mathcal X_q \right)~|~C\left(\mathcal X_p \right)\right) \approx G\left(\mathcal X_q ~|~\mathcal X_p \right).$$ 
		So $\pi_{n_1}\circ ...\circ\pi_{n_0+1}\circ\pi_{n_0}$ corresponds to the finite noncommutative covering projection $C\left(\mathcal X_{p}\right) \to C\left(\mathcal X_{q}\right)$.
	\newline 
	(b) If $k < l < m$ then from the Definition \ref{top_sec_defn} it follows that the sequence
	\begin{equation*}\tag{*}
	\{e\}\to	G\left(\mathcal X_m~|~\mathcal X_l\right) \to 	G\left(\mathcal X_m~|~\mathcal X_k\right)\to 	G\left(\mathcal X_l~|~\mathcal X_k\right)\to \{e\}
	\end{equation*}
	is exact. Clearly $G\left(\mathcal X_m~|~\mathcal X_l\right)$ trivially acts on $C\left(\mathcal X_l \right)$, so from (*) it follows that
	$$
		G\left(C\left(\mathcal X_m \right)~|~C\left(\mathcal X_k \right)\right)C\left(\mathcal X_l \right)=G\left(C\left(\mathcal X_l \right)~|~C\left(\mathcal X_k \right)\right)C\left(\mathcal X_l \right)=C\left(\mathcal X_l \right).
	$$
\newline
(c)	
	The sequence
	\begin{equation*}\tag{**}
	\begin{split}
	\{e\}\to	G\left(C_0\left( \mathcal X_m\right) ~|~C_0\left( \mathcal X_l\right) \right) \to 	G\left(C_0\left( \mathcal X_m\right) ~|~C_0\left( \mathcal X_k\right) \right)\to \\ \to 	G\left(C_0\left( \mathcal X_l\right) ~|~C_0\left( \mathcal X_k\right)\right) \to \{e\}
	\end{split}
\end{equation*}
is equivalent to (*). From  exactness of (*) 		it follows that the sequence (**) is exact.

\end{proof}

\section{Noncommutative inverse limits}\label{bas_constr}
\paragraph{}
This section is concerned with a noncommutative generalization of the described in the Section \ref{inf_to} construction. 
\begin{defn}
The composable sequence 
\begin{equation*}
\mathfrak{S} =\left\{ A =A_0 \xrightarrow{\pi^1} A_1 \xrightarrow{\pi^2} ... \xrightarrow{\pi^n} A_n \xrightarrow{\pi^{n+1}} ...\right\}
\end{equation*}
of irreducible separable  $C^*$-algebras and  finite noncommutative covering projections 
	is said to be an \textit{(algebraical)  finite covering sequence}.
		For any finite covering sequence we will use the notation $\mathfrak{S} \in \mathfrak{FinAlg}$.
\end{defn}
\begin{defn}\label{regular_defn}
Let $\mathfrak{S} =\left\{ A =A_0 \xrightarrow{\pi^1} A_1 \xrightarrow{\pi^2} ... \xrightarrow{\pi^n} A_n \xrightarrow{\pi^{n+1}} ...\right\}$ be an algebraical  finite covering sequence. Let $\widehat{A} = \varinjlim A_n$ be the $C^*$-inductive limit \cite{murphy,takeda:inductive}, and let $\widehat{G}\subset \Aut\left( \widehat{A}\right) $ be such that $A =\widehat{A}^{\widehat{G}}=\left\{\widehat{a} \in \widehat{A}~|~ \widehat{G}\widehat{a}=\left\{\widehat{a}\right\}\right\}$.  The sequence $\mathfrak{S}$ is said to be \textit{regular} if for any $n \in \mathbb{N}$ following conditions hold:
\begin{itemize}
\item[(a)] $\widehat{G}A_n = A_n$;
\item[(b)] The given by $\widehat{g} \mapsto \widehat{g}|_{A_n}$ group homomorphism $h_n: \widehat{G} \to G\left( A_n~|~A \right)$  is surjective.
\end{itemize}
The group $\widehat{G}$ is said to be the {\it  covering transformation group} of $\mathfrak{S}$. We will use the notation $G\left(\widehat{A}~|~A \right) \stackrel{\text{def}}{=}\widehat{G}$.
\end{defn}
 
\begin{lem}
If the sequence $\mathfrak{S}$ is regular then $\bigcap_{n \in \mathbb{N}} \ker\left( G\left(\widehat{A}~|~A \right) \to G\left(A_n~|~A \right)\right)$ is the trivial group.  
\end{lem}
\begin{proof}
If $\widehat{g}\in\bigcap_{n \in \mathbb{N}} \ker\left( G\left(\varinjlim A_n~|~A \right) \to G\left(A_n~|~A \right)\right)$ then $\widehat{g}$ induces the trivial automorphism of $A_n$ for any $n \in \mathbb{N}$. Hence $\widehat{g}$ is the trivial. 
\end{proof}
Following definition is motivated by  the construction \ref{comm_alg_constr_susub}.
\begin{defn}\label{special_el_defn}
Let $\mathfrak{S} =\left\{ A =A_0 \xrightarrow{\pi^1} A_1 \xrightarrow{\pi^2} ... \xrightarrow{\pi^n} A_n \xrightarrow{\pi^{n+1}} ...\right\}$ be a regular  algebraical  finite covering sequence. Let $\widehat{A} = \varinjlim A_n$  be the $C^*$-inductive limit, and let $\widehat{G}= G\left(\widehat{A}~|~A \right) $ be the covering transformation group of $\mathfrak{S}$. Let $\overline{A}$ be a $C^*$-algebra with action of $\widehat{G}$ and the injective $\widehat{G}$- equivariant inclusion   $\varphi:\widehat{A} \xrightarrow{\subset} \overline{A}$, i.e. $\varphi\left(\widehat{g}\widehat{a} \right) = \widehat{g}\varphi\left(\widehat{a} \right)$ for any $\widehat{g}\in\widehat{G}$ and  $\widehat{a}\in\widehat{A}$. A positive element  $\overline{a}  \in \overline{A}$ is said to be \textit{special} if following conditions hold:
	\begin{enumerate}
		\item[(a)]  If $n \in \mathbb{N}^0$ and $J_n = \ker\left( \widehat{G} \to  G\left( A_n~|~A \right)\right) $  then following two  series 
		\begin{equation*}
		a_n = \sum_{g \in J_n} g \overline{a};~	b_n = \sum_{g \in J_n} g \overline{a}^2;
		\end{equation*}
		are strongly convergent and the sums lie in $A_n$, i.e. $a_n, b_n \in A_n$;
		\item[(b)] For any $\eps > 0$ there is $N \in \mathbb{N}$ such that if $n \ge N$ then
		\begin{equation}\label{spec_square_eqn}
\left\|a_n^2 - b_n\right\|	< \eps.
		\end{equation}
	\end{enumerate}
\end{defn}
\begin{lem}\label{special_in_w_lem} Let $\mathfrak{S} =\left\{ A =A_0 \xrightarrow{\pi^1} A_1 \xrightarrow{\pi^2} ... \xrightarrow{\pi^n} A_n \xrightarrow{\pi^{n+1}} ...\right\}$ be a regular  algebraical  finite covering sequence. 
If $\widehat{A}''$ is the enveloping $W^*$-algebra of $\widehat{A}= \varinjlim A_n$ and $\overline{a}\in \overline{A}$ is a special element then $\overline{a} \in \widehat{A}''$.
\end{lem}
\begin{proof}
Use notation of the Definition \ref{special_el_defn}. There is a decreasing  sequence of positive elements  $\left\{a_n  \in \widehat{A}\right\}_{n \in \N}$ such that $a_n = \sum_{g \in J_n} g \overline{a}$. From the Lemma \ref{increasing_convergent_w} it follows that the sequence is strongly convergent, so $\lim_{n \to \infty} a_n \in \widehat{A}''$. However  $\lim_{n \to \infty} a_n = \overline{a}$ is sense of the strong convergence, whence $\overline{a} \in \widehat{A}''$. 
\end{proof}
\begin{rem}
Enveloping $W^*$-algebras are associated with  Borel sets, so the Lemma   \ref{special_in_w_lem} can be regarded as a noncommutative  analog  of the Lemma \ref{top_equ_borel_set_lem}.
\end{rem}

\begin{defn}\label{main_defn}
Let $\mathfrak{S} =\left\{ A =A_0 \xrightarrow{\pi^1} A_1 \xrightarrow{\pi^2} ... \xrightarrow{\pi^n} A_n \xrightarrow{\pi^{n+1}} ...\right\}$ be a regular  algebraical  finite covering sequence. Let $\widetilde{A}$ be the $C^*$-norm completion of an algebra generated by $\mathbb{C}$-linear span of special elements. We say that $\widetilde{A}$ is the {\it inverse noncommutative limit} of $\mathfrak{S}$. We will use following notation  $ \varprojlim  \mathfrak{S}\stackrel{\text{def}}{=}\widetilde{A}$.
\end{defn}
\begin{rem}
	Above definition is motivated by the Definition \ref{top_topological_inv_lim_defn} and the Theorem \ref{comm_main_thm}.
\end{rem}

\begin{lem}\label{ext_group_action_lem}
Use notation of the Definition \ref{main_defn}.
If $\mathfrak{S}$ is regular $\widehat A = \varinjlim A_n$ and  $\widehat{G} = G\left(\widehat A~|~A \right)$ then there is the action of $\widehat{G}$ on the inverse limit $\widetilde{A} = \varprojlim  \mathfrak{S}$. 

\end{lem}
\begin{proof}
$\widehat{A}$ is strongly dense in $\widehat{A}''$, it follows that action $\widehat{G} \times \widehat{A} \to \widehat{A}$ can be uniquely extended to the action $\widehat{G} \times \widehat{A}'' \to \widehat{A}''$. If $\overline{a}$ is a special element then from the Lemma \ref{special_in_w_lem} it follows that $\overline{a} \in  \widehat{A}''$. Let 
		\begin{equation*}
		a_n = \sum_{g \in J_n} g \overline{a};~	b_n = \sum_{g \in J_n} g \overline{a}^2.
		\end{equation*}
		If $\eps > 0$ then from (b) of the Definition \ref{special_el_defn}  it follows that there is $N \in \N$ such that for any $n \ge N$ following condition holds
			\begin{equation*}\tag{*}
			\left\|a_n^2 - b_n\right\|	< \eps.
			\end{equation*}
If $\widehat{g} \in \widehat{G}$ is any element,  $h_n:\widehat{G} \to G\left(A_n~|~A \right)$ is the homomorphism from the Definition \ref{regular_defn}, $J_n = \ker h_n$ then $\left\{\widehat{g}g | g \in J_n \right\}=  \left\{g\widehat{g} | g \in J_n \right\}= h_n^{-1}\circ h_n	\left( \widehat{g}\right)$. If $\overline{a}' = \widehat{g}\overline{a}$ then
		\begin{equation*}
	\sum_{g \in J_n} g \overline{a}'=\sum_{g \in J_n} g \widehat{g}\overline{a}=\sum_{g \in J_n}  \widehat{g}g\overline{a} = \widehat{g}\sum_{g \in J_n} g \overline{a}=h_n\left( \widehat{g}\right) a_n;~
		\end{equation*}
		\begin{equation*}
	\sum_{g \in J_n} g \overline{a}'^2=\sum_{g \in J_n} g \widehat{g}\overline{a}^2=\sum_{g \in J_n}  \widehat{g}g\overline{a}^2 = \widehat{g}\sum_{g \in J_n} g \overline{a}^2=h_n\left( \widehat{g}\right) b_n;
		\end{equation*}
$$
\left\|\left( \sum_{g \in J_n} g \overline{a}'\right)^2 - \sum_{g \in J_n} g \overline{a}'^2 \right\|= \left\|\left( h_n\left( \widehat{g}\right) a_n\right)^2 - h_n\left( \widehat{g}\right) b_n \right\|.
$$ 
Since $h_n\left( \widehat{g}\right) $ is *-automorphism following condition holds
$$
\left\|\left( h_n\left( \widehat{g}\right) a_n\right)^2 - h_n\left( \widehat{g}\right) b_n \right\| = \left\| h_n\left( \widehat{g}\right) \left( a_n^2  - b_n\right)  \right\|=	\left\|a_n^2 - b_n\right\|	< \eps~, 
$$
whence $\overline{a}' = \widehat{g}\overline{a}$ is special, i.e. $\widehat{g}$ maps special elements to special ones. Since a linear span of special elements is dense in $\widetilde{A}$ it follows that $\widehat{G}\widetilde{A} = \widetilde{A}$, i.e. $\widehat{G}$ naturally acts on $\widetilde{A}$. 
\end{proof}
\begin{defn}\label{main_defn_group_defn}
Use notation of the Definition \ref{main_defn} and suppose that $\mathfrak{S}$ is regular. The group $\widehat{G}$ is said to be the \textit{covering transformation group} of the inverse limit $\widetilde{A}=\varprojlim \mathfrak{S}$. We will use the notation   $G\left(\widetilde{A}~|~ A\right) \stackrel{\text{def}}{=}G\left(\widehat{A}~|~A \right)$. As well as the algebra $\widetilde{A}$ the triple
 $\left(A, \widetilde{A}, \widehat{G}\right)$ (or $\left(A, \widetilde{A}, G\left(\widetilde{A}~|~ A\right)\right)$) is said to be the  {\it infinite noncommutative covering projection} of $\mathfrak{S}$.
\end{defn}
\begin{rem}
This definition is motivated by the Lemma \ref{top_biject_lem} and the Lemma \ref{ext_group_action_lem}. 
\end{rem}

\section{Commutative case}\label{comm_case_sec}
\paragraph*{} This section supplies a purely algebraic  analog of the topological construction given by the Subsection \ref{inf_to}. 
We need the following theorem.
\begin{thm}\label{direct_lim_state_thm}\cite{takeda:inductive}
	If a $C^*$-algebra $A$ is a $C^*$-inductive limit of $A_\ga$ ($\ga \in \Ga$), the
	state space $\Om$ of A is homeomorphic to the projective limit of the state spaces $\Om_\ga$ of $A_\ga$.
\end{thm}
	
\begin{cor}\label{direct_lim_state_cor}\cite{takeda:inductive}
	If a commutative $C^*$-algebra $A$ is a $C^*$-inductive limit of the
	commutative  $C^*$-algebras $A_\ga$ ($\ga \in \Ga$), the spectrum $\mathcal X$ of $A$ is the projective limit of spectrums $\mathcal X_\ga$ of $A_\ga$ ($\ga \in \Ga$).
\end{cor}
If $\mathfrak{S}_{\mathcal X} = \left\{\mathcal{X} = \mathcal{X}_0 \xleftarrow{}... \xleftarrow{} \mathcal{X}_n \xleftarrow{} ...\right\} \in \mathfrak{FinTop}$ then from the Theorem \ref{comm_fin_thm} it follows that
$\mathfrak{S}_{C_0\left(\mathcal{X}\right)}=
\left\{C_0(\mathcal{X})=C_0(\mathcal{X}_0)\to ... \to C_0(\mathcal{X}_n) \to ...\right\} \in \mathfrak{FinAlg}$.

\begin{lem}\label{comm_reg_lem}
If $\mathfrak{S}_{\mathcal X} = \left\{\mathcal{X} = \mathcal{X}_0 \xleftarrow{}... \xleftarrow{} \mathcal{X}_n \xleftarrow{} ...\right\} \in \mathfrak{FinTop}$ then the sequence
$\mathfrak{S}_{C_0\left(\mathcal{X}\right)}=
\left\{C_0(\mathcal{X})=C_0(\mathcal{X}_0)\to ... \to C_0(\mathcal{X}_n) \to ...\right\} \in \mathfrak{FinAlg}$ is regular.
\end{lem}
\begin{proof}
We should proof (a) and (b) of the Definition \ref{regular_defn}. 
\newline
(a)
If $\widehat{\mathcal X} = \varinjlim {\mathcal X}_n$ then from the Corollary \ref{direct_lim_state_cor} it follows that $\varinjlim C_0\left(\mathcal X_n \right)= C_0\left( \widehat{\mathcal X}\right)$. It follows from the Theorem \ref{gelfand-naimark}  that there is the natural *-isomorphism $G\left(\widehat{\mathcal X}~|~ {\mathcal X} \right)\approx G\left(C_0\left( \widehat{\mathcal X}\right) ~|~ C_0\left( {\mathcal X}\right)  \right)$. If $a \in C_0\left(\mathcal X_n \right)$ and $\widehat{g} \in G\left(\widehat{\mathcal X}~|~ {\mathcal X} \right)$  then from the Lemma \ref{top_surj_group_lem} it follows that
\begin{equation}\tag{*}
\widehat{g} a = h_n(\widehat{g}) a
\end{equation}
where 
\begin{equation}\tag{**}
h_n:G\left(\widehat{\mathcal X}~|~ {\mathcal X} \right)\to G\left(\mathcal X_n~|~ {\mathcal X} \right)
\end{equation}
is a group homomorphism.
From (*) and $h_n(\widehat{g}) a \in C\left(\mathcal X_n \right)$ it follows that  $$G\left(C_0\left( \widehat{\mathcal X}\right)~|~C_0\left({\mathcal X}\right)\right)  C_0\left( \mathcal X_n\right)=  C_0\left( \mathcal X_n\right).$$
\newline
(b) From the Lemma \ref{top_surj_group_lem} it follows that homomorphism (**) is surjective. However from the Theorem \ref{gelfand-naimark} it follows the homomorphism (**) is equivalent to $h'_n: G\left(C_0\left( \widehat{\mathcal X}\right)~|~C_0\left({\mathcal X}\right)\right) \to G\left(C_0\left( {\mathcal X}_n\right)~|~C_0\left({\mathcal X}\right)\right)$. So $h'_n$ is surjective. 
\end{proof}
\begin{lem}\label{comm_l_infty_lem}
If $\mathfrak{S}_{\mathcal X} = \left\{\mathcal{X} = \mathcal{X}_0 \xleftarrow{}... \xleftarrow{} \mathcal{X}_n \xleftarrow{} ...\right\} \in \mathfrak{FinTop}$, $\widetilde{\mathcal X} = \varprojlim \mathfrak{S}_{\mathcal X}$ and 
$\mathfrak{S}_{C_0\left(\mathcal{X}\right)}=
\left\{C_0(\mathcal{X})=C_0(\mathcal{X}_0)\to ... \to C_0(\mathcal{X}_n) \to ...\right\} \in \mathfrak{FinAlg}$, then there is the natural $G\left(\widetilde{\mathcal X} | \mathcal X\right)$ equivariant inclusion 
$$\left( \varinjlim C_0\left(\mathcal X_n\right)\right)''=\left( C_0\left(\varprojlim\mathcal X_n\right)\right)'' = L^\infty\left(\varprojlim \mathcal X_n \right) \xrightarrow{\subset}L^\infty\left(\widetilde{\mathcal X} \right).
$$
\end{lem} 
\begin{proof}

From the Corollary \ref{direct_lim_state_cor} it follows that $\varinjlim C_0\left(\mathcal X_n\right) = C_0\left(\widehat{\mathcal X}\right)$ where $\widehat{\mathcal X} = \varprojlim \mathcal X_n$ is the inverse limit of topological  spaces. If $\widetilde{\mathcal X}$ is the inverse limit of $\mathfrak{S}_{\mathcal X}$ then from the Lemma \ref{top_biject_lem} it follows that the continuous map $\widetilde{\mathcal X}\to \widehat{\mathcal X}$ is a bijection.  So the natural *-homomorphism $\varphi:C_0\left(\widehat{\mathcal X}\right) \to C_b\left(\widetilde{\mathcal X}\right)$ is injective and  $G\left(\widehat{\mathcal X}, \mathcal X\right)$-equivariant. The *-homomorphism $\varphi$ induces the  $G\left(\widehat{\mathcal X}, \mathcal X\right)$-equivariant inclusion $\left( \varinjlim C_0\left(\mathcal X_n\right)\right)'' \xrightarrow{\subset}\left(C_b\left((\widetilde{\mathcal X} \right)  \right)''= L^\infty\left(\widetilde{\mathcal X} \right)$
of enveloping $W^*$-algebras.
\end{proof}
\begin{lem}\label{comm_c_is_spec_lem}
On the notation of the Lemma \ref{comm_l_infty_lem} any positive element $\widetilde{a} \in C_c\left(\widetilde{\mathcal X} \right)$ is special. 
\end{lem}
\begin{proof}
	Denote by $\widetilde{\pi}^n: \widetilde{\mathcal{X}} \to \mathcal X_n$. 
If $K = \supp \widetilde{a}$ then $K$ can be covered by special open sets. Since $K$ is compact there is a finite family $\widetilde{\mathcal U}_1,..., \widetilde{\mathcal U}_m$ of special open sets such that $K \subset \widetilde{\mathcal U}_1\bigcup, \dots, \bigcup\widetilde{\mathcal U}_m$.  If $~\widetilde{\mathcal U}_j \bigcap \widetilde{\mathcal U}_k = \emptyset$  and $ \widetilde{\pi}^0\left(\widetilde{\mathcal U}_j \right)\bigcap\widetilde{\pi}^0\left(\widetilde{\mathcal U}_k \right) \neq \emptyset$ then for any  $ x \in \widetilde{\pi}^0\left(\widetilde{\mathcal U}_j \right)\bigcap\widetilde{\pi}^0\left(\widetilde{\mathcal U}_k \right)$  there are  $ x' \in \widetilde{\mathcal U}_j$ and $ x'' \in \widetilde{\mathcal U}_k$ such that $\widetilde{\pi}^0\left(x'\right)=\widetilde{\pi}^0\left(x''\right)=x$. Since $x' \neq x''$ there is $n_{jk}\in \mathbb{N}$ such that $\widetilde{\pi}^{n_{jk}}\left(x'\right)\neq\widetilde{\pi}^{n_{jk}}\left(x''\right)$, whence 
$
\widetilde{\pi}^{n_{jk}}\left(\widetilde{\mathcal U}_j\right)\bigcap\widetilde{\pi}^{n_{jk}}\left(\widetilde{\mathcal U}_k\right)= \emptyset$. If $N = \max n_{jk}$ and $\widetilde{\mathcal U}= \widetilde{\mathcal U}_1\bigcup, \dots, \bigcup\widetilde{\mathcal U}_m$ then $\widetilde{\mathcal U}$ is mapped homemorphically onto $\widetilde{\pi}^N\left(\widetilde{\mathcal U}\right)$. For any $n \ge N$ let $a_n \in C_c\left(\mathcal X_n \right)$ be given by 
$
a_n = \mathfrak{Desc}_{ \widetilde{\pi}^n}\left( \widetilde{a}\right). 
$
Since for any $n > N$ the set $\widetilde{\mathcal U}$ is mapped homeomorphically on $\widetilde{\pi}^{n}\left( \widetilde{\mathcal U}\right)$ following conditions hold:

$$
\sum_{\widetilde{g}\in \ker\left(G\left( \widetilde{\mathcal X}|\mathcal X\right)\to G\left( \mathcal X_n|\mathcal X\right)  \right)  } \widetilde{g}~\widetilde{a} = a_n,
$$
$$
b_n =\sum_{\widetilde{g}\in \ker\left(G\left( \widetilde{\mathcal X}|\mathcal X\right)\to G\left( \mathcal X_n|\mathcal X\right)  \right)  } \widetilde{g}~\widetilde{a}^2 = a^2_n
$$
where sums of series mean strong limits. 
From the above equations it follows that 
$$
a_n,b_n \in C_0\left(\mathcal{X} \right);  
$$
$$
a_n^2= b_n; 
$$
i. e. $\widetilde{a}$ satisfies to the Definition \ref{special_el_defn}.
\end{proof}
\begin{cor}\label{comm_a_in_c0_cor}
If $\widetilde{A}$ is the inverse limit of $\mathfrak{S}_{C_0\left(\mathcal{X}\right)}$ then $\widetilde{A} \subset C_0\left(  \widetilde{\mathcal X}\right)$.
\end{cor} 
\begin{proof}
Follows from the Lemma \ref{comm_c_is_spec_lem} and the Definition \ref{c_c_def_1}.
\end{proof}
We need following fact.

\begin{fact}\label{closure_fact}
 If $f: \mathcal X \to \mathcal Y$ is a continuous map, $\mathcal X' \subset \mathcal X$ and $\mathcal Y' \subset \mathcal Y$ then $f\left(\overline{\mathcal{X}}'\right) \subset \overline{f\left(\mathcal{X}'\right)}$ and $f^{-1}\left(\overline{\mathcal{Y}}'\right) \subset \overline{f^{-1}\left(\mathcal{Y}'\right)}$, where the $\overline{~\cdot~} $ operation means the topological closure.
\end{fact}

\begin{lem}\label{comm_spec_cont_lem}

If $\mathcal X$ is a compact second-countable Hausdorff space then any special element $\widetilde{a}  \in  L^\infty\left(\widetilde{\mathcal X}\right)$ lies in $C_0\left(\widetilde{\mathcal X}\right)$.
\end{lem}
\begin{proof}
The proof of this lemma uses following notation
\newline
\begin{tabular}{|l|c|c|}
	\hline
 & Symbol & Comment (and/or) Meaning (and/or) Requirement\\
	\hline
	& & \\

1& $J_n = \ker\left(G\left( \widetilde{\mathcal X}|\mathcal X\right)\to G\left( \mathcal X_n|\mathcal X\right)\right) $&  \\
2&$J_n^0 = J_n \backslash \{e\}$ & $e \in G\left( \widetilde{\mathcal X}|\mathcal X\right)$ is  the trivial element \\

3&  $\widetilde{\pi}: \widetilde{\mathcal X} \to \mathcal X$, $\widetilde{\pi}^n: \widetilde{\mathcal X} \to \mathcal X_n$& Natural covering projections \\
4&	$a_n = \sum_{g \in J_n} g \widetilde{a};~ b_n = \sum_{g \in J_n} g \widetilde{a}^2$& From the Definition \ref{special_el_defn} it follows that\\
&&	 $a_n, b_n \in C\left(\mathcal X_n\right)$ \\
5&$\widetilde{\mathcal U} \subset \widetilde{\mathcal X}$ is an open connected set& The closure   $\overline{\widetilde{\mathcal U}}$ of $\widetilde{\mathcal U}$ is compact \\
& & and  homeomorphically mapped onto $\mathcal U = \widetilde{\pi}\left(\overline{\widetilde{\mathcal U}}\right)$ \\
6& $\widetilde{a}'= \widetilde{a}|_{\overline{\widetilde{\mathcal U}}}$ & The restriction of $\widetilde{a}$ on $\overline{\widetilde{\mathcal U}}$\\
7& $\widetilde{a}'_n,\widetilde{b}'_n \in C\left(\overline{\widetilde{\mathcal U}}\right)$  & The compact $\overline{\widetilde{\mathcal U}}$-lifts of $a_n|_{\widetilde{\pi}^n\left(\overline{\widetilde{\mathcal U}}\right)}$, $b_n|_{\widetilde{\pi}^n\left(\overline{\widetilde{\mathcal U}}\right)}$  (Definition \ref{lift_desc_defn})\\
8 & $\widetilde{a}'_g = \left(g \widetilde{a}\right)|_{\overline{\widetilde{U}}}$ & $\widetilde{a}'_g\in C\left(\overline{\widetilde{U}}\right)$ \\
	\hline
\end{tabular}
\newline
\newline
From the above notation it follows that
\begin{equation*}
\widetilde{a}'_n = \sum_{g \in J_n} \widetilde{a}'_g = \widetilde{a}' + \sum_{g \in J^0_n} \widetilde{a}'_g,
\end{equation*}
\begin{equation*}
\widetilde{b}'_n = \sum_{g \in J_n} \widetilde{a}'^2_g = \widetilde{a}'^2 + \sum_{g \in J^0_n} \widetilde{a}'^2_g
\end{equation*}
whence
\begin{equation*}
\widetilde{a}'^2_n- \widetilde{b}'_n = 2 \widetilde{a}'_n \sum_{g \in J_n^0} \widetilde{a}'_g + \sum_{g' \in J_n^0}~ \sum_{g'' \in J_n^0 \backslash \left\{g'\right\}} \widetilde{a}'_{g'} \widetilde{a}'_{g''}
\end{equation*}
and all members of the above sums are positive.
Let $ \varepsilon > 0$ be any number. From the Definition \ref{special_el_defn} it follows that there is $N \in \mathbb{N}$  such that for any $n \ge N$ 
	\begin{equation*}
 	\left\|a_n^2 - b_n\right\| < 2\varepsilon^2, 
 	\end{equation*}
 	whence
	\begin{equation}\tag{*}
 2\left\| \widetilde{a}'_n \sum_{g \in J_n^0} \widetilde{a}'_g\right\| \le \left\|2 \widetilde{a}'_n \sum_{g \in J_n^0} \widetilde{a}'_g +  \sum_{g' \in J_n^0}~ \sum_{g'' \in J_n^0 \backslash \left\{g'\right\}} \widetilde{a}'_{g'} \widetilde{a}'_{g''}  \right\| =	\left\| \widetilde{a}'^2_n  -  \widetilde{b}'_n \right\| \le 	\left\|a_n^2 - b_n\right\| < 2\varepsilon^2. 
 	\end{equation}
 
If $\widetilde{x}' \in \overline{\widetilde{U}}$ is such that $\widetilde{a}'_n\left(\widetilde{x}'\right) \ge \varepsilon$ then from (*) it follows that
	$$
 	\widetilde{a}'_n\left(\widetilde{x}'\right) \left( \sum_{g \in J_n^0} \widetilde{a}'_g\left(\widetilde{x}'\right)\right) < \varepsilon^2,
 	$$
\begin{equation*}
\widetilde{a}'_n\left(\widetilde{x}'\right)- \widetilde{a}'\left(\widetilde{x}'\right) = \sum_{g \in J^0_n} \widetilde{a}'_g\left(\widetilde{x}'\right) < \varepsilon. 
\end{equation*} 
If $\widetilde{x}'' \in \overline{\widetilde{U}}$ is such that $\widetilde{a}'_n\left(\widetilde{x}''\right) < \varepsilon$ then from $ 0 \le \widetilde{a}' \le \widetilde{a}'_n$ it follows that 
\begin{equation*}
\widetilde{a}'_n\left(\widetilde{x}''\right)- \widetilde{a}'\left(\widetilde{x}''\right) < \varepsilon, 
\end{equation*} 
whence
$$
\widetilde{a}'_n\left(\widetilde{x}\right)- \widetilde{a}'\left(\widetilde{x}\right) < \varepsilon
$$
for any $\widetilde{x} \in \overline{\widetilde{\mathcal U}}$, i.e.
\begin{equation*}
\left\|\widetilde{a}'_n- \widetilde{a}'\right\| < \varepsilon. 
\end{equation*}

So the sequence $\left\{\widetilde{a}'_n\right\} $ of continuous functions  uniformly converges  to $\widetilde{a}'$, thus $\widetilde{a}'$ is a continuous function. Any point $x \in \widetilde{\mathcal{X}}$ has a neighborhood $\widetilde{\mathcal{U}}$ such that the restriction  $\widetilde{a}_{\widetilde{\mathcal{U}}}$ is continuous, whence $\widetilde{a}$ is  continuous. If $\left\|\cdot\right\|_\infty$ is the $L^\infty$-norm then from $\left\|\widetilde{a}\right\|_\infty = \left\|\widetilde{a}\right\|$ it follows that $\widetilde{a}$ is bounded, i.e. $\widetilde{a}\in C_b\left( \widetilde{\mathcal X}\right)$. Since $\mathcal X$ is compact and for any $n \in \N$ the map $\mathcal X_n \to \mathcal X$ is a finite fold covering projection the space $\mathcal X_n$ is also compact. 
Let $\eps'> 0$, and let $f_{\eps'} \in C\left(\R \right)$ be given by 
\begin{equation*}
f_{\eps'}\left( t\right)  = \left\{
   	\begin{array}{c l}
   	0 & t \le \varepsilon'  \\
   	t - \varepsilon' & t > \varepsilon'
   	\end{array}\right..
   	\end{equation*}
Let $N' =  \mathbb{N}$ be such that for any $n \ge N'$ following condition holds
\begin{equation}\tag{**}
\left\|a_n^2 - b_n\right\| < 2 \varepsilon'^2.
\end{equation}
Let $\widetilde{a}_{\eps'}= f_{\eps'}\left( \widetilde{a}\right)\in C_b\left(\widetilde{\mathcal{X}} \right)$.  
If $\widetilde{K}_{\eps'}= \supp \widetilde{a}_{\eps'}=\supp f_{\eps'}\left(\widetilde{a}\right)$, then $\widetilde{K}_{\eps'} = \left\{\widetilde{x} \in \widetilde{\mathcal X} ~|~\widetilde{a}\left(\widetilde{x}\right)\ge \varepsilon'\right\}$. 
 If $n \ge N'$, $\pi^n: \widetilde{\mathcal{X}} \to \mathcal X_n$ and the restriction ${\pi^n}|_{\widetilde{K}_{\eps'}}$ is not injective then there are  points $\widetilde{x}_1, \widetilde{x}_2 \in \widetilde{K}_{\eps'}$ such that $\widetilde{x}_1\neq \widetilde{x}_2$ and $\widetilde{\pi}^n\left(\widetilde{x}_1\right)=\widetilde{\pi}^n\left(\widetilde{x}_2\right)=x$. Following condition holds
\begin{equation*}
\left(a_n\left(x\right)\right)^2 - b_n\left(x\right) \ge 2\widetilde{a}\left(\widetilde{x}_1\right)\widetilde{a}\left(\widetilde{x}_2\right) \ge 2 \varepsilon'^2.
\end{equation*}
Above equation contradicts with (**), whence the restriction $\widetilde{\pi}^n|_{\widetilde{K}_{\eps'}}$ is an invective map.  The closed set $K_{\eps'} = \supp f_{\eps'}\left( a_n\right) $ (resp. $\widetilde{K}_{\eps'} = \supp f_{\eps'}\left( \widetilde{a}\right) $) is the closure of the open set $\mathring{K}_{\eps'} = \left\{x \in \mathcal X_n ~|~ a_n(x)> \eps'\right\}$ (resp. $\mathring{\widetilde{K}}_{\eps'} = \left\{\widetilde{x} \in \widetilde{\mathcal X} ~|~ \widetilde{a}(\widetilde{x})> \eps'\right\}$).  From the Fact \ref{closure_fact} it follows that 
\begin{equation*}\tag{***}
\begin{split}
\widetilde{\pi}^n|_{\widetilde{K}_{\eps'}}\left(\widetilde{K}_{\eps'}\right)=\widetilde{\pi}^n|_{\widetilde{K}_{\eps'}}\left(\overline{\mathring{\widetilde{K}}}_{\eps'}\right)\subset \overline{\mathring{K}}_{\eps'}=K_{\eps'},
\\
\left(\widetilde{\pi}^n|_{\widetilde{K}_{\eps'}}\right)^{-1}\left(K_{\eps'}\right)=\left(\widetilde{\pi}^n|_{\widetilde{K}_{\eps'}}\right)^{-1}\left(\overline{\mathring{K}}_{\eps'}\right)\subset \overline{\mathring{\widetilde{K}}}_{\eps'}= \widetilde{K}_{\eps'}
\end{split}
\end{equation*}
where the $\overline{~\cdot~} $ operation means the topological closure.
It follows from (***) that $\widetilde{\pi}^n|_{\widetilde{K}_{\eps'}}\left(\widetilde{K}_{\eps'}\right)= K_{\eps'}$. It is known \cite{spanier:at} that any covering projection is an open map.  Any open bicontinuous map is a homeomorphism. So $\widetilde{\pi}^n|_{\widetilde{K}_{\eps'}}: \widetilde{K}_{\eps'} \to K_{\eps'}$ is a homeomorphism. The set $K_{\eps'}= \supp f_{\eps'}\left(  a_n\right) $ is compact, since it is a closed subset of the compact set $\mathcal X_n$. It follows that $\widetilde{K}_{\eps'}=\supp \widetilde{a}_{\eps'}$ is compact, whence $\widetilde{a}_{\eps'} \in C_c\left(\widetilde{\mathcal{X}} \right)$. From the norm limit $\lim_{\eps' \to 0} \widetilde{a}_{\eps'}= \widetilde{a}$ and the Definition \ref{c_c_def_1} it follows that $\widetilde{a} \in C_0\left(\widetilde{\mathcal{X}} \right)$.
\end{proof}


\begin{lem}\label{comm_c0_lem}
	On the notation of the Lemma \ref{comm_l_infty_lem} if $\mathcal X$ is locally compact then any special element $$\widetilde{a}  \in  L^\infty\left(\widetilde{\mathcal X}\right)$$ lies in $C_0\left(\widetilde{\mathcal X}\right)$.
\end{lem}
\begin{proof}
If $a = \sum_{\widetilde{g} \in G\left(\widetilde{\mathcal X} | \mathcal X\right)}\widetilde{g}~ \widetilde{a} \in C_0\left(\mathcal X \right)$ and $\varepsilon > 0$ then from the  Definition  \ref{c_c_def_2} it follows that there is a compact subset $K \subset \mathcal X$ such that $\left\|a|_{\mathcal X \backslash K}\right\| < \varepsilon$. From the Lemma \ref{comm_spec_cont_lem} it follows that  $\widetilde{a}|_{\widetilde{\pi}^{-1}\left(\mathcal X \backslash K\right)} \in C_0\left(\widetilde{\pi}^{-1}\left(\mathcal X \backslash K\right)\right)$, i.e. there is a compact subset $\widetilde{K} \subset \widetilde{\pi}^{-1}\left( \mathcal X\backslash K\right)$ such that $\left\|\widetilde{a}|_{\widetilde{\pi}^{-1}\left(K\right) \backslash \widetilde{K}}\right\|< \varepsilon$. From the positivity of $\widetilde{a}$ it follows that
	\begin{equation*}
	\left\|\widetilde{a}|_{\widetilde{\pi}^{-1}\left(\mathcal X \backslash K\right)}\right\| \le \left\|a|_{\mathcal X \backslash K}\right\|.
	\end{equation*}
 
 In result following condition holds $\left\|\widetilde{a}|_{\widetilde{\mathcal X} \backslash \widetilde{K}}\right\|< \varepsilon$ , and from the Definition \ref{c_c_def_2} it follows that $\widetilde{a} \in C_0\left(\widetilde{\mathcal{X}}\right)$.
\end{proof}

\begin{thm}\label{comm_main_thm}
	If $\mathfrak{S}_{\mathcal X} = \left\{\mathcal{X} = \mathcal{X}_0 \xleftarrow{}... \xleftarrow{} \mathcal{X}_n \xleftarrow{} ...\right\} \in \mathfrak{FinTop}$ is a topological  finite covering sequence then  there is the natural isomorphism:
$\varprojlim \mathfrak{S}_{C_0\left(\mathcal{X}\right)} \approx C_0\left(\varprojlim \mathfrak{S}_{\mathcal X}\right)$.

\end{thm}
\begin{proof} (a)
If $\widetilde{\mathcal X}= \varprojlim \mathfrak{S}_{\mathcal{X}}$ then from the Corollary  \ref{comm_a_in_c0_cor} it follows that $\varprojlim \mathfrak{S}_{C_0\left(\mathcal{X}\right)} \subset C_0\left(\widetilde{\mathcal X} \right)$. From the Lemma \ref{special_in_w_lem} it follows that any special element lies in $\left( C_0\left(\mathcal X_n\right)\right)''$. From the Lemma \ref{comm_l_infty_lem} it follows that $\left( C_0\left(\mathcal X_n\right)\right)'' \subset L^\infty\left(\widetilde{ \mathcal X}\right)$, whence any special element lies in $L^\infty\left(\widetilde{ \mathcal X}\right)$. From this fact and the Lemma \ref{comm_c0_lem} it follows that $C_0\left(\widetilde{\mathcal X} \right) \subset \varprojlim \mathfrak{S}_{C_0\left(\mathcal{X}\right)}$. So $\varprojlim \mathfrak{S}_{C_0\left(\mathcal{X}\right)}\approx C_0\left( \widetilde{\mathcal X}\right) = C_0\left(\varprojlim \mathfrak{S}_{\mathcal X}\right)$.

\end{proof}

\section{Noncommutative tori and Moyal planes}
     \subsection{Noncommutative torus $\mathbb{T}^n_{\Theta}$}\label{nt_descr_subsec}
     \paragraph*{} 	There is the natural norm on $\mathbb{Z}^n$ given by
     \begin{equation*}
     \left\|\left(k_1, ..., k_n\right)\right\|= \sqrt{k_1^2 + ... + k^2_n}
     \end{equation*}
     The space of complex-valued Schwartz functions is given by 
     \begin{equation*}
     \sS\left(\mathbb{Z}^n\right)= \left\{a = \left\{a_k\right\}_{k \in \mathbb{Z}^n} \in \mathbb{C}^{\mathbb{Z}^n}~|~ \mathrm{sup}_{k \in \mathbb{Z}^n}\left(1 + \|k\|\right)^s \left|a_k\right| < \infty, ~ \forall s \in \mathbb{N} \right\}.
     \end{equation*}

     Let $\mathbb{T}^n$ be an ordinary $n$-torus. We will often use real coordinates for $\mathbb{T}^n$, that is, view $\mathbb{T}^n$ as $\mathbb{R}^n / 2\pi\mathbb{Z}^n$. Let $\Coo\left(\mathbb{T}^n\right)$ be the algebra of infinitely differentiable complex-valued functions on $\mathbb{T}^n$. 
     There is the bijective Fourier transform  $\Coo\left(\mathbb{T}^n\right)\to\sS\left(\mathbb{Z}^n\right)$;  $f \mapsto \widehat{f}$ given by
     \begin{equation}\label{nt_fourier}
     \widehat{f}\left(p\right)= \mathcal F (f) (p)=\frac{1}{\left( 2\pi\right)^n }\int_{\mathbb{T}^n}e^{- i x \cdot p}f\left(x\right)dx
     \end{equation}
     where $dx$ is induced by the Lebesgue measure on $\mathbb{R}^n$ and   $\.$ is the  scalar
     product on the Euclidean space $\R^n$.
     The Fourier transform carries multiplication to convolution, i.e.
     \begin{equation*}
     \widehat{fg}\left(p\right) = \sum_{r +s = p}\widehat{f}\left(r\right)\widehat{g}\left(s\right).
     \end{equation*}
     Let $\Theta$ be a real skew-symmetric $n \times n$ matrix, we will define a new noncommutative product on $\sS\left(\mathbb{T}^n\right)$ given by
     \begin{equation}\label{nt_product_defn_eqn}
     \left(\widehat{f}\star_{\Theta}\widehat{g}\right)\left(p\right)= \sum_{r + s = p} \widehat{f}\left(r\right)\widehat{g}\left(s\right) e^{-\pi ir ~\cdot~ \Theta s}.
     \end{equation}
         and an involution
     \begin{equation*}
     \widehat{f}^*\left(p\right)=\overline{\widehat{f}}\left(-p\right)).
     \end{equation*}
     In result there is the involutive algebra $\Coo\left(\mathbb{T}^n_{\Theta}\right) =\left(\sS\left(\mathbb{Z}^n\right), + , \star_{\Theta}~, ^* \right)$. 
       There is a tracial  state on $\Coo\left(\mathbb{T}^n_{\Theta}\right)$ given by
     \begin{equation}\label{nt_state_eqn}
     \tau\left(f\right)= \widehat{f}\left(0\right).
     \end{equation}
       From  $\Coo\left(\mathbb{T}^n_{\Theta} \right) \approx \SS\left( \Z^n\right)$ it follows  that there is  $\C$-linear isomorphism 
           \begin{equation}\label{nt_varphi_inf_eqn}
           \varphi_\infty: \Coo\left(\mathbb{T}^n_{\Theta} \right) \xrightarrow{\approx}  \Coo\left(\mathbb{T}^n \right).
           \end{equation} 
           such that following condition holds
     \begin{equation}\label{nt_state_integ_eqn}
            \tau\left(f \right)=  \frac{1}{\left( 2\pi\right)^n }\int_{\mathbb{T}^n} \varphi_\infty\left( f\right) ~dx.
      \end{equation}
           
     Similarly to \ref{comm_gns_constr} there is the Hilbert space $L^2\left(\Coo\left(\mathbb{T}^n_{\Theta}\right), \tau\right)$ and the natural representation $\Coo\left(\mathbb{T}^n_{\Theta}\right) \to B\left(L^2\left(\Coo\left(\mathbb{T}^n_{\Theta}\right), \tau\right)\right)$ which induces the $C^*$-norm. The norm completion  $C\left(\mathbb{T}^n_{\Theta}\right)$ of $\Coo\left(\mathbb{T}^n_{\Theta}\right)$ is a $C^*$-algebra. We will write $L^2\left(C\left(\mathbb{T}^n_{\Theta}\right), \tau\right)$ instead $L^2\left(\Coo\left(\mathbb{T}^n_{\Theta}\right), \tau\right)$. There is the natural linear map  $\Coo\left(\mathbb{T}^n_{\Theta}\right) \to L^2\left(C\left(\mathbb{T}^n_{\Theta}\right), \tau\right) $  and since $\Coo\left(\mathbb{T}^n_{\Theta}\right) \approx \sS\left( \mathbb{Z}^n \right)$ there is a linear map $\varphi:\sS\left( \mathbb{Z}^n \right) \to L^2\left(C\left(\mathbb{T}^n_{\Theta}\right), \tau\right) $. If $k = \mathbb{Z}^n$ and $U_k \in  \sS\left( \mathbb{Z}^n \right)=\Coo\left(\mathbb{T}^n_{\Theta}\right)$ is such that 
     \begin{equation}\label{unitaty_nt_eqn}
     U_k\left( p\right)= \delta_{kp}: ~ \forall p \in \mathbb{Z}^n
     \end{equation}     
     then denote $\xi_k = \varphi\left(U_k \right)$. From \eqref{nt_product_defn_eqn} it follows that
     $$
     \tau\left(U^*_k \star_\Th U_l \right) = \left(\xi_k, \xi_l \right)  = \delta_{kl},
     $$ 
     i.e. the subset $\left\{\xi_k\right\}_{k \in \mathbb{Z}^n}\subset L^2\left(C\left(\mathbb{T}^n_{\Theta}\right)\right)$ is an orthogonal basis of  $L^2\left(C\left(\mathbb{T}^n_{\Theta}\right), \tau\right)$.
     Hence the Hilbert space  $(L^2\left(C\left(\mathbb{T}^n_{\Theta}\right), \tau\right)$ is naturally isomorphic to the Hilbert space $\ell^2\left(\mathbb{Z}^n\right)$ given by
     \begin{equation*}
     \ell^2\left(\mathbb{Z}^n\right) = \left\{\xi = \left\{\xi_k \in \mathbb{C}\right\}_{k\in \mathbb{Z}^n} \in \mathbb{C}^{\mathbb{Z}^n}~|~ \sum_{k\in \mathbb{Z}^n} \left|\xi_k\right|^2 < \infty\right\}
     \end{equation*}
     and the scalar product on $\ell^2\left(\mathbb{Z}^n\right)$ is given by
     \begin{equation*}
     \left(\xi,\eta\right)= \sum_{k\in \mathbb{Z}^n}    \xi^*_k\eta_k.
     \end{equation*}
   
    There is an alternative description of $\C\left(\mathbb{T}^n_{\Theta}\right)$.  If 
    $$
    \Th = \begin{pmatrix}
    0& \th_{12} &\ldots & \th_{1n}\\
    \th_{21}& 0 &\ldots & \th_{2n}\\
    \vdots& \vdots &\ddots & \vdots\\
    \th_{n1}& \th_{n2} &\ldots & 0
    \end{pmatrix}
    $$
    then $C\left(\mathbb{T}^n_{\Theta}\right)$ is the universal $C^*$-algebra generated by unitary elements   $u_1,..., u_n \in U\left( C\left(\mathbb{T}^n_{\Theta}\right)\right) $ such that following condition holds
    \begin{equation}\label{nt_com_eqn}
    u_ku_j = e^{2\pi i \theta_{jk} }u_ju_k.
    \end{equation}
    Unitary  operators $u_1,..., u_n$ correspond to the standard basis of $\mathbb{Z}^n$.
      \begin{defn}
      	Unitary elements 
      	$u_1,..., u_n \in U\left(C\left(\mathbb{T}^n_{\theta}\right)\right)$ which satisfy the relation \eqref{nt_com_eqn}
      	are said to be \textit{generators} of $C\left(\mathbb{T}^n_{\Theta}\right)$.
      \end{defn}
           If $a \in C\left(\mathbb{T}^n_{\th}\right)$ is presented by the series
           $$
           a = \sum_{l \in \mathbb{Z}^{n}}c_l U_l;~~ c_l \in \mathbb{C}
           $$
           and the series $\sum_{l \in \mathbb{Z}^{n}}\left| c_l\right| $ is convergent then from the triangle inequality it follows that
           \begin{equation}\label{nt_norm_estimation}
           \left\|a \right\| \le \sum_{l \in \mathbb{Z}^{n}}\left| c_l\right|.
           \end{equation}
      \begin{defn}
      	If  $\Theta$ to is nondegenerate, that is to say,
      	$\sigma(s,t) \stackrel{\mathrm{def}}{=} s\.\Theta t$ to be \textit{symplectic}. This implies even
      	dimension, $n = 2N$. One then selects
      	\begin{equation}\label{nt_simpectic_theta_eqn}
      	\Theta = \theta J
      	\stackrel{\mathrm{def}}{=} \th \begin{pmatrix} 0 & 1_N \\ -1_N & 0 \end{pmatrix}
      	\end{equation}
      	where  $\th > 0$ is defined by $\th^{2N} \stackrel{\mathrm{def}}{=} \det\Theta$.
      	Denote by $C\left(\mathbb{T}^{2N}_\th\right)\stackrel{\mathrm{def}}{=}C\left(\mathbb{T}^{2N}_\Th\right)$ and $\Coo\left(\mathbb{T}^{2N}_\th\right)\stackrel{\mathrm{def}}{=}\Coo\left(\mathbb{T}^{2N}_\Th\right)$.
      \end{defn}

\subsection{Finite noncommutative covering projections of $\mathbb{T}^n_{\Theta}$}\label{ntt_fin_cov}
     \paragraph{}  In this section we write $ab$ instead $a\star_\Th b$.
    Let $C\left(\mathbb{T}^n_\Theta\right)$ be a noncommutative torus and $\overline{k} = \left(k_1, ..., k_n\right) \in \mathbb{N}^n$. If 
               $$
               \widetilde{\Theta} = \begin{pmatrix}
               0& \widetilde{\theta}_{12} &\ldots & \widetilde{\theta}_{1n}\\
               \widetilde{\theta}_{21}& 0 &\ldots & \widetilde{\theta}_{2n}\\
               \vdots& \vdots &\ddots & \vdots\\
               \widetilde{\theta}_{n1}& \widetilde{\theta}_{n2} &\ldots & 0
               \end{pmatrix}
               $$
       is a skew-symmetric matrix such that
            \begin{equation*}
            e^{2\pi i \theta_{rs}}= e^{2\pi i \widetilde{\theta}_{rs}k_rk_s}
            \end{equation*}
            then there is a *-homomorphism $C\left(\mathbb{T}^n_\Th\right)\to C\left(\mathbb{T}^n_{\widetilde{\Th}}\right)$ given by
            \begin{equation}\label{nt_cov_eqn}
            u_j \to v~^{k_j}_j; ~ j = 1,...,n
            \end{equation}
            where $u_1,..., u_n \in C\left(\mathbb{T}^n_{\Th}\right)$ (resp. $v_1,..., v_n \in C\left(\mathbb{T}^n_{\widetilde{\Th}}\right)$) are unitary generators of $C\left(\mathbb{T}^n_{\Th}\right)$ (resp. $C\left(\mathbb{T}^n_{\widetilde{\Th}}\right)$).	
            There is an involutive action of $G=\mathbb{Z}_{k_1}\times...\times\mathbb{Z}_{k_n}$ on $C\left(\mathbb{T}^n_{\widetilde{\Th}}\right)$ given by
            \begin{equation*}
            \left(\overline{p}_1,..., \overline{p}_n\right)v_j = e^{\frac{2\pi i p_j}{k_j}}v_j
            \end{equation*}
            and $C\left(\mathbb{T}^n_{\Th}\right)=C\left(\mathbb{T}^n_{\widetilde{\Th}}\right)^G$, i.e. the condition (a) of the Definition \ref{fin_def} holds.
            Let $\{e_{\iota}^m\}_{\iota \in I_m}\in C(S^1)$ be given by \eqref{iota_def}.
            From \eqref{circ_sum} it follows that
            \begin{equation}\tag{*}
            \begin{split}
            \sum_{\iota \in I_{n_j} \ }  e^{k_j}_{\iota_j}(v_j) \left(\overline{p}  e^{k_j}_{\iota_j}(v_j)\right) =  \left\{
            \begin{array}{c l}
            1_{C\left(\mathbb{T}^n_{\widetilde{\Th}}\right)} & \overline{p}=\overline{0} \\
            0 & \overline{p}\neq\overline{0}
            \end{array}\right. \text{ where } \overline{p} \in \mathbb{Z}_{n_j} .
            \end{split}
            \end{equation}
            If $I = I_{k_1} \times ... \times I_{k_n}$, $\iota = \left(\iota_1,..., \iota_n\right) \in I$ and
            \begin{equation*}
            \begin{split}
            \al_{\iota}= e^{k_1}_{\iota_1}(v_1)\cdot...  \cdot e^{k_n}_{\iota_n}(v_n), \\
            \beta_{\iota}= \alpha^*_{\iota}= e^{k_n}_{\iota_n}(v_n)\cdot... \cdot e^{k_1}_{\iota_1}(v_1)
            \end{split}
            \end{equation*}
            
            then from (*) it follows that
            \begin{equation*} 	
            \begin{split}
            \sum_{\iota \in I}\beta_{\iota}\al_{\iota} = 	\sum_{\iota_1 \in I_1, ..., \iota_n \in I_n} e^{k_n}_{\iota_n}(v_n)\cdot... \cdot e^{k_1}_{\iota_1}(v_1)\cdot e^{k_1}_{\iota_1}(v_1)\cdot...  \cdot e^{k_n}_{\iota_n}(v_n)=\\ 
            = 	\sum_{\iota_2 \in I_2, ..., \iota_n \in I_n} \left( e^{k_n}_{\iota_n}(v_n)\cdot... \cdot e^{k_2}_{\iota_2}(v_2) \left(\sum_{\iota_1 \in I_1}\left(e^{k_1}_{\iota_1}(v_1)\right)^2\right)e^{k_2}_{\iota_2}(v_2)\cdot...  \cdot e^{k_n}_{\iota_n}(v_n) \right) =\\
            = 	\sum_{\iota_2 \in I_2, ..., \iota_n \in I_n} e^{k_n}_{\iota_n}(v_n)\cdot... \cdot e^{k_2}_{\iota_2}(v_2) \cdot 1_{C\left(\mathbb{T}^n_{\widetilde{\Th}}\right)}\cdot e^{k_2}_{\iota_2}(v_2)\cdot...  \cdot e^{k_n}_{\iota_n}(v_n) = \\
            = 	\sum_{\iota_2 \in I_2, ..., \iota_n \in I_n}  e^{k_n}_{\iota_n}(v_n)\cdot... \cdot e^{k_2}_{\iota_2}(v_2) \cdot e^{k_2}_{\iota_2}(v_2)\cdot...  \cdot e^{k_n}_{\iota_n}(v_n) 
            = ... = 1_{C\left(\mathbb{T}^n_{\widetilde{\Th}}\right)}
            \end{split}
            \end{equation*} 
            If $g = (\overline{p}_1, ...,\overline{p}_{n}) \in \mathbb{Z}_{k_1}\times ...\times \mathbb{Z}_{k_n}$ is such that $\overline{p}_1\neq \overline{0}$ then from (*) it follows that
            \begin{equation*}
            \begin{split}
            \sum_{\iota \in I}\beta_{\iota}\left(g\al_{\iota}\right) = 	\sum_{\iota_1 \in I_1, ..., \iota_n \in I_n} e^{k_n}_{\iota_n}(v_n)\cdot... \cdot e^{k_1}_{\iota_1}(v_1)\cdot \left(\overline{p}_1e^{k_1}_{\iota_1}(v_1)\right)\cdot...  \left(\overline{p}_ne^{k_n}_{\iota_n}(v_n)\right)=\\ 
            = 	\sum_{\iota_2 \in I_2, ..., \iota_n \in I_n} \left( e^{k_n}_{\iota_n}(v_n)\cdot... \cdot e^{k_2}_{\iota_2}(v_2) \left(\sum_{\iota_1 \in I_1}e^{k_1}_{\iota_1}(v_1)\left(\overline{p}_1e^{k_1}_{\iota_1}(v_1)\right)\right)\left(\overline{p}_2e^{k_2}_{\iota_2}(v_2)\right)\cdot...  \cdot \left(\overline{p}_ne^{k_n}_{\iota_n}(v_n)\right) \right) =\\
            = 	\sum_{\iota_2 \in I_2, ..., \iota_n \in I_n} e^{k_n}_{\iota_n}(v_n)\cdot... \cdot e^{k_2}_{\iota_2}(v_2) \cdot 0\cdot\left(\overline{p}_2e^{k_2}_{\iota_2}(v_2)\right)\cdot...  \cdot \left(\overline{p}_ne^{k_n}_{\iota_n}(v_n)\right)=  0.
            \end{split}
            \end{equation*} 
            
            If $\overline{p}_1 = 0$ and $\overline{p}_2 \neq 0$  then from (*) it follows that
            \begin{equation*}
            \begin{split}
            \sum_{\iota \in I}\beta_{\iota}\left(g\al_{\iota}\right) = 	\sum_{\iota_1 \in I_1, ..., \iota_n \in I_n} e^{k_n}_{\iota_n}(v_n)\cdot... \cdot e^{k_1}_{\iota_1}(v_1)\cdot \left(\overline{p}_1e^{k_1}_{\iota_1}(v_1)\right)\cdot...  \left(\overline{p}_ne^{k_n}_{\iota_n}(v_n)\right)=\\ 
            = 	\sum_{\iota_2 \in I_2, ..., \iota_n \in I_n} \left( e^{k_n}_{\iota_n}(v_n)\cdot... \cdot e^{k_2}_{\iota_2}(v_2) \left(\sum_{\iota_1 \in I_1}e^{k_1}_{\iota_1}(v_1)\left(\overline{p}_1e^{k_1}_{\iota_1}(v_1)\right)\right)\left(\overline{p}_2e^{k_2}_{\iota_2}(v_2)\right)\cdot...  \cdot \left(\overline{p}_ne^{k_n}_{\iota_n}(v_n)\right) \right) =\\
            = 	\sum_{\iota_2 \in I_2, ..., \iota_n \in I_n} e^{k_n}_{\iota_n}(v_n)\cdot... \cdot e^{k_2}_{\iota_2}(v_2) \cdot 	1_{C\left(\mathbb{T}^n_{\widetilde{\Th}}\right)}\cdot\left(\overline{p}_2e^{k_2}_{\iota_2}(v_2)\right)\cdot...  \cdot \left(\overline{p}_ne^{k_n}_{\iota_n}(v_n)\right)= \\
            =	\sum_{\iota_2 \in I_2, ..., \iota_n \in I_n} e^{k_n}_{\iota_n}(v_n)\cdot... \cdot e^{k_2}_{\iota_2}(v_2)\cdot \left(\overline{p}_2e^{k_2}_{\iota_2}(v_2)\right)\cdot...  \left(\overline{p}_ne^{k_n}_{\iota_n}(v_n)\right)=\\
            = 	\sum_{\iota_3 \in I_3, ..., \iota_n \in I_n} \left( e^{k_n}_{\iota_n}(v_n)\cdot... \cdot e^{k_3}_{\iota_3}(v_3) \left(\sum_{\iota_2 \in I_2}e^{k_2}_{\iota_2}(v_2)\left(\overline{p}_2e^{k_2}_{\iota_2}(v_2)\right)\right)\left(\overline{p}_3e^{k_3}_{\iota_3}(v_3)\right)\cdot...  \cdot \left(\overline{p}_ne^{k_n}_{\iota_n}(v_n)\right) \right) =\\
            =	\sum_{\iota_3 \in I_3, ..., \iota_n \in I_n} e^{k_n}_{\iota_n}(v_n)\cdot... \cdot e^{k_3}_{\iota_3}(v_3) \cdot 0\cdot\left(\overline{p}_3e^{k_3}_{\iota_3}(v_3)\right)\cdot...  \cdot \left(\overline{p}_ne^{k_n}_{\iota_n}(v_n)\right)=  0.
            \end{split}
            \end{equation*} 
            Similarly if $\overline{p}_1, \overline{p}_2 = 0$ and $\overline{p}_3 \neq 0$ then 	$\sum_{\iota \in I}\beta_{\iota}\left(g\al_{\iota}\right) = 0$ and so on.
            In result we have
            \begin{equation*}
            \sum_{\iota \in I}\beta_{\iota}\left(g\al_{\iota}\right)=\left\{
            \begin{array}{c l}
            1 & g \in G = \mathbb{Z}_{k_1}\times...\times\mathbb{Z}_{k_n} \text{ is trivial}\\
            0 & g \in  G = \mathbb{Z}_{k_1}\times...\times\mathbb{Z}_{k_n} \text{ is not trivial}
            \end{array}\right.,
            \end{equation*}
            i.e. condition (b) of the Definition \ref{fin_def} holds. It follows to the next lemma.
           
            \begin{lem}
            
           The triple $\left(C\left(\mathbb{T}^n_{\Th}\right), C\left(\mathbb{T}^n_{\widetilde{\Th}}\right),\mathbb{Z}_{k_1}\times...\times\mathbb{Z}_{k_n}\right)$ is a  noncommutative finite covering projection. 
           \end{lem}
\subsection{Multidimensional grading of a noncommutative torus}\label{nt_grad_sec}

\paragraph{}It is known\cite{connes_landi:isospectral}   that  continuous actions of the torus $\T^2$ can define $\Z^2$-grading of $C^*$-algebras. Here similar idea is discussed. Let us define a continuous action of $\R$ on $C\left(u \right)\approx C\left( S^1\right)$ given by  
$$
x \bullet u = e^{2\pi ix} u;~ \forall x \in \R.
$$
Now we define a $\Z$-grading $C\left( u\right)= \bigoplus_{m \in \Z} C\left( u\right)_m$ where $$
C\left( u\right)_m = \left\{a \in C\left(u \right)~|~ x\bullet a =  e^{2\pi imx}a \right\}$$ and $\bigoplus$ means $C^*$-norm completion of the algebraic direct sum. If $C\left(u \right)\to C\left(v \right)$ is a *-homomorphism $u \mapsto v^r$ then the action $\bullet$ can be uniquely extended up to $C\left( v\right)$ and then $\Z$-grading of $C\left(u \right)$ induces the following grading
$$
C\left(v\right)= \bigoplus_{\frac{m}{r} \in \frac{\Z}{r}}C\left( v\right)_{\frac{m}{r}} 
$$
where $\frac{\Z}{r} = \left\{\frac{m}{r} \in \Q~|~ m   \in \Z\right\}$ and
$$
C\left( v\right)_{\frac{m}{r}} = \left\{a \in C\left(v \right)~|~ x\bullet a =  e^{\frac{2\pi imx}{r}}a \right\}.
$$    
Indeed $C\left( v\right)_{\frac{m}{r}}$ is a one-dimensional $\C$-space given by $C\left( v\right)_{\frac{m}{r}} =  \C v^m$. For any $j \in 1,..., n$ there is the action $\bullet_j$ of $\R$ on $C\left(\mathbb{T}^n_{\Th}\right)$ given by
$$
x \bullet_j u_k = e^{2\pi i \delta_{jk}x} u_k
$$
where $u_1, ..., u_n$ is a basis of $C\left(\mathbb{T}^n_{\Th}\right)$. If $j'\neq j''$ then the action $\bullet_{j'}$ commutes with the action $\bullet_{j''}$ and the sum of actions $\bullet_j$ ($j=1,...,n$) defines the action $\bullet$ of $\R^n$ on $C\left(\mathbb{T}^n_{\Th}\right)$. There is a $\Z^n$-grading
$$
C\left(\mathbb{T}^n_{\Th}\right)= \bigoplus_{\left(m_1,...,m_n \right)\in \Z^n }C\left(\mathbb{T}^n_{\Th}\right)_{\left(m_1,...,m_n \right)}
$$ 
where 
$$
C\left(\mathbb{T}^n_{\Th}\right)_{\left(m_1,...,m_n \right)}= \left\{a \in C\left(\mathbb{T}^n_{\Th}\right)~|~ \left(x_1,...,x_n \right)a=  e^{2\pi i \left(m_1x_1+...+m_nx_n \right) } a \right\}
$$
Indeed $
C\left(\mathbb{T}^n_{\Th}\right)_{\left(m_1,...,m_n \right)}$ is a one-dimensional $\C$ space such that
$
C\left(\mathbb{T}^n_{\Th}\right)_{\left(m_1,...,m_n \right)} = \C\cdot U_{\left(m_1,...,m_n \right) }$ where $U_{\left(m_1,...,m_n \right) }$ is given by \eqref{unitaty_nt_eqn}.
If $C\left(\mathbb{T}^n_{\Th}\right)\to C\left(\mathbb{T}^n_{\widetilde{\Th}}\right)$ is described in \ref{ntt_fin_cov} then $\Z^n$-grading of $C\left(\mathbb{T}^n_{\Th}\right)$ induces the $\frac{\Z}{k_1} \times ... \times \frac{\Z}{k_n}$-grading 
$$
C\left(\mathbb{T}^n_{\widetilde{\Th}}\right)= \bigoplus_{\left(\frac{m_1}{k_1},...,\frac{m_n}{k_n} \right)\in \frac{\Z}{k_1} \times ... \times \frac{\Z}{k_n} }C\left(\mathbb{T}^n_{\widetilde{\Th}}\right)_{\left(\frac{m_1}{k_1},...,\frac{m_n}{k_n} \right)}
$$ 
where 
$$
C\left(\mathbb{T}^n_{\widetilde{\Th}}\right)_{\left(\frac{m_1}{k_1},...,\frac{m_n}{k_n} \right)}= \left\{a \in C\left(\mathbb{T}^n_{\widetilde{\Th}}\right)~|~ \left(x_1,...,x_n \right)a=  e^{2\pi i \left(\frac{m_1}{k_1}x_1+...+\frac{m_n}{k_n}x_n \right) } a \right\}.
$$
If $v_1, ..., v_n\in U\left(C\left(\mathbb{T}^n_{\widetilde{\Th}}\right) \right) $ are generators of $C\left(\mathbb{T}^n_{\widetilde{\Th}}\right)$ which satisfy to \eqref{nt_cov_eqn} then
$$
C\left(\mathbb{T}^n_{\widetilde{\Th}}\right)_{\left(\frac{m_1}{k_1},...,\frac{m_n}{k_n} \right)} = \C\cdot v_1^{m^1}\cdot...\cdot v_n^{m_m},
$$
i.e. $C\left(\mathbb{T}^n_{\widetilde{\Th}}\right)_{\left(\frac{m_1}{k_1},...,\frac{m_n}{k_n} \right)}$ is a one-dimensional $\C$-space. 
       
\subsection{The Moyal plane}
\label{sec:Moyal-basics}

\paragraph{}Let  $\Theta$ be given by \eqref{nt_simpectic_theta_eqn}
  and let $\SS(\R^{2N})$ be the space of
complex Schwartz (smooth, rapidly decreasing) functions on $\R^{2N}$. One
defines, for $f,h \in \SS(\R^{2N})$, the corresponding Moyal
or twisted product:
\begin{equation}\tag{*}
f \star_\theta h(x) \stackrel{\mathrm{def}}{=} (2\pi)^{-k} \iint
f(x - \half\Theta u) \, h(x + t) \, e^{-iu\.t} \,du \,dt,
\end{equation}
where $\th > 0$ is defined by $\th^{2N} \stackrel{\mathrm{def}}{=} \det\Theta$. The
formula (*) may be rewritten as
\begin{equation*}
f \star_\theta h(x) = (\pi\th)^{-2N} \iint
f(x + s)\, h(x + t)\, e^{-2is\.\Theta^{-1}t} \,d^{2N}s \,d^{2N}t.
\label{eq:moyal-prod-gen}
\end{equation*}
\begin{defn}\label{r_2_N_repr}\cite{moyal_spectral} Let $\SS'\left(R^{2N} \right)$ be a vector space dual to $\SS\left(R^{2N} \right)$ and let 	$\left\|\cdot\right\|_{2}$ be the $L^2$-norm given by
	\begin{equation}\label{nt_l2_norm_eqn}
	\left\|f\right\|_{2} = \left|\int_{\R^{2N}} \left|f\right|^2 dx \right|^{\frac{1}{2}}.
	\end{equation}
	Denote by $C_b\left(\R^{2N}_\th\right)\stackrel{\mathrm{def}}{=} \set{T \in \SS'\left(\R^{2N}\right) : T \mop g \in L^2\left(\R^{2N}\right) \text{ for all } g \in L^2(\R^{2N})}$, provided with the operator norm $\|T\|_{\mathrm{op}} \stackrel{\mathrm{def}}{=}\sup\set{\|T \mop g\|_2/\|g\|_2 : 0 \neq g \in L^2(\R^{2N})}$ and $\Coo_{0}\left(\R^{2N}_\th\right) \stackrel{\mathrm{def}}{=}  \SS\left(\R^{2N}\right)$ and denote by $C_0\left(\R^{2N}_\th\right)$ the operator norm completion of $\Coo_{0}\left(\R^{2N}_\th\right)$.

\end{defn}
\begin{rem}
Obviously $\Coo_0\left(\R^{2N}_\th\right)  \hookto C_b\left(\R^{2N}_\th\right)$. But $\Coo_{0}\left(\R^{2N}_\th\right)$ is not dense in $C_b\left(\R^{2N}_\th\right)$, i.e. $C_0\left(\R^{2N}_\th\right) \subsetneq C_b\left(\R^{2N}_\th\right)$ (See \cite{moyal_spectral}).
\end{rem}
\begin{rem}
	Notation of the Definition \ref{r_2_N_repr} differs from \cite{moyal_spectral}. Here symbols $A_\th, \A_\th, A^0_\th$ are replaced with $C_b\left(\R^{2N}_\th\right), \Coo_{0}\left(\R^{2N}_\th\right), C_0\left(\R^{2N}_\th\right)$ respectively.
\end{rem}
\begin{lem}\label{nt_l_2_est_lem}\cite{moyal_spectral}
If $f,g \in L^2 \left(\R^{2N} \right)$, then $f\star_\th g \in L^2 \left(\R^{2N} \right)$ and $\left\|f\right\|_{\mathrm{op}} < \left(2\pi\th \right)^{-\frac{N}{2}} \left\|f\right\|_2$.

\end{lem}
\begin{defn}\label{nt_l_2_est_defn}
	Denote by  $L^2 \left(\R^{2N}_\th \right)$ the space $L^2 \left(\R^{2N}\right)$ with $\star_\th$ multiplication.
\end{defn}
Introduce the  the ordinary Fourier transform $\mathcal{F}$ and a symplectic Fourier transform $F$ 
\cite{miracle:ph_d} given by
\begin{equation}\label{eq:Fsympl}
\begin{split}
 \left(\mathcal{F}f\right)(u) = \int_{\mathbb{R}^{2N}} f(t)e^{-it\cdot u}dt,~\left(Ff\right)(u)= \int_{\mathbb{R}^{2N}} f(t)e^{-it\cdot J u}dt 
 \end{split}
\end{equation}
where
$$
J = \begin{pmatrix} 0 & 1_N \\ -1_N & 0 \end{pmatrix}.
$$
 The given by
\begin{equation}\label{nt_l2_scalar_prod_eqn}
\left(f,g\right) = \int_{\mathbb{R}^{2N}}f^*g dx.
\end{equation}
 sesquilinear scalar product  on $\sS\left(\mathbb{R}^k\right)$
is $\mathcal{F}$ invariant, i.e.
\begin{equation}\label{scalar_fourier_inv_eqn}
\left(f,g\right) = \left(\mathcal{F}f,\mathcal{F}g\right),
\end{equation}
and there is the tracial property \cite{moyal_spectral}  
\begin{equation}\label{nt_tracial_prop}
\int_{\R^{2N}} \left( f\star_\th g\right) \left(x \right)dx =  \int_{\R^{2N}}  f\left(x \right) g\left(x \right)dx.
\end{equation}
Consider the unitary dilation operators \cite{moyal_spectral}   $E_a$ given
by
\begin{equation}\label{dilation_eqn}
E_af(x) \stackrel{\mathrm{def}}{=} a^{N/2} f(a^{1/2}x),
\end{equation}
and it is immediate from~\eqref{eq:Fsympl} that $F E_a = E_{1/a} F$.
We also remark that
\begin{equation}\label{eq:starscale}
f \mop g =
(\th/2)^{-N/2} E_{2/\th}(E_{\th/2}f \star_2 E_{\th/2}g).
\end{equation}
Some formulas in this paper can be simplified when $\th = 2$. Thanks to
the scaling relation~\eqref{eq:starscale}, it is often enough, when
studying properties of the Moyal product, to work out the case
$\th = 2$. Denote $\times \stackrel{\text{def}}{=} \star_2$.
According to \cite{varilly_bondia:phobos} following conditions hold:
\begin{equation*}
(f\times g)(u) = \int\int f(v)g(w) \exp\left(
i(u\cdot Jv + v \cdot Jw + w\cdot Ju)\right)dvdw=
\end{equation*}\begin{equation*}
= \int\int f(u + s)g(u + t) e^{is\cdot Jt} ds dt
\end{equation*}

\begin{equation*} 
\left(f\times g\right)\left(u\right)= \int_{\mathbb{R}^{2N}} Ff\left(u-w\right)g\left(w\right)e^{iu\cdot Jw}dw.
\end{equation*}
The twisted convolution \cite{varilly_bondia:phobos} $f\diamond g$ is defined by
\begin{equation}\label{nt_diamond}
f\diamond g \left(u \right) = \int f\left(u - t\right)  g\left(t \right)e^{-iu \cdot J t}  dt.
\end{equation}
Following conditions hold \cite{varilly_bondia:phobos}:
\begin{equation}\label{nt_prod_duality}
\mathcal{F}\left(f \times g \right) = \mathcal{F} f \diamond \mathcal{F} f ; ~ \mathcal{F}\left(f \diamond g \right) = \mathcal{F} f \times \mathcal{F} f.
\end{equation}



  \subsection{The Moyal plane as the inverse noncommutative limit}
  \begin{empt}\label{nt_mp_prel_lim}
  	Let $\{p_j \in \mathbb{N}\}_{j \in \mathbb{N}}$ be an infinite sequence of natural numbers such that $p_j > 1$ for any $j$, and let $m_n = \Pi_{j = 1}^n p_j$. From the \ref{ntt_fin_cov} it follows that there is a  sequence of *-homomorphisms
  	
  	\begin{equation*}
 		C\left(\mathbb{T}^{2N}_\th\right) \to	C\left(\mathbb{T}^{2N}_{\th/m_1^{2}}\right) \to C\left(\mathbb{T}^{2N}_{\th/m_2^{2}}\right) \to... 
  	\end{equation*}
  	such that
  	\begin{enumerate}
  		\item[(a)] 	For any $n \in \mathbb{N}$ there are generators $u_{n,1},..., u_{n,2N}\in U\left(C\left(\mathbb{T}^{2N}_{\th/m_n^{2}}\right)\right)$ be a basis of $C\left(\mathbb{T}^{2N}_{\th/m_n^{2}}\right)$ such that the *-homomorphism $ C\left(\mathbb{T}^{2N}_{\th/m_{n-1}^{2}}\right)\to C\left(\mathbb{T}^{2N}_{\th/m_n^{2}}\right)$ is given by
  		$$
  		u_{n-1,j} \mapsto u^{p_n}_{n,j}; ~~ \forall j =1,..., 2N;
  		$$
  		and there are generators $u_1,...,u_n \in U\left( C\left(\mathbb{T}^{2N}_\th\right)\right) $ such that *-homomorphism $C\left(\mathbb{T}^{2N}_\th\right) \to C\left(\mathbb{T}^{2N}_{\th/m_1^{2}}\right)$ is given by
  	 		$$
  	 		u_{j} \mapsto u^{p_1}_{1,j}; ~~ \forall j =1,..., 2N.
  	 		$$
  	 			 
  		\item[(b)] For any $n \in \N$ the triple $\left(C\left(\mathbb{T}^{2N}_{\th/m_{n - 1}^{2}}, C\left(\mathbb{T}^{2N}_{\th/m_n^{2}}\right), \Z_{p_n}\right) \right)$ is a finite covering projection.
  		\item[(c)] There is the sequence of groups and epimorphisms
  		\begin{equation*}
  		\mathbb{Z}^{2N}_{m_1} \leftarrow\mathbb{Z}^{2N}_{m_2} \leftarrow ...;
  		\end{equation*}
  		which is equivalent to the sequence
  		$$
  		G\left(C\left(\mathbb{T}^{2N}_{\th/m_1^{2}}\right)~|~ C\left(\mathbb{T}^{2N}_{\th}\right)\right)\leftarrow G\left(C\left(\mathbb{T}^{2N}_{\th/m_2^{2}}\right)~|~ C\left(\mathbb{T}^{2N}_{\th}\right)\right)\leftarrow...
  		$$
  	\end{enumerate}
  	\begin{lem}
  		The sequence of of noncommutative finite covering projections
  		
  		\begin{equation*}
  		\mathfrak{S}_{C\left(\mathbb{T}^{2N}_\th\right)}= \left\{ 		C\left(\mathbb{T}^{2N}_\th\right) \xrightarrow{\pi^1}	C\left(\mathbb{T}^{2N}_{\th/m_1^{2}}\right) \xrightarrow{\pi^2} C\left(\mathbb{T}^{2N}_{\th/m_2^{2}}\right) \xrightarrow{\pi^3}... \right\}.
  		\end{equation*}
  		is composable.
  	\end{lem}
\begin{proof}
	We should check conditions (a), (b), (c) of the Definition \ref{comp_defn}.
	\newline (a) Let $\pi^{n_1}\circ ...\circ\pi^{n_0+1}\circ\pi^{n_0}:C\left(\mathbb{T}^{2N}_{\th/m_{n_0}^{2}}\right)\to C\left(\mathbb{T}^{2N}_{\th/m_{n_1}^{2}}\right)$ be any composition, and $p = n_0,~q=n_1$. If  $u_{p,1},..., u_{p,2N}\in U\left(C\left(\mathbb{T}^{2N}_{\th/m_p^{2}}\right)\right)$ (resp. $u_{q,1},..., u_{q,2N}\in U\left(C\left(\mathbb{T}^{2N}_{\th/m_q^{2}}\right)\right)$) are generators of $C\left(\mathbb{T}^{2N}_{\th/m_p^{2}}\right)$ (resp $C\left(\mathbb{T}^{2N}_{\th/m_q^{2}}\right)$) that the *-homomorphism 
	\begin{equation*}\tag{*}
 C\left(\mathbb{T}^{2N}_{\th/m_{p}^{2}}\right)\to C\left(\mathbb{T}^{2N}_{\th/m_q^{2}}\right)
	\end{equation*}
 is given by
	$$
	u_{p,j} \mapsto u^{d}_{q,j}; ~~ \forall j =1,..., 2N
	$$
	where $d =  \Pi_{j = p+1}^q p_j$. From \ref{ntt_fin_cov} it follows that (*) is a finite noncommutative covering projection.
	\newline 
	(b) If $k < l < m$ then there the sequence
	\begin{equation*}
C\left(\mathbb{T}^{2N}_{\th/m_{k}^{2}}\right)\to C\left(\mathbb{T}^{2N}_{\th/m_{l}^{2}}\right) \to C\left(\mathbb{T}^{2N}_{\th/m_{m}^{2}}\right)
	\end{equation*}
	of noncommutative covering projections. There is an isomorphism
	$$
	G\left( C\left(\mathbb{T}^{2N}_{\th/m_{m}^{2}}\right)~|~C\left(\mathbb{T}^{2N}_{\th/m_{k}^{2}}\right) \right) \approx \Z^{2N}_{r}
	$$
	where $r = \Pi_{j = k+1}^m p_j$. If $u_{m,1},..., u_{m,2N}\in U\left(C\left(\mathbb{T}^{2N}_{\th/m_m^{2}}\right)\right)$ is a basis of $C\left(\mathbb{T}^{2N}_{\th/m_m^{2}}\right)$ then the action of $\Z^{2N}_{r}$ is given by
	$$
	\left(\overline{t}_1,...,\overline{t}_n \right)u_{m,j} = e^{{2\pi i t_j}}u_{m,j}.
	$$
	If $u_{l,1},..., u_{l,2N}\in U\left(C\left(\mathbb{T}^{2N}_{\th/m_l^{2}}\right)\right)$ are generators  of $C\left(\mathbb{T}^{2N}_{\th/m_l^{2}}\right)$   then the action of $\Z^{2N}_{r}$ on $C\left(\mathbb{T}^{2N}_{\th/m_{l}^{2}}\right) $ satisfies to the following condition
	\begin{equation*}
	\left(\overline{t}_1,...,\overline{t}_n \right)u_{l,j} = e^{\frac{2\pi i r t_j}{s}}u_{l,j}
	\end{equation*}
	where $s = \Pi_{j = l+1}^m p_j$.
	Thus $G\left( C\left(\mathbb{T}^{2N}_{\th/m_{m}^{2}}\right)~|~C\left(\mathbb{T}^{2N}_{\th/m_{k}^{2}}\right) \right)$ transforms generators of $C\left(\mathbb{T}^{2N}_{\th/m_l^{2}}\right)$ to generators of $C\left(\mathbb{T}^{2N}_{\th/m_l^{2}}\right)$, it follows that
	$$
		G\left( C\left(\mathbb{T}^{2N}_{\th/m_{m}^{2}}\right)~|~C\left(\mathbb{T}^{2N}_{\th/m_{k}^{2}}\right) \right)C\left(\mathbb{T}^{2N}_{\th/m_l^{2}}\right) = \Z^{2N}_{r}C\left(\mathbb{T}^{2N}_{\th/m_l^{2}}\right) = C\left(\mathbb{T}^{2N}_{\th/m_l^{2}}\right).
		$$
	\newline
	(c)	
	The sequence
	\begin{equation*}\tag{**}
	\begin{split}
	\{e\}\to		G\left( C\left(\mathbb{T}^{2N}_{\th/m_{m}^{2}}\right)~|~C\left(\mathbb{T}^{2N}_{\th/m_{l}^{2}}\right) \right) \to 		G\left( C\left(\mathbb{T}^{2N}_{\th/m_{m}^{2}}\right)~|~C\left(\mathbb{T}^{2N}_{\th/m_{k}^{2}}\right) \right)\to \\ \to 		G\left( C\left(\mathbb{T}^{2N}_{\th/m_{l}^{2}}\right)~|~C\left(\mathbb{T}^{2N}_{\th/m_{k}^{2}}\right) \right) \to \{e\}
	\end{split}
	\end{equation*}
	is equivalent to 
\begin{equation*}\tag{***}
\{e\}\to \Z^{2N}_{s} \xrightarrow{\times \frac{r}{s}} \Z^{2N}_{r} \xrightarrow{\mathrm{mod}~ s} \Z^{2N}_{\frac{r}{s}} \to \{e\}.
\end{equation*}	

From  exactness of (***) it follows that the sequence (**) is exact.
\end{proof}

   	If $n \in \mathbb{N}$ and $U_k\in U\left(  C\left(\mathbb{T}^{2N}_{\th/m_n^2}\right)\right) $ is given by  \eqref{unitaty_nt_eqn}  then there is the natural the inclusion $C\left(\mathbb{T}^{2N}_{\th/m_n^2}\right)  \to C_b\left(\mathbb{R}^{2N}_\th\right)$ given by 
  	\begin{equation}\label{u_l_mapsto_eqn}
  	U_k \mapsto \exp(2\pi i \frac{k}{m_n} \cdot-);~ k \in \mathbb{Z}^{2N}
  	\end{equation}
  such that the following diagram
  	\newline
  	\begin{tikzpicture}
  	\matrix (m) [matrix of math nodes,row sep=3em,column sep=4em,minimum width=2em]
  	{
  		C\left(\mathbb{T}^{2N}_{\th/m_{k_1}^2}\right) &	& \left(\mathbb{T}^{2N}_{\th/m_{k_2}^2}\right)    \\
  		& C_b\left(\R^{2N}_\th\right) &   \\};
  	\path[-stealth]
  	(m-1-1) edge node [left] {} (m-1-3)
  	(m-1-1) edge node [right] {} (m-2-2)
  	(m-1-3) edge node [left] {}  (m-2-2);
  	
  	\end{tikzpicture}
  	\newline
  	is commutative. From the Definition \ref{r_2_N_repr} it follows that there is a faithful representation $ C_b\left(\R^{2N}_\th\right) \to B\left(L^2\left(\R^{2N}\right)\right)$
  	which for any $n \in \mathbb{N}$ induces the representation $C\left(\mathbb{T}^{2N}_{\th/m_n^2}\right) \to B\left(L^2\left(\R^{2N}\right)\right)$, whence there are the natural faithful injective homomorphisms:
 
  	\begin{equation}\label{nt_psi_eqn}
  	\psi: \varinjlim C\left(\mathbb{T}^{2N}_{\th/m_n^2}\right) \xrightarrow{\subset} C_b\left(\R^{2N}_\th\right),
  	  	\end{equation} 
  	  	
  	\begin{equation*}
  	\left(\varinjlim C\left(\mathbb{T}^{2N}_{\th/m_n^2}\right)\right)''\xrightarrow{\subset} \left(C_b\left(\R^{2N}_\th\right)\right)''
  	\end{equation*}
   There is the  action of $\mathbb{Z}^{2N}$ on $\mathbb{R}^{2N}$ given by
  	\begin{equation}\label{nt_group_action}
  	g x = 2\pi g + x; ~ g \in \mathbb{Z}^{2N}, ~ x \in \mathbb{R}^{2N} 
  	\end{equation}
  	and the action naturally induces the action of $\mathbb{Z}^{2N}$  on  $C_b\left(\R^{2N}_\th\right)$.  For any $ n \in \mathbb{N}$ there is the natural surjective homomorphism 
  	$$h_n: \mathbb{Z}^{2N}\to G\left(C\left(\mathbb{T}^{2N}_{\th/m_n^2}\right)~|~C\left(\mathbb{T}^{2N}_{\th}\right)\right) \approx\mathbb{Z}^{2N}_{m_n}.
  	$$
  	From the inclusion \ref{nt_psi_eqn} it follows that for any $n \in \mathbb{N}$ 
  	the action of $\mathbb{Z}^{2N}$ on $C_b\left(\R^{2N}_\th\right)$ induces the action $\mathbb{Z}^{2N}$ on $C\left(\mathbb{T}^{2N}_{\th/m_n^2}\right)$ such that for any $a \in C\left(\mathbb{T}^{2N}_{\th/m_n^2}\right)$
  \begin{equation}\label{nt_action_desc}
  	 \psi\left( h_n\left(g \right)a \right) = g\psi\left( a \right) 
  \end{equation}
 It follows that there is the action of $\mathbb{Z}^{2N}$ on $\varinjlim C\left(\mathbb{T}^{2N}_{\th/m_n^2}\right)$, i.e.
  	\begin{equation}\label{nt_zn_act}
  		\psi\left( g\widehat{a} \right) = g\psi\left( \widehat{a} \right);~ \forall \widehat{a} \in \varinjlim C\left(\mathbb{T}^{2N}_{\th/m_n^2}\right)
  \end{equation}
 and the inclusion \eqref{nt_psi_eqn} is $\Z^{2N}$ equivariant. Let
  	\begin{equation}\label{nt_pr_k_eqn}
  	\mathrm{pr}_n: \sS\left(\mathbb{R}^{2N}\right)\to \Coo_0\left(\mathbb{T}^{2N}_{\th/m_n^{2}}\right): ~~ a \mapsto \sum_{g \in J_n} ga
  	\end{equation} 
  	where $J_n = \ker\left(  \mathbb{Z}^{2N}\to\mathbb{Z}^{2N}_{m_n}\right) $. The map  $\mathrm{pr}_n$ can be defined by Fourier transform. Let $a \in \sS\left(\mathbb{R}^{2N}\right)$ and let $\hat a(u)$ be the Fourier transform 
  	$$ 
  	\hat a(u) = \left( \mathcal{F}a\right) \left(u \right)  = \int_{\mathbb{R}^{2N}} a(t)e^{-it\cdot u} dt.
  	$$
  	If $a_n = \mathrm{pr}_n\left( a\right) = \sum_{g \in J_n} ga \in \left(\mathbb{T}^{2N}_{\th/m_n^2}\right)$  
  	then from  \eqref{nt_fourier} it follows that 
  	\begin{equation}\label{nt_proj_fourier}
  	\hat a_n(p)= \left(\mathcal F\mathrm{pr}_n\left( a\right)\right) \left( p\right) = \frac{1}{\left(2\pi m_n\right)^{2N}} \hat a\left(\frac{2\pi p}{m_n}\right)
  	\end{equation}
  	where $p \in \mathbb{Z}^{2N}$ and $\hat a_n \in \sS\left(\mathbb{Z}^{2N}\right)$ is the Fourier transform of $a_n$ given by \eqref{nt_fourier}.
  	For simplicity the we use the dilation transformation $E_{2/\th}$ given by \eqref{dilation_eqn} which induces following:
  	\begin{itemize}
  		\item $\theta = 2$, $\star_\th = \times$,
  		\item The equation \eqref{u_l_mapsto_eqn} is replaced with
  		\begin{equation*}
  		U_l \mapsto \left(\th /2 \right)^{N/2} \exp(2\pi~\sqrt{2/\th}~ i \frac{l}{m_n} \cdot-);~ \forall l \in \mathbb{Z}^N.
  		\end{equation*}
  		
  \end{itemize}

  \end{empt}

\begin{lem}\label{nt_long_delta_lem}
Let $a \in \sS\left(\mathbb{R}^{2N}\right)$, and let   $a_\Delta\in \sS\left(\mathbb{R}^{2N}\right)$ be given by
\begin{equation}\label{nt_a_delta_eqn}
a_\Delta\left(x\right)= a(x + \Delta);~\forall x \in \mathbb{R}^{2N}
\end{equation}
where $\Delta \in \mathbb{R}^{2N}$.
For any $m \in \mathbb{N}$ there is  a dependent on $a$ real constant $C_m > 0$  such  that for any $n \in \mathbb{N}$ following condition holds
$$
 \left\|\mathrm{pr}_n\left(a_{\Delta} \times a\right) \right\|  \le \frac{C_m}{\left\| \Delta\right\|^m}.
$$

\end{lem}

\begin{proof}
From the definition of Schwartz functions it follows that for any $f \in \sS\left( \mathbb{R}^{2N} \right)$ and any $m \in \mathbb{N}$ there is $C^f_m$ such that
\begin{equation}\tag{*}
\left|f \left(u\right)\right|<\frac{C^f_m}{\left( 1 + \left\|u\right\|\right)^m }.
\end{equation}
From \eqref{nt_norm_estimation} and \eqref{nt_proj_fourier} it follows that
$$
\left\|\mathrm{pr}_n\left(a_{\Delta} \times a\right) \right\| \le \frac{1}{\left(2\pi m_n\right)^{2N}}~~\sum_{l \in \mathbb{Z}^{2N}} \left| \mathcal{F}\left(a_{\Delta} \times a\right)\left(\frac{2\pi l}{m_n} \right) \right|.
$$
From \eqref{nt_diamond}, \eqref{nt_prod_duality} it follows that
$$
 \frac{1}{\left(2\pi m_n\right)^{2N}}~~\sum_{l \in \mathbb{Z}^{2N}} \left| \mathcal{F}\left(a_{\Delta} \times a\right)\left(\frac{2\pi l}{m_n} \right) \right| = \frac{1}{\left(2\pi m_n\right)^{2N}}~~\sum_{l \in \mathbb{Z}^{2N}} \left| \left( \mathcal{F}a_{\Delta} \diamond \mathcal{F}a\right) \left(\frac{2\pi l}{m_n} \right) \right|=
$$
$$
=\frac{1}{\left(2\pi m_n\right)^{2N}}\sum_{l \in \mathbb{Z}^{2N}} \left|\int \mathcal{F}a\left(\frac{2\pi l}{m_n} - t - \Delta \right)  \mathcal{F}a\left(t \right)e^{-i\frac{2\pi l}{m_n} \cdot J t}  dt\right|\le
$$
$$\le \frac{1}{\left(2\pi m_n\right)^{2N}}\sum_{l \in \mathbb{Z}^{2N}} \int \left| b\left( t + \Delta - \frac{2\pi l}{m_n} \right)  c\left(t \right)e^{-i\frac{2\pi l}{m_n} \cdot J t}  \right|dt=
$$

  \begin{equation*}
=
  \frac{1}{\left(2\pi m_n\right)^{2N}}~~\sum_{l \in \mathbb{Z}^{2N}, ~\left\| \frac{2\pi l}{m_n}\right\|\le \frac{\left\| \Delta\right\|}{2 }}~ \int \left| b\left( t + \Delta - \frac{2\pi l}{m_n} \right)  c\left(t \right)e^{-i\frac{2\pi l}{m_n} \cdot J t}  \right|dt +
  \end{equation*}
  \begin{equation*}
+
  \frac{1}{\left(2\pi m_n\right)^{2N}}~~\sum_{l \in \mathbb{Z}^{2N}, ~\left\| \frac{2\pi l}{m_n}\right\|> \frac{\left\| \Delta\right\|}{2 }}~ \int \left| b\left( t + \Delta - \frac{2\pi l}{m_n} \right)  c\left(t \right)e^{-i\frac{2\pi l}{m_n} \cdot J t}  \right|dt.
  \end{equation*}
where $b\left( u\right) = \mathcal{F}a\left(-u \right)$, $c\left( u\right) = \mathcal{F}a\left(u \right)$.  From (*) it follows that
    \begin{equation*}
     \frac{1}{\left(2\pi m_n\right)^{2N}}~~\sum_{l \in \mathbb{Z}^{2N}, ~\left\| \frac{2\pi l}{m_n}\right\|\le \frac{\left\| \Delta\right\|}{2 }}~ \int \left| b\left( t + \Delta - \frac{2\pi l}{m_n} \right)  c\left(t \right)e^{-i\frac{2\pi l}{m_n} \cdot J t}  \right|dt \le
    \end{equation*}
    \begin{equation*}
   \le \frac{1}{\left(2\pi m_n\right)^{2N}}~~\sum_{l \in \mathbb{Z}^{2N}, ~\left\| \frac{2\pi l}{m_n}\right\|\le \frac{\left\| \Delta\right\|}{2 }}~ \int_{\mathbb{R}^{2N}}\frac{C^{b}_{m}}{\left(1 + \left\|t + \Delta - \frac{2\pi l}{m_n}\right\|  \right)^{m}}~\frac{C^{c}_{2m}}{\left(1 + \left\|t\right\|  \right)^{2m}} dt  =
    \end{equation*}
   \begin{equation*}
   = \frac{1}{\left(2\pi m_n\right)^{2N}}~~\sum_{l \in \mathbb{Z}^{2N}, ~\left\| \frac{2\pi l}{m_n}\right\|\le \frac{\left\| \Delta\right\|}{2 }}~ \int_{\mathbb{R}^{2N}}\frac{C^{b}_{m}}{\left(1 + \left\|t + \Delta - \frac{2\pi l}{m_n}\right\|  \right)^{m}\left(1 + \left\|t\right\|  \right)^{m} }~\frac{C^{c}_{2m}}{\left(1 + \left\|t\right\|  \right)^{m}} dt  \le
    \end{equation*}
       \begin{equation*}
       \le \frac{N^\Delta_{m_n}}{\left(2\pi m_n\right)^{2N}}~~\sup_{l \in \mathbb{Z}^{2N}, ~\left\| \frac{2\pi l}{m_n}\right\|\le \frac{\left\| \Delta\right\|}{2 },~s\in \mathbb{R}^{2N}}~ \frac{C^{b}_{m}C^{c}_{2m}}{\left(1 + \left\|s + \Delta - \frac{2\pi l}{m_n}\right\|  \right)^{m} \left(1 + \left\|s\right\|  \right)^{m}}~\times
   \end{equation*}
        \begin{equation*}
    \times  \int_{\mathbb{R}^{2N}}\frac{1}{\left(1 + \left\|t\right\|  \right)^{m}} dt.
        \end{equation*}
        where $N^\Delta_{m_n} = \left|\left\{l \in \mathbb{Z}^{2N}~| ~\left\| \frac{2\pi l}{m_n}\right\|\le \frac{\left\| \Delta\right\|}{2 }\right\}\right|$. The number $N^\Delta_{m_n}$ can be estimated as a number of points with integer coordinates inside $2N$-dimensional cube
        $$
        N^\Delta_{m_n} < \left\|\frac{m_n \Delta}{2\pi}\right\|^{2N}.
        $$  
        
        If $m \ge 2N + 1$ then integral $ \int_{\mathbb{R}^{2N}}\frac{1}{\left(1 + \left\|t\right\|  \right)^{m}} dt$ is convergent, whence
        
        $$
         \frac{1}{\left(2\pi m_n\right)^{2N}}~\sum_{l \in \mathbb{Z}^{2N}, ~\left\| \frac{2\pi l}{m_n}\right\|\le \frac{\left\| \Delta\right\|}{2 }}~~\left|\int b\left( t + \Delta - \frac{2\pi l}{m_n} \right)  c\left(t \right)e^{-i\frac{2\pi l}{m_n} \cdot J t}  dt\right| \le
        $$
        $$
        \le C_1'  \sup_{l \in \mathbb{Z}^{2N}, ~\left\| \frac{2\pi l}{m_n}\right\|\le \frac{\left\| \Delta\right\|}{2 },~s\in \mathbb{R}^{2N}}~  \frac{ \left\|\Delta\right\|^{2N} }{\left(1 + \left\|s + \Delta - \frac{2\pi l}{m_n}\right\|  \right)^{m} \left(1 + \left\|s\right\|  \right)^{m}}
        $$
        where 
        $$
        C_1' = \frac{1}{\left(2\pi \right) ^{4N}} C^{b}_mC^c_{2m}\int_{\mathbb{R}^{2N}}\frac{1}{\left(1 + \left\|t\right\|  \right)^{m}} dt.
        $$
        If $x,y \in \mathbb{R}^{2N}$  then from the triangle inequality it follows that $\left\|x + y\right\|>  \left\|y\right\| - \left\|x\right\|$, whence 
        $$
        \left(1 + \left\|x\right\| \right)^m \left(1 + \left\|x+ y\right\| \right)^m \ge \left(1 + \left\|x\right\| \right)^m \left(1 +  \max\left(0, \left\|y\right\| - \left\|x\right\| \right)\right)^m
        $$
        If $ \left\|x\right\| \le \frac{\left\|y\right\|}{2}$ then $\left\|y\right\| - \left\|x\right\| \ge \frac{\left\|y\right\|}{2}$ and 
        \begin{equation}\tag{**}
       \left(1 + \left\|x\right\| \right)^m \left(1 + \left\|x+ y\right\| \right)^m > \left( \frac{\left\|y\right\|}{2}\right)^m. 
        \end{equation}
    
        Clearly if $ \left\|x\right\| > \frac{\left\|y\right\|}{2}$ then condition (**) also holds, whence (**) is always true.
 It follows from the (**) that
        $$
        \mathrm{inf}_{l \in \mathbb{Z}^{2N}, ~\left\| \frac{2\pi l}{m_n}\right\|\le \frac{\left\| \Delta\right\|}{2 },~s\in \mathbb{R}^{2N}} \left(1 + \left\|s + \Delta - \frac{2\pi l}{m_n}\right\|  \right)^{m} \left(1 + \left\|s\right\|  \right)^{m} >\left\|\frac{\Delta}{4}\right\|^m.
          $$
   If $m > 2N+1$   then from above equations it follows that
   $$
     \frac{1}{\left(2\pi m_n\right)^{2N}}~\sum_{l \in \mathbb{Z}^{2N}, ~\left\| \frac{2\pi l}{m_n}\right\|\le \frac{\left\| \Delta\right\|}{2 }}~~\left|\int b\left( t + \Delta - \frac{2\pi l}{m_n} \right)  c\left(t \right)e^{-i\frac{2\pi l}{m_n} \cdot J t}  dt\right| \le \frac{C_1}{\left\| \Delta\right\|^{m}} 
    $$  
    where $C_1=C'_1/4^{m}$. 
    
     Clearly 
          \begin{equation*}
    \frac{1}{\left(2\pi m_n\right)^{2N}}~\sum_{l \in \mathbb{Z}^{2N}, ~\left\| \frac{2\pi l}{m_n}\right\|> \frac{\left\| \Delta\right\|}{2 }}~~\left|\int b\left( t + \Delta - \frac{2\pi l}{m_n} \right)  c\left(t \right)e^{-i\frac{2\pi l}{m_n} \cdot J t}  dt\right|=
   \end{equation*} 
          \begin{equation*}
          \frac{1}{\left(2\pi m_n\right)^{2N}}~\sum_{l \in \mathbb{Z}^{2N}, ~\left\| \frac{2\pi l}{m_n}\right\|> \frac{\left\| \Delta\right\|}{2 }}~~\left|\int \left( \left( b\left( t + \Delta - \frac{2\pi l}{m_n} \right)\right)^* \right)^*   c\left(t \right)e^{-i\frac{2\pi l}{m_n} \cdot J t}  dt\right|=
          \end{equation*}  
   $$
   \frac{1}{\left(2\pi m_n\right)^{2N}}~~\sum_{l \in \mathbb{Z}^{2N}, ~\left\| \frac{2\pi l}{m_n}\right\|> \frac{\left\| \Delta\right\|}{2 }}~ \left|\left(  \left(  b\left( \bullet+ \Delta-\frac{2\pi l}{m_n}\right)\right)^* ,~ c\left(\bullet\right)e^{-i\frac{2\pi l}{m_n}\cdot J\bullet} \right)  \right|
   $$
   where $\left( \cdot, \cdot \right) $ (resp. $^*$) means the given by \eqref{nt_l2_scalar_prod_eqn} scalar product (resp. complex adjunction).
   From \eqref{scalar_fourier_inv_eqn} it follows that 
   $$
      \frac{1}{\left(2\pi m_n\right)^{2N}}~~\sum_{l \in \mathbb{Z}^{2N}, ~\left\| \frac{2\pi l}{m_n}\right\|> \frac{\left\| \Delta\right\|}{2 }}~ \left|\left(  \left(  b\left( \bullet+ \Delta-\frac{2\pi l}{m_n}\right)\right)^* ,~ c\left(\bullet\right)e^{-i\frac{2\pi l}{m_n}\cdot J\bullet} \right)  \right|=
      $$
  $$
  \frac{1}{\left(2\pi m_n\right)^{2N}}~~\sum_{l \in \mathbb{Z}^{2N}, ~\left\| \frac{2\pi l}{m_n}\right\|> \frac{\left\| \Delta\right\|}{2 }}~ \left|\left(  \mathcal{F}\left(  b\left( \bullet+ \Delta-\frac{2\pi l}{m_n}\right)\right)^* ,~ \mathcal{F}\left( c\left(\bullet\right)e^{-i\frac{2\pi l}{m_n}\cdot J\bullet}\right)  \right)  \right|=
  $$
    
   \begin{equation*}
  = \frac{1}{\left(2\pi m_n\right)^{2N}}~~\sum_{l \in \mathbb{Z}^{2N}, ~\left\| \frac{2\pi l}{m_n}\right\|> \frac{\left\| \Delta\right\|}{2 }}~ \left|\int_{\mathbb{R}^{2N}} 
  \left(\mathcal{F}\left( b\right)^*\left( \bullet + \Delta -\frac{2\pi l}{m_n}\right)\right)^*\left(u\right)\mathcal{F}\left(c\left(\bullet\right)e^{-i\frac{2\pi l}{m_n}\cdot J\bullet}\right)\left(u\right) du \right|\le
   \end{equation*}
    
  \begin{equation*}
  \le \frac{1}{\left(2\pi m_n\right)^{2N}}~~\sum_{l \in \mathbb{Z}^{2N}, ~\left\| \frac{2\pi l}{m_n}\right\|> \frac{\left\| \Delta\right\|}{2 }}~ \int_{\mathbb{R}^{2N}}\left| e^{-i\left(\Delta - \frac{2\pi l}{m_n}\right) \cdot u}\left(\left(\mathcal{F}\left( b\right)^*\right)\right)^*\left( u\right) \mathcal{F}\left(c\right)\left(u+J\frac{2\pi l}{m_n}\right)\right| du \le
   \end{equation*} 
  \begin{equation*}
 \le \frac{1}{\left(2\pi m_n\right)^{2N}}~~\sum_{l \in \mathbb{Z}^{2N}, ~\left\| \frac{2\pi l}{m_n}\right\|> \frac{\left\| \Delta\right\|}{2 }}~ \int_{\mathbb{R}^{2N}}\frac{C^{\left(\mathcal{F}\left( b\right)^*\right)^*}_{3m}}{\left(1 + \left\|u\right\|  \right)^{3m}}\frac{C^{\mathcal{F}\left(c\right)}_{2m}}{\left(1 + \left\|u-J\frac{2\pi l}{m_n}\right\|  \right)^{2m}} du  \le
 \end{equation*}
 \begin{equation*}
 \le \frac{1}{\left(2\pi m_n\right)^{2N}}~~\sup_{l \in \mathbb{Z}^{2N}, ~\left\| \frac{2\pi l}{m_n}\right\|> \frac{\left\| \Delta\right\|}{2 },~ s \in \mathbb{R}^{2N}}~ \frac{C^{\left(\mathcal{F}\left( b^*\right)\right)^*}_{3m}}{\left(1 + \left\|s\right\|  \right)^{m}}\frac{C^{\mathcal{F}\left(c\right)}_{2m}}{\left(1 + \left\|s-J\frac{2\pi l}{m_n}\right\|  \right)^{m}}
 \times
 \end{equation*}
 
 \begin{equation*}
 \times \sum_{l \in \mathbb{Z}^{2N}, ~\left\| \frac{2\pi l}{m_n}\right\|> \frac{\left\| \Delta\right\|}{2 }}     
 \int_{\mathbb{R}^{2N}}\frac{1}{\left(1 + \left\|u-J\frac{2\pi l}{m_n}\right\|  \right)^{m}\left(1 + \left\|u\right\|  \right)^{m}}\frac{1}{\left(1 + \left\|u\right\|  \right)^{m}} du.
 \end{equation*}
 Since we consider the asymptotic dependence $\left\|\Delta\right\|\to \infty$ only  large values of $\left\|\Delta\right\|$ are interested, so we can suppose that  $\left\|\Delta\right\| > 2$.
If $\left\|\Delta\right\| > 2$ then from $\left\|\frac{2\pi l}{m_n}\right\|> \frac{\left\| \Delta\right\|}{2 }$ it follows that $\left\|\frac{2\pi l}{m_n}\right\| > 1$, and  from (**) it follows that
 $$
 \left(1 + \left\|u\right\|  \right)^{m}\left(1 + \left\|u-J\frac{2\pi l}{m_n}\right\|  \right)^{m} > \left(\frac{\left\|J\frac{2\pi l}{m_n}\right\|}{2} \right)^m, 
 $$
 $$
 \inf_{l \in \mathbb{Z}^{2N}, ~\left\| \frac{2\pi l}{m_n}\right\|> \frac{\left\| \Delta\right\|}{2 },~ s \in \mathbb{R}^{2N}}~ \left(1 + \left\|s\right\|  \right)^{m}\left(1 + \left\|s-J\frac{2\pi l}{m_n}\right\|  \right)^{m} > \left\|\frac{\Delta}{4}\right\|^m
 $$
 whence
 $$
 \frac{1}{\left(2\pi m_n\right)^{2N}}~~\sum_{l \in \mathbb{Z}^{2N}, ~\left\| \frac{2\pi l}{m_n}\right\|> \frac{\left\| \Delta\right\|}{2 }}~ \left|\left(  \left(  b\left( \bullet+ \Delta-\frac{2\pi l}{m_n}\right)\right)^* ,~ c\left(\bullet\right)e^{i\frac{2\pi l}{m_n}\cdot J\bullet} \right)  \right|\le
 $$
 $$
 \le \frac{1}{\left(2\pi m_n\right)^{2N}} \frac{C_2'}{\left\| \Delta\right\|^m} \sum_{l \in \mathbb{Z}^{2N}, ~\left\| \frac{2\pi l}{m_n}\right\|> 1}\int_{\R^{2N}}\frac{1}{\left\|J\frac{2\pi l}{m_n}\right\|^m}\frac{1}{\left(1+\left\|u\right\| \right)^m }=
 $$
 $$
 \frac{C_2'}{\left\| \Delta\right\|^m}  \frac{1}{\left(2\pi m_n\right)^{2N}} \left(  \sum_{l \in \mathbb{Z}^{2N}, ~\left\| \frac{2\pi l}{m_n}\right\|> 1}~ \frac{1}{\left\|J\frac{2\pi l}{m_n}\right\|^m}  \right) \left(  \int_{\R^{2N}}\frac{1}{\left(1+\left\|u\right\| \right)^m }du\right) 
 $$
 where $C'_2=C^{\left(\mathcal{F}\left( b^*\right)\right)^*}_{3m}C^{\mathcal{F}\left(c\right)}_{2m}$.
 If $m \ge 2N + 1$ then  $\int_{\mathbb{R}^{2N}}\frac{1}{\left(1 + \left\|u\right\|  \right)^{m}} du$ is convergent. Any sum can be represented as an integral of step function, in particular  following condition holds
 $$
 \frac{1}{\left(2\pi m_n\right)^{2N}} \sum_{l \in \mathbb{Z}^{2N},\left\| \frac{2\pi l}{m_n}\right\|> 1}~ \frac{1}{\left\|J\frac{2\pi l}{m_n}\right\|^m} = \int_{\R^{2N} - \left\{x \in \R^{2N}~|~\left\|x\right\|> 1 \right\}} f_{m_n}\left( x\right) dx
 $$
 where $f_{m_n}$ is a multidimensional step function such that
 $$
 f_{m_n}\left(J\frac{2\pi l}{m_n} \right) = \frac{1}{\left(2\pi \right)^{2N} }\frac{1}{\left\|J\frac{2\pi l}{m_n}\right\|^m}
 $$
 From $$f_{m_n}\left(x\right)  < \frac{2}{\left(2\pi \right)^{2N} \left\|x\right\|^m}$$ it follows that
$$
\int_{\R^{2N} - \left\{x \in \R^{2N}~|~\left\|x\right\|> 1 \right\}} f_{m_n}\left( x \right)dx < \int_{\R^{2N} - \left\{x \in \R^{2N}~|~\left\|x\right\|> 1 \right\}} \frac{2}{\left(2\pi \right)^{2N} \left\|x\right\|^m}dx.
$$
If $m > 2N$ then the integral $$\int_{\R^{2N} - \left\{x \in \R^{2N}~|~\left\|x\right\|> 1 \right\}} \frac{2}{\left(2\pi \right)^{2N} \left\|x\right\|^m}dx$$ is convergent, whence
 $$
 \frac{1}{\left(2\pi m_n\right)^{2N}} \sum_{l \in \mathbb{Z}^{2N},\left\| \frac{2\pi l}{m_n}\right\|> 1}~ \frac{1}{\left\|J\frac{2\pi l}{m_n}\right\|^m} < C''_2= \int_{\R^{2N} - \left\{x \in \R^{2N}~|~\left\|x\right\|> 1 \right\}} \frac{2}{\left(2\pi \right)^{2N} \left\|x\right\|^m}dx.
 $$
 From above equations it follows that
 $$
 \frac{1}{\left(2\pi m_n\right)^{2N}}~~\sum_{l \in \mathbb{Z}^{2N}, ~\left\| \frac{2\pi l}{m_n}\right\|> \frac{\left\| \Delta\right\|}{2 }}~ \left|\int_{\mathbb{R}^{2N}} b\left( w-\frac{2\pi l}{m_n}\right)c\left(w+\Delta\right)e^{i\frac{2\pi l}{m_n}\cdot Jw} dt \right| \le \frac{C_2}{\left\| \Delta\right\|^m} 
 $$  
 where $m \ge 2N+1$ and $C_2 = C'_2C''_2\int_{\R^{2N}}\frac{1}{\left(1+\left\|u\right\| \right)^m }du$. In result for any $m > 0$ there is $C_m \in \mathbb{R}$ such that 
 \begin{equation}
 \begin{split}
 \left\|\mathrm{pr}_n\left(a_{\Delta} \times a\right) \right\|  \le 
 \frac{1}{\left(2\pi m_n\right)^{2N}}\times\\\times\sum_{l \in \mathbb{Z}^{2N}}  \left|\int_{\mathbb{R}^{2N}} b\left(t + \Delta - \frac{2\pi l}{m_n}\right)c\left(w+\Delta\right)e^{i\frac{2\pi l}{m_n}\cdot Jw} dt \right| \le \frac{C_m}{\left\| \Delta\right\|^m}
 \end{split}.
 \end{equation}
 \end{proof}
 In the following lemma we use $ab$ instead $a \times b$.
 \begin{lem}\label{nt_w_spec_lem} 
 Any positive element $\overline{a}$ in $\Coo_{0}\left( \mathbb{R}^{2N}_\th\right) \approx \sS\left(\mathbb{R}^{2N} \right) $  is  special.	
 \end{lem}
 \begin{proof}
 	If
 	$$
 	a_n = \sum_{g \in J_n  }g \overline{a}
 	$$
 	$$
 	b_n = \sum_{g \in J_n  }  g \overline{a}^2
 	$$
  where $J_n = \ker\left( \mathbb{Z}^{2N} \to \mathbb{Z}^{2N}_{m_n}\right)= m_n\Z^n$,	then
 \begin{equation}\tag{*}
 	a^2_n - b_n = \sum_{g \in J_n  }g\overline{a} ~ \left(  \sum_{g' \in J_n  \backslash \{g\} }g'' \overline{a}\right) .
 \end{equation}
 	From \eqref{nt_group_action} and dilation transformation $E_{\th/2}$ it follows that $g \overline{a}= \overline{a}_{g 2\pi~\sqrt{2/\th}}$ where  $\overline{a}_{g 2\pi~\sqrt{2/\th}}\left(x \right)  = \overline{a}\left( x +g 2\pi~\sqrt{2/\th}\right)$ for any $x \in \R^{2N}$. Hence the equation (*) is equivalent to
	$$
	a^2_n - b_n = \sum_{g \in J_n  }\overline{a}_{g 2\pi~\sqrt{2/\th}}  \sum_{g' \in J_n  \backslash \{g\} } \overline{a}_{g' 2\pi~\sqrt{2/\th}} = \sum_{g \in \mathbb{Z}^{2N}  }\overline{a}_{m_ng 2\pi~\sqrt{2/\th}}  \sum_{g' \in \mathbb{Z}^{2N}  \backslash \{g\} } \overline{a}_{m_ng' 2\pi~\sqrt{2/\th}}=
	$$
	$$
	=\mathrm{pr}_n\left( \overline{a}  \sum_{g \in \mathbb{Z}^{2N}  \backslash \{0\} } \overline{a}_{m_ng 2\pi~\sqrt{2/\th}} \right) 
	$$
	where $\mathrm{pr}_n$ is given by \eqref{nt_pr_k_eqn}. From the Lemma \ref{nt_long_delta_lem} it follows that there is $C \in \mathbb{R}$ such that
	$$
	\left\|\mathrm{pr}_n\left( aa_\Delta\right) \right\| < \frac{C}{\left\| \Delta\right\|^m}.
	$$
	From the triangle inequality it follows that
	$$
	\left\|a^2_n - b_n\right\|=\left\|\mathrm{pr}_n\left( \overline{a}  \sum_{g \in \mathbb{Z}^{2N}  \backslash \{0\} } \overline{a}_{m_ng 2\pi~\sqrt{2/\th}} \right)\right\| \le
	$$
	$$
	\le \sum_{g' \in \mathbb{Z}^{2N}  \backslash \{0\} } \left\|\mathrm{pr}_n\left( \overline{a}  \sum_{g \in \mathbb{Z}^{2N}  \backslash \{0\} } \overline{a}_{m_ng 2\pi~\sqrt{2/\th}} \right)\right\|   \le \sum_{g \in \mathbb{Z}^{2N}  \backslash \{0\} } \frac{C}{\left( m_ng 2\pi~\sqrt{2/\th}\right)^m }.
	$$
	If $m > 2N+1$ then the series 
	$$
	C' = \sum_{g \in \mathbb{Z}^{2N}  \backslash \{0\} } \frac{C}{\left( g 2\pi~\sqrt{2/\th}\right)^m }
	$$
	is convergent and 
	$$
	\sum_{g \in \mathbb{Z}^{2N}  \backslash \{0\} } \frac{C}{\left( m_ng 2\pi~\sqrt{2/\th}\right)^m } = \frac{C'}{\left(m_n\right)^m }
	$$
	If $\varepsilon > 0$ is a small number and $N\in \mathbb{N}$ is such $N \ge \sqrt[m]{\frac{C'}{\varepsilon}}$ then from above equations it follows that for any $n \ge N$ following condition holds
		$$
		\left\|a^2_n - b_n\right\| < \eps.
		$$
 \end{proof}
\paragraph{} In the following lemma we use $\star_\th$ notion for the star product.
\begin{lem}\label{nt_l2_spec_lem}
 Use the notation of the proof of the Lemma \ref{nt_w_spec_lem}. The special element $\widetilde{a}\in \varinjlim C\left(\mathbb{T}^{2N}_{\th/m_n^2}\right) $ lies in  $L^2 \left(\R^{2N}_\th \right)$ in sense of the Definition  \ref{nt_l_2_est_defn}. Moreover if $b_n = \sum_{g \in J_n}g\left( \widetilde{a}\star_\th\widetilde{a}\right) \in C\left(\mathbb{T}^{2N}_\th\right)$ then
$$
\left\| \widetilde{a}\right\|_2 = \sqrt{\frac{1}{\left(2\pi m_n\right)^{2N} }\tau\left(b_n \right)} < \infty
$$
where $\tau$ is the tracial state on $C\left(\mathbb{T}^{2N}_{\th/m_n^{2}}\right)$ given by \eqref{nt_state_eqn}, \eqref{nt_state_integ_eqn} and $\left\|\cdot\right\|_2$ is given by \eqref{nt_l2_norm_eqn}.
\end{lem}
\begin{proof}
From \eqref{nt_state_integ_eqn} it follows that
\begin{equation}\tag{*}
\tau\left(b_n \right) = \frac{1}{\left(2\pi m_n \right)^{2N}}  \int_{\mathcal X_n} b_n dx_n   < \infty
\end{equation}
where $dx_n$ is associated with Lebesgue measure on $\R^{2N}$ and a homeomorphism $\mathcal X_n \approx \R^{2N}/2\pi m_n\Z^{2N}$. 
From \eqref{nt_tracial_prop} and (*) it follows that
\begin{equation*}
\begin{split}
\left( \left\| \widetilde{a}\right\|_2\right)^2 =\left\| \widetilde{a}\widetilde{a}^*\right\|_2=\left\| \widetilde{a}\widetilde{a}\right\|_2 = \int_{\R^{2N}}\widetilde{a}\left(x \right) \widetilde{a}\left( x\right) ~ dx=\\= \int_{\R^{2N}}\left( \widetilde{a}\star_\th\widetilde{a} \right)\left(x \right)  dx  
=\frac{1}{\left(2\pi m_n \right)^{2N}}  \int_{\mathcal X_n} b_n~ dx_n =\tau\left(b_n \right) < \infty 
\end{split}
\end{equation*}

whence $\left\| \widetilde{a}\right\|_2 <  \infty$, therefore $\widetilde{a}\in L^2\left(\R^{2N}_\th \right)$. 
\end{proof}
\begin{empt}
	We need the notion of essential support  functions described in \cite{elliot:an}. For a given Borel measure $\mu$ let $f$ be a Borel-measurable function. Consider the collection $\widetilde{\Om}$ of open subsets $\om$ with the property that $f\left(x \right)=0$ for $\mu$-almost every $x \in \om$ and let the open set $\om^*$ be the union of all $\om$'s in  $\widetilde{\Om}$. Now we define \textit{essential support} of $f$ or $\mathfrak{ess~supp}\left( f\right)$ to be the complement of $\om^*$. There is the following formula
	$$
	\mathfrak{ess~supp}\left( f\right) = \mathcal X \backslash \bigcup \left\{\om \in \mathcal X~|~ \om \text{ is open and } f=0~\mu-\text{almost everywhere in }\om\right\}.
	$$
	If $x \notin \mathfrak{ess~supp}\left( f\right)$ then there is an open neighborhood $\mathcal U$ of $x$ and representative $\varphi$ of $f$ such that $\varphi\left({\mathcal U}\right) = \{0\}$. 
\end{empt}
 \begin{thm}
 	The sequence $\mathfrak{S}_{C\left(\mathbb{T}^{2N}_\th\right)}= \left\{ 		C\left(\mathbb{T}^{2N}_\th\right) \to	C\left(\mathbb{T}^{2N}_{\th/m_1^{2}}\right) \to C\left(\mathbb{T}^{2N}_{\th/m_2^{2}}\right) \to... \right\}$ is regular and  $G\left( \varinjlim C\left(\mathbb{T}^{2N}_{\th/m_n^{2}}\right)  ~|~C\left(\mathbb{T}^{2N}_{\th}\right)\right)  \approx \Z^{2N}$.
 
 \end{thm}
 \begin{proof}
 	We should proof (a) and (b) of the Definition \ref{regular_defn}. 
 	\newline
 	(a)
 	If $u_1,..., u_n \in C\left(\mathbb{T}^{2N}_{\th}\right)$ (resp. $v_1,..., v_{2N} \in C\left(\mathbb{T}^{2N}_{\th/m_n^{2}}\right)$) are  generators then from \eqref{nt_cov_eqn} it follows that
 	$$
 	v_j^{m_n}= u_j;~ j = 1,...,2N.
 	$$
 	If $g \in G\left( \varinjlim C\left(\mathbb{T}^{2N}_{\th/m_n^{2}}\right)  ~|~C\left(\mathbb{T}^{2N}_{\th}\right)\right)$ then 
 	\begin{equation}\tag{*}
 	ga= a; ~ \forall a\in C\left(\mathbb{T}^{2N}_{\th/m_n^{2}}\right) .
 	\end{equation}
 	Since $g$ is a homomorphism the following condition holds
 	\begin{equation}\tag{**}
 	\left(g v_j \right)^{m_n} = g v_j^{m_n}=gu_j=u_j\in C\left(\mathbb{T}^{2N}_{\th}\right).
 	\end{equation}
 	From \ref{nt_grad_sec} it follows that there is $\Z^{2N}$-grading of $C\left(\mathbb{T}^{2N}_\th\right)$ and induced $\left(\frac{\Z}{m_n} \right)^{2N}$-grading of  $C\left(\mathbb{T}^{2N}_{\th/{m_n^2}}\right)$. Moreover if $$\Q'=\left\{\frac{a}{b}\in \Q~|~a \in \Z, ~b \in \left\{m_n\right\}_{n \in \N}\right\}$$ then there is $\Q'^{2N}$-grading of $\widehat{A} =  C\left(\mathbb{T}^{2N}_{\th/m_n^{2}}\right)$, i.e. $\widehat{A} = \bigoplus_{q \in \Q'^{2N}} \widehat{A}_q$ where $\bigoplus$ means the $C^*$-norm completion of the direct sum. This grading depends on generators $u_1,...,u_n \in U\left( C\left(\mathbb{T}^{2N}_\th\right)\right) $.
 	Since $\widehat{G}=G\left( \varinjlim C\left(\mathbb{T}^{2N}_{\th/m_n^{2}}\right)~|~C\left(\mathbb{T}^{2N}_{\th}\right)\right)$ trivially acts on $C\left(\mathbb{T}^{2N}_\th\right)$ the grading is $\widehat{G}$-invariant, i.e. $\widehat{G} \widehat{A}_q = \widehat{A}_q$ for any $q \in \Q'^{2N}$. If 
 	$$q=\left(0,..., \underbrace{\frac{1}{m_n}}_{j^{\text{th}}-\text{place}},...,0 \right)$$ then $gv_j, v_j \in  \widehat{A}_q$. However $\widehat{A}_q$ is a one-dimensional $\C$-space, so there is $c \in \C$ such that $g v_j = cv_j$. It follows that $g$ transforms a set of unitary generators of $C\left(\mathbb{T}^{2N}_{\th/m_n^{2}}\right)$ to the set of generators, 
 	or equivalently $$G\left( \varinjlim C\left(\mathbb{T}^{2N}_{\th/m_n^{2}}\right)  ~|~C\left(\mathbb{T}^{2N}_{\th}\right)\right)~C\left(\mathbb{T}^{2N}_{\th/m_n^{2}}\right) = C\left(\mathbb{T}^{2N}_{\th/m_n^{2}}\right).$$
 	\newline
 	(b)
 	Let us $\mathfrak{S}_{\T^{2N}}	= \mathfrak{S} = \left\{\T^{2N} = \mathcal{X}_0 \xleftarrow{}... \xleftarrow{} \mathcal{X}_n \xleftarrow{} ...\right\}$ be a topological finite covering sequence such that
 	\begin{itemize}
 		\item $\mathcal{X}_n \approx \T^{2N}$ for any $n \in \mathbb{N}$ (it is known \cite{spanier:at} that $\pi_1\left( \T^{2N}, x_0\right)\approx \Z^{2N}$ );
 		\item If $f_n : \mathcal{X}_n \to \mathcal{X}_0$ then corresponding homomorphism of fundamental groups $\pi_1\left(f_n \right): \pi_1\left(\mathcal{X}_n, x_0\right) \to \pi_1\left( \mathcal{X}_0, f_n\left( x_0\right) \right)$ is given by an integer valued diagonal matrix 
 		$$
 		\begin{pmatrix}
 		m_n& 0 &\ldots & 0\\
 		0& m_n &\ldots & 0\\
 		\vdots& \vdots &\ddots & \vdots\\
 		0& 0 &\ldots & m_n
 		\end{pmatrix}.
 		$$
 	\end{itemize}
 	
 	Clearly there are covering projections $\widetilde{\pi}^n:\R^{2N}\to \mathcal X_n$   such that the following diagram
 	\newline
 	\begin{tikzpicture}
 	\matrix (m) [matrix of math nodes,row sep=3em,column sep=4em,minimum width=2em]
 	{
 		& \R^{2N}	&     \\
 		\mathcal X_m	&  & \mathcal X_n  \\};
 	\path[-stealth]
 	(m-1-2) edge node [left] {$\widetilde{\pi}^m~$} (m-2-1)
 	(m-1-2) edge node [right] {$~\widetilde{\pi}^n~$} (m-2-3)
 	(m-2-1) edge node [left] {}  (m-2-3);
 	
 	\end{tikzpicture}
 	\newline
 	is commutative. Let $\widetilde{x}_0 \in \R^{2N}$, and let $\widetilde{\mathcal U} \subset \R^{2N}$ is an open connected subset such that $\widetilde{x}_0 \in \widetilde{\mathcal U}$ and the restriction $\widetilde{\pi}^n|_{\widetilde{\mathcal U}}: \widetilde{\mathcal U} \to \widetilde{\pi}^n\left( \widetilde{\mathcal U}\right) $ is a homeomorphism.  Denote by ${\mathcal U}_n = \widetilde{\pi}^n\left(\widetilde{\mathcal U} \right)$,   $\overline{\mathcal U}_n = \widetilde{\pi}^n\left(\overline{\widetilde{\mathcal U}}\right) $. Let $\widetilde{a} \in \sS\left(\R^{2N} \right)\approx \Coo_0\left( \R^{2N}_\th\right) $ be a positive  element of $\Coo_0\left( \R^{2N}_\th\right)$ such that $\widetilde{a}\left( \widetilde{x}_0\right) = 1$ and  $\mathfrak{supp}\left( \widetilde{a}\right) \subset \widetilde{\mathcal U}$. If $\widehat{g} \in  	G\left( \varinjlim C\left(\mathbb{T}^{2N}_{\th/m_n^{2}}\right)  ~|~C\left(\mathbb{T}^{2N}_{\th}\right)\right)$ then from  of the Lemma  \ref{ext_group_action_lem} it follows that $\widehat{g}\widetilde{a}\in \left( \varinjlim C\left(\mathbb{T}^{2N}_{\th/m_n^{2}}\right)\right)''$ is a special element. From (a) it follows that there is a the sequence $$\left\{g_n \in G\left( \varinjlim C\left(\mathbb{T}^{2N}_{\th/m_n^{2}}\right)  ~|~C\left(\mathbb{T}^{2N}_{\th}\right)\right)\right\}_{n \in \N}$$ such that
 	\begin{equation}\tag{****}
 	\widehat{g}a = g_na;~ \forall a \in C\left(\mathbb{T}^{2N}_{\th/m_n^{2}}\right),~ \forall n \in \N
 	\end{equation}
 	
 	For any $n \in \N$ there is the natural epimorphism
 	$$
 	G\left(\R^{2N}~|~\T^{2N}\right) \to G\left(\mathcal X_n~|~\T^{2N}\right)
 	$$
 	where $G\left(\R^{2N}~|~\mathcal \T^{2N}\right) = \Z^{2N}$ and $G\left(\mathcal X_n~|~\T^{2N}\right) = \Z^{2N}_{m_n}$ and there is the natural isomorphism
 	$$
 	G\left(\mathcal X_n~|~\T^{2N}\right) \approx G\left(  C\left(\mathbb{T}^{2N}_{\th/m_n^{2}}\right)  ~|~C\left(\mathbb{T}^{2N}_{\th}\right)\right).
 	$$
 	There is the action of $\Z^{2N}$ (resp. $\Z^{2N}_{m_n}$) on $\sS\left( \R^{2N}\right)$ (resp. $\Coo\left(\mathcal X_n\right)$) arising from the action of $\Z^{2N}$ on $\R^{2N}$. If $J_n = \ker\left(\Z^{2N} \to \Z^{2N}_{m_n}\right)$ then
 	$$
 	a_n = \sum_{g \in J_n}g \widetilde{a} = \mathfrak{Desc}_{\widetilde{\pi}^n}\left(\widetilde{a} \right) \in \Coo\left(\mathcal X_n \right)\approx \Coo\left(\mathbb{T}^{2N}_{\th/m_n^{2}}\right)  
 	$$ 
 	If $\widetilde{a}_{\widehat{g}}= \widehat{g}\widetilde{a}$ then from the proof of the Lemma \ref{ext_group_action_lem} it follows that $\widetilde{a}_{\widehat{g}}$ is special and from the Lemma \ref{nt_l2_spec_lem} it follows that $\widetilde{a}_{\widehat{g}}$ corresponds to a Borel-measured function on $\R^{2N}$. From (****) it follows that
 \begin{equation*}
 	\sum_{\widehat{g}'\in J_n}\widehat{g}' \widetilde{a}_{\widehat{g}} = g_na_n;
 \end{equation*}
 From the above equation and proprieties of essential support it follows that
 $$
 \widetilde{\pi}^n\left(\mathfrak{ess~supp}\left(\widetilde{a}_{\widehat{g}} \right)  \right) = \mathfrak{ess~supp}\left(g_na_n \right)  =\mathfrak{supp}\left(g_na_n \right) \subset g_n{\mathcal U}_n.
 $$
 Since $\mathfrak{ess~supp}\left(\widetilde{a}_{\widehat{g}}\right) $ is not empty there is nonempty special subset $\widetilde{\mathcal U}_{\widehat{g}}\subset \R^{2N}$ such that for any $n \in \N$ there is a homeomorphism
 	$$
 	\widetilde{\pi}^n|_{{\widetilde{\mathcal U}}_{\widehat{g}}}: {\widetilde{\mathcal U}}_{\widehat{g}} \to g_n {\mathcal U}_n.
 	$$
  There is the unique $g \in G\left( \R^{2N}~|~\T^{2N}\right)=\mathbb{Z}^{2N}$ such that $\widetilde{\mathcal U}_{\widehat{g}} = g \widetilde{\mathcal U}$. From  \eqref{nt_group_action}, \eqref{nt_action_desc} it follows that $g$ uniquely defines *-automorphism of $\varinjlim C\left(\T^{2N}_{\th/m_n} \right)$ such that
 	$$
 	g|_{C\left(\T^{2N}_{\th/m_n} \right)}=g_n = \widehat{g}|_{C\left(\T^{2N}_{\th/m_n} \right)}; ~\forall n \in \N
 	$$
 	From the above equation it follows that $\widehat{g} = g$. Hence $G\left(\varinjlim C\left(\T^{2N}_{\th/m_n} \right)~|~ \T^{2N}_{\th}\right) \approx \Z^{2N}$. The homomorphism $G\left(\varinjlim C\left(\T^{2N}_{\th/m_n} \right)~|~ \T^{2N}_{\th}\right) \to G\left( C\left(\T^{2N}_{\th/m_n} \right)~|~ \T^{2N}_{\th}\right)$ is surjective because $\Z^{2N}\to \Z^{2N}_{m_n}$ is surjective.
 
 \end{proof}
 
 \begin{thm}
If  $\mathfrak{S}_{C\left(\mathbb{T}^{2N}_\th\right)}= \left\{ 		C\left(\mathbb{T}^{2N}_\th\right) \to	C\left(\mathbb{T}^{2N}_{\th/m_1^{2}}\right) \to C\left(\mathbb{T}^{2N}_{\th/m_2^{2}}\right) \to... \right\}$ then 
  $$\varprojlim \mathfrak{S}_{C\left(\mathbb{T}^{2N}_\th\right)} = C_0\left(\R^{2N}_\th\right).$$
  \end{thm}
 \begin{proof}
 	 From  the Lemma \ref{nt_w_spec_lem} it follows that any positive $\widetilde{a}\in \Coo_0\left(\R^{2N}_\th \right)$ is  special, whence $\varprojlim \mathfrak{S}_{C\left(\mathbb{T}^{2N}_\th\right)} \subset C_0\left(\R^{2N}_\th \right)$, because $C_0\left(\R^{2N}_\th \right)$ is the operator norm completion of $\Coo_0\left(\R^{2N}_\th \right)$ . From the Lemma \ref{nt_l2_spec_lem} it follows that any special element $\widetilde{a} \in \left(\varinjlim C\left(\mathbb{T}^{2N}_{\th/m_n^{2}}\right) \right)''$ lies in  $L^2 \left(\R^{2N}_\th \right)$. The algebra $\Coo_0\left(\R^{2N}_\th \right)\approx \sS\left( \R^{2N}\right) $ is a dense subset of $L^2 \left(\R^{2N} \right)$ with respect to the $L^2$-norm given by \eqref{nt_l2_norm_eqn}. From the Lemma \ref{nt_l_2_est_lem} it follows $\Coo_0\left(\R^{2N}_\th \right)$ is dense subalgebra of $L^2 \left(\R^{2N}_\th \right)$ with respect to the operator norm.  Since $\varprojlim \mathfrak{S}_{C\left(\mathbb{T}^{2N}_\th\right)}$ is the operator norm completion of generated by special elements algebra, following condition holds
  		$$
  		C_0\left(\R^{2N}_\th \right) \subset \varprojlim \mathfrak{S}_{C\left(\mathbb{T}^{2N}_\th\right)}.
  		$$
  		From $C_0\left(\R^{2N}_\th \right) \subset \varprojlim \mathfrak{S}_{C\left(\mathbb{T}^{2N}_\th\right)}$ and $\varprojlim \mathfrak{S}_{C\left(\mathbb{T}^{2N}_\th\right)} \subset C_0\left(\R^{2N}_\th \right)$ it follows that
  		$$
  	\varprojlim	\mathfrak{S}_{C\left(\mathbb{T}^{2N}_\th\right)} = C_0\left(\R^{2N}_\th \right).
  		$$

 \end{proof}

 

\section{Acknowledgment}
\paragraph{} I would like to acknowledge the seminar "Noncommutative geometry and topology" of the Lomonosov Moscow State University. Vladimir Manuilov explained me some aspects of the theory of $W^*$-algebras. Theodore Popelensky gave a consultation about coverings and fundamental groups. I would like to thank Joseph Varilly for fruitful discussions via academia.edu.


\begin{thebibliography}{10}








\bibitem{arveson:c_alg_invt} W. Arveson. {\it An Invitation to $C^*$-Algebras}, Springer-Verlag. ISBN 0-387-90176-0, 1981.



\bibitem{blackadar:ko} B. Blackadar. {\it K-theory for Operator Algebras}. Second edition. Cambridge University Press 1998.


\bibitem{bogachev_measure_v2}V. I. Bogachev. {\it Measure Theory}. (volume 2). Springer-Verlag, Berlin, 2007.






\bibitem{chun-yen:separability} Chun-Yen Chou. {\it Notes on the Separability of $C^*$-Algebras.} TAIWANESE JOURNAL OF MATHEMATICS Vol. 16, No. 2, pp. 555-559, April 2012 This paper is available online at http://journal.taiwanmathsoc.org.tw, 2012.









\bibitem{elliot:an} Elliott H. Lieb, Michael Loss. \textit{Analysis}, American Mathematical Soc., 2001.

\bibitem{clare_crisp_higson:adj_hilb} Pierre Clare, Tyrone Crisp, Nigel Higson {\it Adjoint functors between categories of Hilbert modules}.  arXiv:1409.8656, 2014.

\bibitem{connes_landi:isospectral} Alain Connes, Giovanni Landi. \textit{ Noncommutative Manifolds the Instanton Algebra and Isospectral Deformations}. arXiv:math/0011194, 2001.




\bibitem{moyal_spectral} V. Gayral, J. M. Gracia-Bond\'{i}a, B. Iochum, T. Sch\"{u}cker, J. C. Varilly {\it Moyal Planes are Spectral Triples}. arXiv:hep-th/0307241, 2003.







\bibitem{varilly_bondia:phobos}Jos\'e M. Gracia-Bondia, Joseph C. Varilly.  \textit{Algebras of Distributions suitable for phase-space quantum mechanics. I}. Escuela de Matem\'{a}atica, Universidad de Costa Rica, San Jos\'e, Costa Rica J. Math. Phys 29 (1988), 869-879, 1988.





\bibitem{halmos:set} Paul R.  Halmos {\it Naive Set Theory.} D. Van Nostrand Company, Inc., Prineston, N.J., 1960.



 
 

\bibitem{hartshorne:ag} Robin Hartshorne. {\it Algebraic Geometry.} Graduate Texts in Mathematics, Volume 52, 1977.

\bibitem{hajac:toknotes}
{\it Lecture notes on noncommutative geometry and quantum groups}. Edited by Piotr M. Hajac.












\bibitem{milne:etale}J.S. Milne. {\it \'Etale cohomology.} Princeton Univ. Press  1980.

\bibitem{miracle:ph_d} S. Miracle-Sol\'e, Versi\'on algebraica del formalismo can\'onico de la f\'isica cu\'antica", Ph. D.thesis, Universidad de Barcelona, 1966.


\bibitem{demeyer:genreal_galois}DeMeyer, F. R. \textit{Some notes on the general Galois theory of rings}. Osaka J. Math.
2 (1965), 117-127, 1965.

\bibitem{mitchener:c_cat}Paul D. Mitchener. \textit{$C^*$-categories}. Odense University,  September 13, 2001.

\bibitem{miyashita_fin_outer_gal} Y\^oichi Miyashita, {\it Finite outer Galois theory of noncommutative rings}. Department of Mathematics, Hokkaido, University, 1966.



\bibitem{munkres:topology} James R. Munkres. {\it Topology.} Prentice Hall, Incorporated, 2000.

\bibitem{murphy} G.J. Murphy. {\it $C^*$-Algebras and Operator Theory.}Academic Press 1990.


\bibitem{pavlov_troisky:cov} Alexander Pavlov, Evgenij Troitsky. {\it Quantization of branched coverings.}   arXiv:1002.3491, 2010.

\bibitem{pedersen:ca_aut}Pedersen G.K. {\it $C^*$-algebras and their automorphism groups}. London ; New York : Academic Press, 1979.































\bibitem{spanier:at}
E.H. Spanier. {\it Algebraic Topology.} McGraw-Hill. New York 1966.

	\bibitem{takeda:inductive} Zir\^{o} Takeda. \textit{Inductive limit and infinite direct product of operator algebras.} Tohoku Math. J. (2) 	Volume 7, Number 1-2 (1955), 67-86. 1955.

\bibitem{takesaki:oa_ii} Takesaki, Masamichi. {\it Theory of Operator Algebras II } Encyclopaedia of Mathematical Sciences, 2003. 



















































\end{thebibliography}
\end{document}